\documentclass[a4paper,11pt,reqno]{amsart}

\usepackage{amsmath,amsthm,amssymb}
\usepackage{mathrsfs}
\usepackage{braket}
\usepackage{multicol}
\usepackage{pgfplots}
\usepackage{caption}
\usepackage{graphicx}
\usepackage{float}
\usepackage[all]{xy}
\usepackage{comment}
\usepackage{mathtools}
\usepackage{extarrows}

\numberwithin{equation}{section}
\allowdisplaybreaks

\allowdisplaybreaks

\newtheorem{definition}{Definition}[section]
\newtheorem*{definition*}{Definition}
\newtheorem{proposition}[definition]{Proposition}
\newtheorem{example}[definition]{Example}
\newtheorem{theorem}[definition]{Theorem}
\newtheorem*{theorem*}{Theorem}
\newtheorem{remark}[definition]{Remark}
\newtheorem{corollary}[definition]{Corollary}
\newtheorem{lemma}[definition]{Lemma}
\newtheorem{conjecture}[definition]{Conjecture}
\newtheorem{fact}[definition]{Fact}
\newtheorem{notation}[definition]{Notation}

\usepackage{tikz}
\usetikzlibrary{intersections, calc,shapes.geometric,patterns,arrows,decorations.markings, plotmarks} 	
\tikzset{
        MRarrow/.style={
        decoration={             
            markings, 
            mark=at position 0.75 with {%
  \makebox[2pt][l]{\arrow{triangle 45}};}
        },
        postaction={decorate}
    }
}

\newcommand{\vl}{\boldsymbol{\lambda}}
\newcommand{\vm}{\boldsymbol{\mu}}
\newcommand{\vn}{\boldsymbol{\nu}}
\newcommand{\vx}{\boldsymbol{x}}
\newcommand{\vy}{\boldsymbol{y}}
\newcommand{\vs}{\boldsymbol{s}}
\newcommand{\lo}{\lambda^{(1)}}
\newcommand{\lt}{\lambda^{(2)}}
\newcommand{\lN}{\lambda^{(N)}}
\newcommand{\mo}{\mu^{(1)}}
\newcommand{\mt}{\mu^{(2)}}
\newcommand{\mN}{\mu^{(N)}}
\newcommand{\vu}{\boldsymbol{u}}
\newcommand{\vv}{\boldsymbol{v}}
\newcommand{\vw}{\boldsymbol{w}}
\newcommand{\Xo}{X^{(1)}}

\newcommand{\vnn}{\boldsymbol{n}}
\newcommand{\vmm}{\boldsymbol{m}}
\newcommand{\ketzero}{\ket{\boldsymbol{0}}}
\newcommand{\brazero}{\bra{\boldsymbol{0}}}

\newcommand{\cF}{\mathcal{F}}
\newcommand{\cU}{\mathcal{U}}
\newcommand{\cR}{\mathcal{R}}
\newcommand{\cV}{\mathcal{V}}
\newcommand{\nub}{\bar{\nu}}
\newcommand{\sfa}{\mathsf{a}}

\newcommand{\wPhi}{\widehat{\Phi}}

\newcommand{\cFhol}{\mathcal{F}^{(1,0)}}

\newcommand{\hg}{c}

\newcommand{\cTH}{\mathcal{T}^H}
\newcommand{\cTV}{\mathcal{T}^V}
\newcommand{\lbl}[3]{[#1, #2]_{\boldsymbol{#3}}}
\newcommand{\overstar}[1]{\mathop{\overset{*}{ #1 }}}
\newcommand{\nik}{[i,k]_{\vnn}}
\newcommand{\ik}{[i,k]}
\newcommand{\ts}[1]{q^{-\theta_{#1}}s_{#1}}
\newcommand{\pp}{\mathsf{p}} 
\newcommand{\WA}{\mathsf{U}(N)} 
\newcommand{\aotimes}[2]{\ooalign{$\textstyle \bigotimes$\crcr \hss \raisebox{1.85ex}{\scriptsize$\curvearrowright$}\hss}\rule[-0.6ex]{0ex}{2.4ex}_{#1}^{#2}}

\newcommand{\Mphi}[4]{
	\phi^{#1}_{#2}\Big(
	\arraycolsep=1pt
	\renewcommand{\arraystretch}{0.8}
	\begin{array}{cccccccc}#3\end{array}\Big|
	\begin{array}{cccccccc}#4\end{array}
	\renewcommand{\arraystretch}{1.0}
	\Big)
}
\newcommand{\Mphiu}[4]{
\phi^{#1}\Big(
\arraycolsep=1pt
\renewcommand{\arraystretch}{0.8}
\begin{array}{cccccccc}#2\end{array}\Big|
\begin{array}{cccccccc}#3\end{array}; {#4}
\renewcommand{\arraystretch}{1.0}
\Big)
}
\newcommand{\Vweight}[3]{
	V^{#1}\Big(
	\arraycolsep=1pt
	\renewcommand{\arraystretch}{0.8}
	\begin{array}{c}#2\end{array}\, ; #3
	\renewcommand{\arraystretch}{1.0}
	\Big)
}
\newcommand{\tVweight}[3]{
	\widetilde{V}^{#1}\Big(
	\arraycolsep=1pt
	\renewcommand{\arraystretch}{0.8}
	\begin{array}{c}#2\end{array}\, ; #3
	\renewcommand{\arraystretch}{1.0}
	\Big)
}
\newcommand{\cVweight}[2]{
	\cV\Big(
	\arraycolsep=1pt
	\renewcommand{\arraystretch}{0.8}
	\begin{array}{c}#1\end{array}\, ; #2
	\renewcommand{\arraystretch}{1.0}
	\Big)
}

\title[Generalized Macdonald functions and Changing Preferred Direction]{Generalized Macdonald Functions on Fock Tensor Spaces and Duality Formula for Changing Preferred Direction}
\author{Masayuki~Fukuda, Yusuke~Ohkubo and  Jun'ichi~Shiraishi}
\address{MF: Department of Physics, Faculty of Science, The University of Tokyo, Bunkyo-ku, Tokyo 113-0033, Japan}
\email{fukuda@hep-th.phys.s.u-tokyo.ac.jp}
\address{YO, JS: Graduate School of Mathematical Sciences, University of Tokyo, Komaba, Tokyo 153-8914, Japan}
\email{yusuke.ohkubo.math@gmail.com}
\email{shiraish@ms.u-tokyo.ac.jp}

\begin{document}
\maketitle

\begin{abstract}
An explicit formula is obtained for the generalized Macdonald functions on the 
$N$-fold Fock tensor spaces,
calculating a certain matrix element of a composition of several 
screened vertex operators. 
As an application, 
we prove the factorization property of 
the arbitrary matrix elements of the multi-valent intertwining operator
(or refined topological vertex operator)
associated with the Ding--Iohara--Miki algebra (DIM algebra) with respect to 
the generalized Macdonald functions, 
which was conjectured by Awata, Feigin, Hoshino, Kanai, Yanagida and one of the authors. 
Our proof is based on the combinatorial and analytic properties of the 
asymptotic eigenfunctions of the ordinary Macdonald operator of $A$-type, and the 
Euler transformation formula for Kajihara and Noumi's multiple basic hypergeometric series. 
That factorization formula provides us with a reasonable algebraic 
description of the 5D (K-theoretic) Alday-Gaiotto-Tachikawa (AGT) correspondence, and
the interpretation of the invariance under the preferred direction   
from the point of view of the $SL(2,\mathbb{Z})$ duality of the DIM algebra. 
\end{abstract}

\section{Introduction}\label{sec: Intro}
Let $\cU$ be the DIM algebra \cite{DI, Miki}. 
As for the definition of the DIM algebra, 
see Definition \ref{def: DIM}. 
The central object in the present paper is 
the intertwining operator $\cV (x)$ associated with some structure of the $\cU$ modules. 
From the point of view of the geometric engineering, or 
topological vertex construction for the partition functions for the quantum supersymmetric
gauge theories, we regard the $\cV(x)$ as the {\it (multi-valent) topological vertex operator}.
The intertwining operator $\cV(z)$ is defined 
through a certain set of commutation relations with the $\cU$-generators. Let
$X^{(i)}(z)$'s ($i=1,\ldots,N$) be the generating currents constructed from the 
standard Drinfeld current of $\cU$ 
(Definition \ref{def: X}), acting on the $N$-fold tensor space 
$\cF_{\boldsymbol{u}}= \aotimes{j=1}{N} \cF_{u_j}$  of the Fock spaces 
(For the definition of $\aotimes{}{}$, see Notation \ref{not: aotimes}).

\begin{definition}[topological vertex]
Let  $\cV(x) : \cF_{\boldsymbol{u}} \to \cF_{\boldsymbol{v}}$ be a linear map 
satisfying the commutation relations
\begin{equation}
\left(1-\frac{x}{z}\right)X^{(i)}(z) \cV(x) 
=\left(1- (t/q)^i\frac{x}{z}\right)\cV(x) X^{(i)}(z) 
\qquad (i=1,\ldots,N),
\end{equation}
and the normalization condition $\bra{\boldsymbol{0}}\cV(x) \ket{\boldsymbol{0}}=1$. 
Here $\ket{\boldsymbol{0}}$ (resp. $\bra{\boldsymbol{0}}$)  is the vacuum (resp. dual vacuum) state.
\end{definition}

We refer to this operator as the Mukad\'{e} operator. 
Mukad\'{e} is a Japanese word which means a centipede. 
The operator $\cV(x)$ can be realized by connecting the trivalent intertwining operators 
as Figure \ref{fig_TH} in Section \ref{Sec_Sdual}, 
and that figure is the reason why we call $\cV(x)$ the Mukad\'{e} operator. 
Let $\boldsymbol{\lambda}=(\lambda^{(1)},\ldots,\lambda^{(N)})$ be 
an $N$-tuple of partitions. 
Denote by $\ket{K_{\boldsymbol{\lambda}}}=\ket{K_{\boldsymbol{\lambda}}(\boldsymbol{u})}$ 
the canonical form of the generalized 
Macdonald function on $\cF_{\boldsymbol{u}}$ (Definition \ref{def: K}).  
The main result of this paper is to prove the following 
factorization formula conjectured in \cite{AFHKSY1}. 

\begin{theorem}\label{main-theorem}
We have 
\begin{align}\label{eq: main thm}
\bra{K_{\boldsymbol{\lambda}}(\vv)}\cV (x)\ket{K_{\boldsymbol{\mu}}(\vu)}
=&
\frac{\left((-\gamma^2)^N e_{N}(\vu)x \right)^{|\vl|}}
{\left( \gamma^2 x \right)^{|\vm|}}
\prod_{i=1}^N
\frac{u_i^{|\mu^{(i)}|}g_{\mu^{(i)}}} 
{\left(v_i^{|\lambda^{(i)}|}g_{\lambda^{(i)}}\right)^{N-1}} \cdot
\prod_{i,j=1}^N N_{\lambda^{(i)},\mu^{(j)}}(qv_i/tu_j). 
\end{align}
Here, $\gamma=(t/q)^{1/2}$, $e_{N}(\vu)=u_1\cdots u_N$ and $N_{\lambda,\mu}$ is the Nekrasov factor. 
For the definition of $g_{\lambda}$, see Definition \ref{def: flaming factor}. 
\end{theorem}

In our proof, 
the first important thing is 
to construct the explicit formula for the generalized Macdonald functions 
$\ket{K_{\vl}}$. 
We construct this explicit formula by using 
some screened vertex operators (Theorem \ref{thm: GM}). 
Next, the operator $\cV(x)$ can also be realized by 
a certain fusion of the screened vertex operators. 
Matrix elements of the composition of the screened vertex operators satisfy 
a difference equation and 
correspond to a multiple series $p_n$ 
(Macdonald functions of A-type, Definition \ref{def: p_n f_n}). 
The analyticity of matrix elements of screened vertex operators 
can be clarified by the analyticity of the Macdonald functions $p_n$. 
The final step of our proof is attributed to a transformation problem 
with respect to this multiple series $p_n$, 
and this is proved by the Euler transformation formula for 
Kajihara and Noumi's multiple basic hypergeometric series. 

We remark that the case $N=1$ of 
Theorem \ref{main-theorem} was proved by using Kajihara-Noumi's
Euler transformation formula and the Pieri rules for the Macdonald 
polynomials \cite{AFHKSY1}. As for the detail of the proof, 
unfortunately, 
it remains unpublished yet \cite{AFHKSY2}.
Since the Pieri rules are known only for the standard Macdonald case ($N=1$), 
it is totally unclear how we  extend that method 
to the case $N>1$, which drove the present authors to have the 
formulation presented in this paper.

This paper is organized as follows.
In Section \ref{sec_Prelim}, we briefly remind the readers of the basic facts about the DIM algebra $\cU$
and its module structure on the Fock space $\cF$. 
In Section \ref{Sec_GenMac}, 
we introduce the generalized Macdonald functions, 
and they are explicitly constructed by the screened vertex operators (Therorem \ref{thm: GM}). 
We also discuss the analyticity of the matrix elements of the composition of 
screened vertex operators (Theorem \ref{thm_Analytic V}). 
In Section \ref{Sec_matrix element}, 
we give a proof of the main theorem (\ref{eq: main thm}). 
In Section \ref{Sec_Sdual}, the existence of $\cV(x)$ 
is proved through the explicit construction in terms of the topological vertex (Proposition \ref{prop: existance_V}). 
As its application, we can show the invariance under changing the preferred direction 
of toric diagrams, which is the natural consequence of the dualities in the string theory (Theorem \ref{prop: changing PD}).

In Appendix \ref{App: Mac from top vertex}, 
the Macdonald functions $p_n$ are constructed in terms of intertwining operators. 
In Appendix \ref{App: ord. Macdonald}, we give a brief review of the definition of 
ordinary Macdonald functions and their basic facts. 
In Appendix \ref{App: Useful Formula}, some useful formulas used in paper are explained. 
Some proofs in the main text are given in Appendix \ref{App: proofs}. 
In Appendix \ref{App: reargof SM}, 
we revisit the proof of a certain formula for the Kac determinant 
with respect to $\ket{X_{\vl}}$ (Definition \ref{def: PBW-type basis}) given in \cite{O}. 
In particular, we clarify the choice of integral cycles of screening operators. 
In Appendix \ref{App: example}, we give some examples. 
At last, we present list of notations in Appendix \ref{App: list of notation}.  

\subsection*{Notations}

A partition $\lambda=(\lambda_1,\lambda_2,\ldots)$ is a sequence of non-negative integers 
satisfying $\lambda_1 \geq \lambda_2\geq \cdots$ and containing finitely many nonzero elements.  
They are identified if their elements without $0$ are the same. 
$\ell(\lambda)$ denotes the length of $\lambda$, i.e., 
the number of elements in $\lambda$ without $0$. 
For an integer $i>\ell(\lambda)$, 
we occasionally write $\lambda_i$. 
Note that it is just $0$. 
The partition in which all elements are $0$ is denoted by $\emptyset$. 
Write $|\lambda|=\sum_{i\geq 1} \lambda_i$ and $n(\lambda)=\sum_{i\geq 1}(i-1)\lambda_i$. 
Partitions are identified with the Young diagrams. 
For example, if $\lambda=(4,4,2,1)$, 
the Young diagram of $\lambda$ is 
\setlength{\unitlength}{1mm}
\newsavebox{\young}
\savebox{\young}{
\begin{picture}(20,20)(0,0)
\put(0,20){\line(1,0){20}}
\put(0,15){\line(1,0){20}}
\put(0,10){\line(1,0){20}}
\put(0,5){\line(1,0){10}}
\put(0,0){\line(1,0){5}}
\put(20,20){\line(0,-1){10}}
\put(15,20){\line(0,-1){10}}
\put(10,20){\line(0,-1){15}}
\put(5,20){\line(0,-1){20}}
\put(0,20){\line(0,-1){20}} 
\put(11.5,11.5){$s$}
\end{picture}}
\begin{center}
\usebox{\young}.
\end{center}
The transpose of $\lambda$ is denoted by $\lambda'$. 
The coordinate of the box in the $i$-th row and the $j$-th column is denoted by $(i,j)$, 
say the coordinate of the box $s$ in the above Young diagram is $(2,3)$. 
For a coordinate $(i,j)\in \mathbb{Z}^2$, 
the arm length and the leg length are defined 
by $a_{\lambda}(i,j)= \lambda_i-j$ and $\ell_{\lambda}(i,j)=\lambda'_j-i$, 
respectively.  
$A(\lambda)$ (resp. $R(\lambda)$) is the set of coordinates 
where we can add (resp. remove) a box. 
For example, if $\lambda=(3,3,1)$, 
$A(\lambda)=\{(1,4),(3,2),(4,1) \}$ and $R(\lambda)=\{ (2,3),(3,1) \}$. 
In this paper, we often use $N$-tuples of partitions $\vl=(\lo,\ldots , \lN)$. 
For an $N$-tuple of partitions $\vl$, 
we put $|\vl|=\sum_{k=1}^N |\lambda^{(k)}|$. 

Further, we use the following factorials ($q$-Pochhammer symbol) and the theta function: \begin{align}
&(a;q)_{\infty}:=\prod_{n=1}^{\infty}(1-q^{n-1}a),\quad 
(a;q)_{m}:=
\left\{ 
\begin{array}{c l}
\displaystyle \prod_{n=1}^{m}(1-q^{n-1}a) & (m\geq 1);\\
1 & (m=0);\\
\displaystyle \prod_{n=1}^{-m}(1-q^{-n}a)^{-1} &  (m \leq -1), 
\end{array} 
\right. \\
&\theta_q(a):=(a;q)_{\infty}(qa^{-1};q)_{\infty}.
\end{align}

\section{Preliminaries}\label{sec_Prelim}
\subsection{Ding--Iohara--Miki algebra $\mathcal{U}$}
Recall briefly some basic facts about the DIM algebra $\mathcal{U}$ \cite{DI, Miki}. 
Let $q$ and $t$ be nonzero complex parameters 
satisfying $|q|<1$ and 
$q^{\frac{n}{2}}t^{\frac{m}{2}}\neq 1$ for all integers $n, m$.

\begin{definition}\label{def: DIM}
Let $\mathcal{U}=\mathcal{U}_{q,t}$ be the unital associative algebra over $\mathbb{C}$  generated by the Drinfeld currents 
\begin{equation*}
    x^\pm(z) = \sum_{n\in \mathbb{Z}} x_n^\pm z^{-n}\,,\quad \psi^\pm(z) = \sum_{\pm n\in \mathbb{Z}_{\geq 0}} \psi_n^\pm z^{-n}\,,
\end{equation*}
and the invertible central element $\hg^{1/2}$, satisfying the following defining relations: 
\begin{equation}\label{DIMdef}
 \begin{split}
&\psi^+(z)x^\pm(w)=g(c^{\mp 1/2}w/z)^{\mp1} x^\pm(w)\psi^+(z),\qquad \psi^-(z)x^\pm(w)=g(c^{\mp 1/2}z/w)^{\pm1} x^\pm(w)\psi^-(z),\\
&\psi^\pm(z) \psi^\pm(w)= \psi^\pm(w) \psi^\pm(z),\qquad 
\psi^+(z)\psi^-(w)=
\dfrac{g(c^{+1} w/z)}{g(c^{-1}w/z)}\psi^-(w)\psi^+(z),\\
&[x^+(z),x^-(w)]=\dfrac{(1-q)(1-1/t)}{1-q/t}
\bigg( \delta(c^{-1}z/w)\psi^+(c^{1/2}w)-
\delta(c z/w)\psi^-(c^{-1/2}w) \bigg),\\
&G^{\mp}(z/w)x^\pm(z)x^\pm(w)=G^{\pm}(z/w)x^\pm(w)x^\pm(z)\,,
\end{split}
\end{equation}
where
\begin{equation}
g(z)=\dfrac{G^+(z)}{G^-(z)}\,,\qquad
G^\pm(z)=(1-q^{\pm1}z)(1-t^{\mp 1}z)(1-q^{\mp1}t^{\pm 1}z)\,,\qquad\delta(z) = \sum_{n\in\mathbb{Z}}z^n\,.
\end{equation}
\end{definition}

\begin{fact}
    The $\mathcal{U}$ admits the (topological) Hopf algebra structure
    with the Drindeld coprocuct $\Delta$:
    \begin{equation}\label{copro}
        \begin{split}
&\Delta(\hg^{\pm 1/2})=\hg^{\pm 1/2} \otimes \hg^{\pm 1/2}\,,\\
&\Delta (x^+(z))=
x^+(z)\otimes 1+
\psi^-(c_{(1)}^{1/2}z)\otimes x^+(c_{(1)}z)\,,\\
&\Delta (x^-(z))=
x^-(c_{(2)}z)\otimes \psi^+(c_{(2)}^{1/2}z)+1 \otimes x^-(z)\,,\\
&\Delta (\psi^\pm(z))=
\psi^\pm (c_{(2)}^{\pm 1/2}z)\otimes \psi^\pm (c_{(1)}^{\mp 1/2}z)\,,
        \end{split}
    \end{equation}
    where $\hg_{(1)}^{\pm 1/2}=\hg^{\pm 1/2}\otimes 1$
and $\hg_{(2)}^{\pm 1/2}=1\otimes \hg^{\pm 1/2}$.
We omit the antipode and the counit.
\end{fact}

\subsection{Fock and Fock tensor modules}\label{subsec_F1M}

A $\cU$-module is called of level-$(n,m)$, if the two central elements act as
$c=(t/q)^{n/2}$ and $(\psi^+_0/\psi^-_0)^{1/2}=(q/t)^{m/2}$.
Let $M\in \mathbb{Z}$. 
The Fock module $\mathcal F$ of level-$(1,M)$ is constructed as follows.
Let $\{a_n | n\in\mathbb{Z}\}$ be the Heisenberg algebra with the relation 
\begin{equation}\label{def_boson}
     [a_m,a_n]=m\dfrac{1-q^{|m|}}{1-t^{|m|}}\delta_{m+n,0} \, a_0\,,
\end{equation}
and act on the Fock space $\cF$ with the vacuum $\ket{0}$ in the usual way (that is, $a_n \ket{0} = 0$ for $n\in\mathbb{Z}_{>0}$).
$\mathcal{F}^*$ is the dual space generated by $\bra{0}$ satisfying 
$\bra{0}a_n=0$ ($n<0$). 
The basis of $\cF$ (resp. $\cF^*$) is given by $\ket{a_\lambda} = a_{-\lambda_1}a_{-\lambda_2}\cdots \ket{0}$ 
(resp. $\bra{a_\lambda} = \bra{0}\cdots a_{\lambda_2}a_{\lambda_1}$) 
with a partition $\lambda = (\lambda_1, \lambda_2, \dots)$.
The bilinear form $\cF^* \otimes \cF \rightarrow \mathbb{C}$ 
is given by $\braket{0|0}=1$.

\begin{fact}[\cite{FHHSY}]\label{Fact_Hor_rep}
Let $u$ be a nonzero complex parameter. 
The following algebra homomorphism $\rho_u: \cU \to \mathrm{End}(\cF) $ endows the $\cU$-module structure on $\cF$: 
\begin{equation}
\begin{split}
&\hg^{1/2}\mapsto (t/q)^{1/4},\quad
x^+(z)\mapsto
u z^{-M} q^{-M/2}t^{M/2}\eta(z),\quad x^-(z)\mapsto
u^{-1} z^{M}q^{M/2}t^{-M/2} \xi(z) ,\\
&\psi^+(z)\mapsto
q^{M/2}t^{-M/2}
\varphi^+(z),\quad
\psi^-(z)\mapsto
q^{-M/2}t^{M/2}\varphi^-(z)\,,
\end{split}
\end{equation}
where 
\begin{equation}
    \begin{split}
&\eta(z)=
\exp\Big( \sum_{n=1}^{\infty} \dfrac{1-t^{-n}}{n}a_{-n} z^{n} \Big)
\exp\Big(-\sum_{n=1}^{\infty} \dfrac{1-t^{n} }{n}a_n    z^{-n}\Big),\\
&\xi(z)=
\exp\Big(-\sum_{n=1}^{\infty} \dfrac{1-t^{-n}}{n}
q^{-n/2}t^{n/2}
a_{-n} z^{n}\Big)
\exp\Big( \sum_{n=1}^{\infty} \dfrac{1-t^{n}}{n} q^{-n/2}t^{n/2} a_n z^{-n}\Big),\\
&\varphi^{+}(z)=
\exp\Big(
 -\sum_{n=1}^{\infty} \dfrac{1-t^{n}}{n} (1-t^n q^{-n})
 q^{n/4}t^{-n/4} a_n z^{-n}
    \Big),\\
&\varphi^{-}(z)=
\exp\Big(
 \sum_{n=1}^{\infty} \dfrac{1-t^{-n}}{n} (1-t^n q^{-n}) q^{n/4}t^{-n/4} a_{-n}z^{n}
    \Big).
    \end{split}
\end{equation}
We call $u$ the spectral parameter, and denote this $\cU$-module by $\cF^{(1,M)}_u$.

We can also define the dual $\cU$-module structure $\cF^{(1,M)*}_u$ on $\cF^*$ through the same $\rho_u$ by regarding its image as in $\mathrm{End}(\cF^*)$.
\end{fact}

Let $N\in \mathbb{Z}_{\geq 1}$, and let $\vu=(u_1,\ldots, u_N)$
be an $N$-tuple of nonzero complex parameters. 
Set 
\begin{align} 
&\rho_{\vu}^{(N,0)}:=
(\rho_{u_1}\otimes\rho_{u_2}\otimes\cdots \otimes\rho_{u_N}) 
\circ\Delta^{(N)}, \\
&\mbox{where $\Delta^{(1)}:=\mathrm{id}$, 
$\Delta^{(N)}:=
(\underbrace{\mathrm{id} \otimes \cdots \otimes{\rm id}}_{N-2} \otimes \Delta) 
\circ \cdots \circ 
(\mathrm{id} \otimes \Delta) \circ \Delta$ ($N\geq 2$)}. \nonumber
\end{align} 
The $\rho_{\vu}^{(N,0)}$ gives the level-$(N,0)$ module structure  
on the $N$-fold tensor space of $\mathcal F$. 
We denote it by $\cFhol_{u_1}\otimes \cdots \otimes \cFhol_{u_N}$.
In what follows, 
we will use the shorthand notation 
$\cF_{\boldsymbol{u}}=\cFhol_{u_1}\otimes \cdots \otimes \cFhol_{u_N}$ for simplicity.
We also denote $\cF^*_{\boldsymbol{u}}=\cF^{(1,0)*}_{u_1}\otimes \cdots \otimes \cF^{(1,0)*}_{u_N}$.
Let $\ketzero$ (resp. $\brazero$) be the tensor product of the vacuum 
(resp. dual vacuum) states 
$\ket{0}^{\otimes N}$ (resp. $\bra{0}^{\otimes N}$) in $\cF_{\vu}$ (resp. $\cF^*_{\vu}$). 
Set 
\begin{align}
a^{(i)}_n=\overbrace{1\otimes \cdots \otimes 1}^{i-1}\otimes
a_n\otimes\overbrace{ 1 \otimes \dots \otimes 1}^{N-i}, 
\end{align}
for simplicity.

\begin{definition}\label{Def_X^k}
For $k = 1,2,\dots,N$, 
set
\begin{align}
&\Lambda^{(i)}(z) := \varphi^-(\gamma^{1/2}z)\otimes \cdots \otimes\varphi^-(\gamma^{i-3/2}z)\otimes\overbrace{\eta(\gamma^{i-1} z)}^{i\text{-th Fock space}}\otimes 1\otimes\cdots\otimes1 . 
\end{align}
\end{definition}

\begin{fact}[\cite{FHSSY}]
On $\cF_{\vu}$, we have
\begin{align}
\rho_{\vu}^{(N,0)}( x^+(z))=\sum_{i=1}^N u_i \Lambda^{(i)}(z).
\end{align}
\end{fact}

\begin{definition}\label{def: X}
Set $X^{(1)}(z)=\rho_{\vu}^{(N,0)}( x^+(z))$. Introduce
the set of generators $X^{(k)}_n$ ($k=1,\ldots,N,n\in \mathbb{Z}$) by performing the 
fusion of several $X^{(1)}$'s as
\begin{equation}
X^{(k)}(z) =
\sum_{n\in\mathbb{Z}} X^{(k)}_n z^{-n} = \Xo(\gamma^{2(1-k)} z)\Xo(\gamma^{2(2-k)} z)\cdots\Xo(z)
\in \mathrm{End}(\mathcal{F}_{\vu})[[z^{\pm 1}]]. 
\end{equation}
\end{definition}

\begin{fact}[\cite{AFHKSY1}]
We have
\begin{align}
&X^{(k)}(z) 
= \sum_{1\leq j_1 <\cdots <j_k \leq N} : \Lambda^{(j_1)}(z) \cdots \Lambda^{(j_k)}((q/t)^{k-1}z): u_{j_1} \cdots u_{j_k}.
\end{align}
Here, $:\ast:$ denotes the usual normal ordering in the Heisenberg algebra. 
\end{fact}

\begin{definition}\label{Def_WA}
We denote by $\WA$ (the completion in the sence of the adic topology, of)
the algebra 
$\langle X^{(i)}_n|n\in \mathbb{Z}, i=1,\ldots N  \rangle$ in $\mathrm{End}(\mathcal{F}_{\vu})$.
Namely, $\WA$ is the completion of the algebra generated by the 
set of operators 
$\left\{ X^{(i)}_n\right\}$.
\end{definition}

This algebra $\WA$ can be regarded as the tensor product of 
the deformed $W_{N}$ algebra and some Heisenberg algebra \cite{FHSSY}. 

\begin{definition}\label{def: PBW-type basis}
For an $N$-tuple of partitions 
$\vl=(\lo, \lt, \ldots, \lambda^{(N)})$, 
we define 
the vectors $\Ket{X_{\vl}}=\Ket{X_{\vl}(\vu)} \in \mathcal{F}_{\vu}$ and  
$\Bra{X_{\vl}}=\Bra{X_{\vl}(\vu)}\in \mathcal{F}_{\vu}^*$ by 
\begin{align}
&\Ket{X_{\vl}(\vu)} := 
  X^{(1)}_{-\lambda^{(1)}_1} X^{(1)}_{-\lambda^{(1)}_2} \cdots  X^{(2)}_{-\lambda^{(2)}_1} X^{(2)}_{-\lambda^{(2)}_2}\cdots X^{(N)}_{-\lambda^{(N)}_1} X^{(N)}_{-\lambda^{(N)}_2} \cdots \ketzero, \\
& \Bra{X_{\vl}(\vu)} := 
\brazero \cdots X^{(N)}_{\lambda^{(N)}_2} X^{(N)}_{\lambda^{(N)}_1} \cdots X^{(2)}_{\lambda^{(2)}_2} X^{(2)}_{\lambda^{(2)}_1} \cdots 
X^{(1)}_{\lambda^{(1)}_2} X^{(1)}_{\lambda^{(1)}_1} . 
\end{align}
\end{definition}
In what follows, we omit the spectral parameter $\vu$, as long as there is no confusion. 

\begin{fact}[\cite{O}]\label{fact: PBW basis}
The set $(\Ket{X_{\vl}})$ (resp. $(\Bra{X_{\vl}})$) forms a PBW-type basis 
of $\mathcal{F}_{\vu}$ (resp. $\mathcal{F}^*_{\vu}$), if 
$u_i\neq q^st^{-r}u_j$ and $u_i\neq 0$ 
for all $i, j$ and $r, s \in \mathbb{Z}$. 
\end{fact}

In Appendix \ref{App: reargof SM}, 
we reproduce the proof (presented in \cite{O}) of Fact \ref{fact: PBW basis}, 
with detailed treatments concerning the integration cycles.

\section{Generalized Macdonald Functions}\label{Sec_GenMac}

In  \cite{AFHKSY1}, 
a sort of generalization of Macdonald functions are introduced 
on the level $(N,0)$ module $\mathcal{F}_{\vu}$. 
In this section, 
we give a certain explicit formula for the generalized Macdonald functions.

\subsection{Definition of generalized Macdonald functions}\label{SSec_GMDef}

The generalized Macdonald functions are defined as eigenfunctions of $\Xo_{0}$. 
In the theory of ordinary Macdonald functions, 
the fundamental existence theorem follows from the triangulation of Macdonald's difference operator. 
In that triangulation, 
the order of a basis is given by the dominance partial ordering. 
In the $N=1$ case, $\Xo_{0}$ corresponds to Macdonald's difference operator 
\cite{AMOS}. 
In the general $N$ case, 
the operator $\Xo_{0}$ can be triangulated by the following partial ordering.

\begin{definition}\label{def:ordering1}
The partial ordering $\overstar{>}$ on the set of $N$-tuples of partitions 
is defined by 
\begin{align}
\vl \overstar{>} \vm \quad \overset{\mathrm{def}}{\Longleftrightarrow} \quad  
& |\vl| = |\vm|, \quad 
\sum_{i=k}^N |\lambda^{(i)}| \geq \sum_{i=k}^N |\mu^{(i)}| \quad (\forall k ) \quad \mathrm{and}   \\
&  (|\lo|,|\lt|,\ldots ,|\lN|) \neq (|\mo|,|\mt|,\ldots ,|\mN|).  \nonumber
\end{align}
\end{definition}

\begin{remark}
This partial ordering is defined so that 
as we move boxes in a right Young diagram to a left one, 
it gets smaller. 
On the other hand, 
as boxes are moved from left to right, 
it gets larger. 
Note that if $N$-tuples of partitions $\vl$ and $\vm$ 
have the same number of boxes in each Young diagram, 
then neither $\vl\overstar{>} \vm$ nor $\vl\overstar{<} \vm$. 
\end{remark}

\begin{example}
If $N=3$ and the number of boxes is $2$, 
then 
\[\xymatrix@!=0.5em{
  & (\emptyset, \emptyset, (2))  \ar[rd] & & (\emptyset, \emptyset, (1,1)) \ar[ld]&  \\
  & & (\emptyset, (1), (1))\ar[lld]\ar[d]\ar[rrd] & &  \\
((1), \emptyset, (1))\ar[rrd] & & (\emptyset, (2),\emptyset)\ar[d] & & ( \emptyset, (1,1),\emptyset )\ar[lld] \\
  & & ((1), (1), \emptyset)\ar[ld]\ar[rd] & & \\
  & ((2), \emptyset, \emptyset) & & ((1,1), \emptyset, \emptyset). &
}\]
Here $\vl \rightarrow \vm$ stands for $\vl \overstar{>} \vm$. 
\end{example}

Regarding the products of ordinary Macdonald functions $\prod_{i=1}^N P_{\lambda^{(i)}}(a^{(i)}_{-n})\ketzero$ 
as a basis of $\mathcal{F}_{\vu}$, 
we can triangulate the operator $\Xo_0$. 
Here, $P_{\lambda}(a_{-n})$ is the abbreviation for $P_{\lambda}(a_{-1},a_{-2}\ldots)$, 
which is the Macdonald function obtained by replacing the power sum symmetric functions 
with generators of the Heisenberg algebra. 
For the definition and notations of the ordinary Macdonald functions 
see Appendix \ref{App: ord. Macdonald}. 
In Appendix \ref{App: ord. Macdonald}, we also explain some well-known facts for Macdonald functions. 

\begin{fact}[\cite{AFO}]
We have
\begin{equation}
X_{0}^{(1)} \prod_{j=1}^{N} P_{\lambda^{(j)}}(a_{-n}^{(j)}) \ketzero
= \epsilon_{\vl} \prod_{j=1}^{N} P_{\lambda^{(j)}}(a_{-n}^{(j)}) \ketzero
 + \sum_{\vm \overstar{<} \vl} v_{\vl, \vm} \prod_{j=1}^{N} P_{\mu^{(j)}}(a_{-n}^{(j)}) \ketzero, 
\end{equation}
\begin{equation}
\brazero \prod_{j=1}^{N} P_{\lambda^{(j)}}(a_{n}^{(j)}) X_{0}^{(1)} 
= \epsilon^*_{\vl}\brazero \prod_{j=1}^{N}  P_{\lambda^{(j)}}(a_{n}^{(j)}) 
 + \sum_{\vm \overstar{>} \vl} v^*_{\vl, \vm} \brazero \prod_{j=1}^{N} P_{\mu^{(j)}}(a_{n}^{(j)}). 
\end{equation}
Here, $v_{\vl, \vm}, v^*_{\vl, \vm} \in \mathbb{C}$. 
The eigenvalues $\epsilon_{\vl}$ and $\epsilon^*_{\vl}$ are in the following form: 
\begin{equation}
\epsilon_{\vl}(\vu)= 
\epsilon_{\vl}^*(\vu)= 
\sum_{k=1}^N u_k e_{\lambda^{(k)}}, \quad 
e_{\lambda}:= 1+(t-1) \sum_{i \geq 1} (q^{\lambda_i}-1)t^{-i}. 
\end{equation}
\end{fact}

\begin{remark}
In this paper, 
we assume that the complex parameters $q$, $t$ and 
the spectra parameters $\vu$ are generic 
in the following sense: 
\begin{align}
&\epsilon_{\vl}(\vu)\neq 0 \quad (\forall \vl);\\
&\epsilon_{\vl}(\vu) \neq \epsilon_{\vm}(\vu) \quad (\vl \neq \vm);\\
&u_i \neq q^{n}t^{m} u_j \quad ( n,m \in \mathbb{Z},\, i,j=1,\ldots, N).
\end{align}
\end{remark}

Under this assumption, 
the eigenfunctions of $\Xo_0$ can be characterized as follows.

\begin{fact}[Existence and Uniqueness \cite{AFO}]
\label{fact:existence thm of Gn Mac}
For an $N$-tuple of partitions $\vl$, 
there exists a unique vector $\Ket{P_{\vl}} =\Ket{P_{\vl}(\vu)} \in \mathcal{F}_{\vu}$ 
such that
\begin{align}
 &\bullet \quad \Ket{P_{\boldsymbol{\lambda}}(\vu)} 
  = \prod_{i=1}^N P_{\lambda^{(i)}}(a^{(i)}_{-n}) \ketzero 
  + \sum_{\vm \mathop{\overset{*}{<}} \vl} u_{\vl, \vm} \prod_{i=1}^N P_{\mu^{(i)}}(a^{(i)}_{-n}) \ketzero,
\quad u_{\vl, \vm}\in \mathbb{C}; 
\\
 &\bullet \quad X^{(1)}_0 \Ket{P_{\vl}(\vu)} = \epsilon_{\vl}(\vu) \Ket{P_{\vl}(\vu)}, \quad 
\epsilon_{\vl}(\vu) \in \mathbb{C}.
\end{align}
Similarly, there exists a unique vector 
$\Bra{P_{\vl}}=\Bra{P_{\vl}(\vu) }\in \mathcal{F}_{\vu}^*$ such that 
\begin{align}
 &\bullet \quad\Bra{P_{\boldsymbol{\lambda}}(\vu)} 
  = \brazero \prod_{i=1}^N P_{\lambda^{(i)}}(a^{(i)}_{n}) 
  + \sum_{\boldsymbol{\mu} \overstar{>} \boldsymbol{\lambda}} u_{\vl, \vm}^* \brazero \prod_{i=1}^N P_{\mu^{(i)}}(a^{(i)}_{n}), 
\quad u^*_{\vl, \vm}\in \mathbb{C};
\\
 &\bullet \quad \Bra{P_{\boldsymbol{\lambda}}(\vu)} \Xo_0 = \epsilon_{\boldsymbol{\lambda}}^*(\vu) \Bra{P_{\vl}(\vu)},\quad 
\epsilon_{\vl}^*(\vu) \in \mathbb{C}. 
\end{align}
Again, we omit $\vu$ unless mentioned otherwise.
\end{fact}

Note that for integers $n_i$ ($i=1,\ldots N$) with $n_i \geq \ell(\lambda^{(i)})$, 
the eigenvalues can be rewritten as 
\begin{equation}
    \epsilon_{\boldsymbol{\lambda}}(\vu) = (1-t^{-1})\sum_{i=1}^N u_i\sum_{k=1}^{n_i} q^{\lambda_k^{(i)}}t^{1-k} + \sum_{i=1}^N u_i t^{-n_i}. 
\end{equation}

By this Fact, 
it is easily seen that the generalized Macdonald functions $\ket{P_{\vl}}$ form a basis 
of $\mathcal{F}_{\vu}$. 
It is also known that the operator $\Xo_0$ is a member of a certain family 
of infinitely many commutative operators, 
and  $\ket{P_{\vl}}$ can be regarded as simultaneous eigenfunctions of them 
\cite{FHHSY}. 
Moreover,  $\ket{P_{\vl}}$ correspond to the torus fixed points in the instanton moduli space, 
which form a significant basis in the AGT correspondence. 
Note that this is a $q$-analogue of the AFLT basis \cite{AFLT}.

\begin{fact}[\cite{AFHKSY1}]
$\ket{P_{\vl}}$ and $\bra{P_{\vl}}$ are orthogonal, and 
their inner products take the form 
\begin{align}
\braket{P_{\vl}|P_{\vm}}=\prod_{i=1}^N\frac{c'_{\lambda^{(i)}}}{c_{\lambda^{(i)}}}\delta_{\vl,\vm}.
\end{align}
Here, we put
\begin{align}\label{eq: c and c'}
c_{\lambda}:= \prod_{(i,j)\in \lambda}(1-q^{a_{\lambda}(i,j)}t^{\ell_{\lambda}(i,j)+1}), \quad 
c'_{\lambda}:= \prod_{(i,j)\in \lambda}(1-q^{a_{\lambda}(i,j)+1}t^{\ell_{\lambda}(i,j)}). 
\end{align}
\end{fact}

\begin{definition}
We introduce $\ket{Q_{\vl}}$ such that $\braket{P_{\vl}|Q_{\vm}}=\delta_{\vl,\vm}$, 
i.e., $\ket{Q_{\vl}}:=\prod_{i=1}^N\frac{c_{\lambda^{(i)}}}{c'_{\lambda^{(i)}}} \ket{P_{\vl}}$. 
\end{definition}

\subsection{Screening currents and vertex operators}\label{SSec_VO}

This subsection and the next are devoted to the construction of 
an explicit formula for the generalized Macdonald functions $\ket{Q_{\vl}}$. 
Our explicit formula is made up of two ingredients. 
The one is a certain screened vertex operators on $\mathcal{F}_{\vu}$, 
and the other is 
an explicit solution to the eigenfunction problem  
associated with the Macdonald $q$-difference operator of A-type. 
In this subsection, 
we introduce the first ingredient, screening currents and some vertex operators. 
They play an important role also in a realization of our main vertex operator $\cV(x)$.

\begin{notation}
For an $N$-tuple of parameters $\vu=(u_1,\ldots, u_N)$, 
we write 
\begin{align}
t^{\pm \delta_i}\cdot \vu&:=(u_1,\ldots,u_{i-1}, t^{\pm 1} u_i, u_{i+1}, \ldots, u_N),\\ 
t^{\alpha_i}\cdot \vu&:=(u_1,\ldots, u_{i-1}, t u_i, t^{-1} u_{i+1}, u_{i+2}\ldots, u_N).
\end{align}
\end{notation}

\begin{definition}\label{def: screening}
Define the screening currents 
$S^{(i)}(y):\cF_{t^{\alpha_i} \cdot \vu} \to \cF_{\vu}$ 
by the following, 
\begin{equation}
\begin{split}
& S^{(i)}(z) := \overbrace{1\otimes\cdots\otimes 1}^{i-1}\otimes \phi^{\mathrm{sc}}(\gamma^{i-1} z)\otimes \overbrace{1\otimes \cdots \otimes 1}^{n-i-1}\,,\qquad i=1,\dots,N-1\,,\\
\end{split}
\end{equation}
with 
\begin{equation}
    \begin{split}
    \phi^{\mathrm{sc}}(z):= &\exp\left(-\sum_{n>0} \frac{1}{n}\frac{1-t^{n}}{1-q^n}\gamma^{2n} a_{-n} z^n\right)\exp\left(\sum_{n>0} \frac{1}{n}\frac{1-t^{-n}}{1-q^{-n}}a_{n} z^{-n}\right) \\
    & \otimes\exp\left(\sum_{n>0} \frac{1}{n}\frac{1-t^{n}}{1-q^n}\gamma^n a_{-n} z^n\right)\exp\left(-\sum_{n>0} \frac{1}{n}\frac{1-t^{-n}}{1-q^{-n}}\gamma^{-n} a_{n} z^{-n}\right)\,.
    \end{split}
\end{equation}
\end{definition}

\begin{remark}
In this paper, 
we do not attach zero modes to the screening currents. 
When discussing commutativity with the algebra $\WA$, 
we use a function satisfying some difference equation 
as in the following proposition. 
\end{remark}

\begin{proposition}\label{prop: commutativity of S}
Let $g_i(w)$ ($i=1,\ldots, N-1$) be a function which satisfies the difference equation $g_i(qw)=\frac{u_{i+1}}{tu_i}g_i(w)$. 
Then, 
the commutation relation between $S^{(i)}(w) g_i(w)$ and
the algebra $\WA$ is a total $q$-difference: 
\begin{align}
\left[S^{(i)}(w) g_i(w),X^{(k)}_n\right]=(1-T_{q,w})\left( \text{some operators} \right) \quad (\forall i,k,n). 
\end{align}
Here, $T_{q,w}$ is the difference operator such that $T_{q,w}: w \mapsto q w$. 
\end{proposition}

The proof is given in Appendix \ref{sec: pf of screening}. 
\begin{remark}
The function $g_{i}(w)$ can be realized as 
$g_{i}(w)=\frac{\theta_q(t^2u_{i}w/u_{i+1})}{\theta_q(tw)}$, 
and integrals of the screening currents commute with $\WA$ 
if the spectral parameters $\vu$ are degenerated 
(See Proposition \ref{prop: commutativity}).
\end{remark}

\begin{definition}\label{Def_screendV}
Define the screened vertex operators 
$\Phi^{(k)}(x):\cF_{t^{-\delta_{k+1}}\cdot \vu} \to \cF_{\vu} $ 
($k=0,1,\dots ,N-1$) by
\begin{align}
\Phi^{(0)}(x)=&:\exp\left( \sum_{n>0} \frac{1}{n}\frac{1-t^{n}}{1-q^n}a^{(1)}_{-n} x^n\right)\exp\left(\sum_{n>0} \frac{1}{n}\frac{1-\gamma^{2n}t^{n}}{1-q^{-n}}t^{-n}a^{(1)}_{n} x^{-n}\right)\\
&\times \exp\left(\sum_{n>0} \frac{1}{n}\frac{1-\gamma^{2n}}{1-q^{-n}}
\sum_{j=2}^N\gamma^{(j-1)n} a^{(j)}_{n} x^{-n}\right):, 
\end{align}
and  as the composition of operators, 
\begin{equation}
\Phi^{(k)}(x):=
\prod_{i=1}^k \frac{(q;q)_{\infty} (q/t;q)_{\infty}}{(\frac{q u_i}{u_{k+1}};q)_{\infty}(\frac{q u_{k+1}}{tu_{i}};q)_{\infty}}\cdot 
\oint_C \prod_{i=1}^k \frac{dy_i}{2 \pi \sqrt{-1}y_i}\Phi^{(0)}(x) S^{(1)}(y_1)\cdots S^{(k)}(y_k) g(x,y_1,\ldots,y_k)
\end{equation}
with an integral kernel 
\begin{equation}
g(x,y_1,\ldots,y_k)=\frac{\theta_q(tu_1 y_1/u_{k+1}x)}{\theta_q(t y_1/x)}
\prod_{i=1}^{k-1}\frac{\theta_q(tu_{i+1} y_{i+1}/u_{k+1}y_i)}{\theta_q(t y_{i+1}/y_i)}\,.
\end{equation}
Here, the contour of the integration $C$ is chosen so that $|t^{-1}|<|y_j/y_i|<|q|$ for $1\leq i<j \leq k$, 
and $|q/t|<|y_i/x|<1$ for $i\geq 1$. 
\end{definition}

\begin{remark}
The integration contour is well-defined if $|t^{-1}|<|q^{N-2}|$. 
Hence, in what follows, we assume this condition in this paper 
without Appendix \ref{App: reargof SM}. 
\end{remark}

These screened vertex operators can be obtained by 
combining a certain specialized intertwiner of the DIM algebra 
and replacing some Jackson integrals with contour ones. 
In Appendix \ref{App: Mac from top vertex}, 
we explain details of this construction and correspondence between 
Jackson integrals and contour ones.

\begin{remark}\label{rem: expand Phik}
The screened vertex operator $\Phi^{(k)}(x)$ 
is normalized such that 
$\oint \frac{dx}{2\pi \sqrt{-1}x}\Phi^{(k)}(x)\ketzero=\ketzero$. 
Indeed, $\Phi^{(k)}(x)$ can be expanded as 
\begin{align}
\Phi^{(k)}(y_0) =\oint_C \prod_{i=1}^k \frac{dy_i}{2 \pi \sqrt{-1}y_i} 
\sum_{r_1,\ldots , r_k \in \mathbb{Z}} \prod_{i=1}^k 
\frac{(tu_i/u_{k+1};q)_{r_i} }{(qu_i/u_{k+1};q)_{r_i}} (y_i/y_{i-1})^{r_i}
:\Phi^{(0)}(y_0) S^{(1)}(y_1)\cdots S^{(k)}(y_k):. \label{eq: expand scr. vetex}
\end{align}
This expansion follows from the operator products 
(\ref{eq: Phi0 S OPE})-(\ref{eq: S S OPE1}) and 
Ramanujan's ${}_1\psi_1$ summation formula (Fact \ref{fact: ramanujan}). 
By (\ref{eq: expand scr. vetex}), 
it can be shown that the coefficient of $y_0^{r_1}$ in the expansion of $\Phi^{(0)}(y_0)\ketzero$ 
is given by a finite sum and the constant term with respect to $y_0$ is $1$. 
\end{remark}

\begin{fact}[(5.2.1) in \cite{GR}]\label{fact: ramanujan}
\begin{align}
{}_1\psi_1(a;b;q;z):=
\sum_{n=-\infty}^{\infty} \frac{(a;q)_{n}}{(b;q)_{n}}z^n=
\frac{(q;q)_{\infty}(b/a;q)_{\infty}(az;q)_{\infty}(q/az;q)_{\infty}}
{(b;q)_{\infty}(q/a;q)_{\infty}(z;q)_{\infty}(b/az;q)_{\infty}} \quad 
(|b/a|<z<1). 
\end{align}
\end{fact}

\begin{remark}
Note that $g(x,y_1,\ldots,y_k)$ satisfies 
\begin{align}
T_{q,y_i}g(x,y_1,\ldots,y_k)=\frac{u_{i+1}}{u_i}g(x,y_1,\ldots,y_k). 
\end{align}
This corresponds to the difference equation which the function $g_i(z)$ satisfies
in Proposition \ref{prop: commutativity of S}. 
Moreover, 
the screening currents in $\Phi^{(k)}(x)$ commute with the algebra $\WA$ 
by taking the integrals. 
For more details, see the proof of the following lemma. 
\end{remark}

The screened vertex operators $\Phi^{(k)}(x)$ satisfy the following relation with 
the algebra $\WA$.

\begin{lemma}\label{lem: rel of X and Phi}
For $k=0,1,\ldots,N-1$ and $r=1,\ldots, N$, 
\begin{align}
X^{(r)}(z) \Phi^{(k)}(x)- \frac{1-(q/t)^rz/tx}{1-z/tx}\Phi^{(k)}(x) X^{(r)}(z) 
=u_{k+1} (1-t^{-1})\delta(tx/z)Y^{(r)}(x) \Phi^{(k)}(qx) \Psi^{+}(x), 
\end{align}
where $\delta(z)=\sum_{n\in \mathbb{Z}} z^n$ is the formal delta function, 
and we defined
\begin{align}
&Y^{(r)}(x):=\sum_{2\leq i_2<\cdots < i_r\leq N}
:\Lambda^{(i_2)}((q/t)tx) \cdots \Lambda^{(i_r)}((q/t)^{r-1}tx): 
u_{i_2}\cdots u_{i_r}, \\
&\Psi^+(z):= \exp\left(\sum \frac{1}{n}(1-\gamma^{2n}) \sum_{j=1}^N\gamma^{(j-1)n} a^{(j)}_n z^{-n}\right)
=\prod_{k=0}^\infty \frac{1}{\rho^{(N)} (\psi^+(\gamma^{-1}t^{-k}z))}\,.
\end{align}
Note that in particular $Y^{(1)}(z)=1$. 
\end{lemma}

For the proof, see Appendix \ref{sec: pf rel of X and Phi}

\subsection{Explicit formula for generalized Macdonald functions}

In this subsection, 
we explain an explicit solution to the eigenfunction equation of the Macdonald 
$q$-difference operator of A-type. 
By using its solution, 
we give an explicit formula for generalized Macdonald functions. 

First, we introduce a modified Macdonald 
$q$-difference operator. 
\begin{definition}
Define the operators on $\mathbb{C}[[x_2/x_1,x_3/x_2\ldots,x_n/x_{n-1}]]$ by 
\begin{align}
D^1_n(\boldsymbol{s};q,t)&:= \sum_{k=1}^n
s_k
\prod_{1\leq \ell < k} \frac{1-tx_k/x_{\ell}}{1-x_k/x_{\ell}} 
\prod_{k < \ell \leq n} \frac{1-x_{\ell}/tx_k}{1-x_{\ell}/x_{k}}
T_{q,x_k}, 
\end{align}
where $\boldsymbol{s} = (s_1,\dots,s_n)$ are generic parameters, 
and $T_{q,x_k}$ is the difference operator 
defined by  
\begin{align}
T_{q,x_k}F(x_1,\ldots,x_n)=F(x_1,\ldots,qx_k,\ldots,x_n). 
\end{align}
Though it is essentially equivalent, we also introduce the operator 
\begin{align}
\widetilde{D^1_n}(\boldsymbol{s};q, t) &:= \sum_{k=1}^n s_k
\prod_{1\leq \ell < k} \frac{1-tx_k/qx_{\ell}}{1-x_k/q x_{\ell}} 
\prod_{ k < \ell \leq n} \frac{1-qx_{\ell}/tx_k}{1-q x_{\ell}/x_{k}} T_{q^{-1},x_k}\,.
\end{align}
\end{definition}

An explicit formula for the eigenfunction of $\widetilde{D^1_n}(\boldsymbol{s};q, t)$ was conjectured in \cite{S}. 
Afterward, 
it was proved, and a solution to the bispectral problem was also given in \cite{NS}. 

\begin{definition}\label{def: p_n f_n}
Let $\boldsymbol{x}=(x_i)_{1\leq i \leq n}$ be inderminates and $\boldsymbol{s}=(s_i)_{1\leq i \leq n}$ be generic parameters.  
Define $p_n(\boldsymbol{x};\boldsymbol{s}|q,t)$ and $f_n(\boldsymbol{x};\boldsymbol{s}|q,t) \in \mathbb{C}[[x_2/x_1, x_3/x_2,\ldots , x_n/x_{n-1}]]$ 
by 
\begin{align}
p_n(\boldsymbol{x};\boldsymbol{s}|q,t) &=
\sum_{\theta \in M_n} c_n(\theta;\boldsymbol{s}|q,t) \prod_{1\leq i<j\leq n}(x_j/x_i)^{\theta_{i,j}}, \\
f_n(\boldsymbol{x};\boldsymbol{s}|q,t) &=\prod_{1\leq k<\ell\leq n}(1-x_\ell/x_k)\,\cdot\,p_n(\boldsymbol{x};\boldsymbol{s}|q^{-1},t^{-1}). 
\end{align}
For convenience, we put $f_0=p_0=1$. 
Here $M_n$ is the set of all $n \times n$ upper triangular matrices with nonnegative integers, 
in which diagonal elements are $0$. 
$c_n(\theta;\boldsymbol{s}|q,t)$ are coefficients defined by the following recurrence relations
\begin{equation}
    \begin{split}
&c_1(-;s_1,q,t)=1\,,\\
&c_n((\theta_{i,j})_{1\leq i<j\leq n};(s_i)_{1\leq i \leq n}|q,t)\\
&\quad = c_{n-1}((\theta_{i,j})_{1\leq i<j\leq n-1}; (q^{-\theta_{i,n}}s_i)_{1\leq i \leq n}|q,t) \times d_n((\theta_{i,n})_{1\leq i \leq n-1};(s_i)_{1\leq i \leq n}|q,t),\\
&(n = 2, 3, \dots),
    \end{split}
\end{equation}
with
\begin{align}
&d_n((\theta_i)_{1\leq i \leq n-1};(s_i)_{1\leq i \leq n}|q,t)	\nonumber  \\
& \qquad =\prod_{i=1}^{n-1}
(q/t)^{\theta_{i}} { (t;q)_{\theta_{i}} \over (q;q)_{ \theta_{i}} }
{(t s_n/s_i;q)_{\theta_{i}} \over (q s_n/s_i;q)_{\theta_{i}} }\prod_{1\leq i<j\leq n-1}
{(t s_j/s_i ;q)_{\theta_{i}} \over (q s_j/s_i;q)_{\theta_{i}} }
{(q^{-\theta_{j}}q s_j/t s_i;q)_{\theta_{i}} \over 
	(q^{-\theta_{j}}s_j/s_i;q)_{\theta_{i}} }. \label{eq: form of d}
\end{align}
\end{definition}

From the recurrence relations, we can compute the explicit form of $c_n$ as
\begin{equation}
\begin{split}
&c_n((\theta_{i,j})_{1\leq i<j\leq n};(s_i)_{1\leq i \leq n}|q,t) =\prod_{1\leq i<j\leq n} (q/t)^{\theta_{i,j}}{ (t;q)_{\theta_{i,j}} \over (q;q)_{ \theta_{i,j}} }
{(q^{\sum_{a=j+1}^n (\theta_{i,a}-\theta_{j,a})} t s_j/s_i;q)_{\theta_{i,j}} \over (q^{\sum_{a=j+1}^n (\theta_{i,a}-\theta_{j,a})} q s_j/s_i;q)_{\theta_{i,j}} }\\
&\quad\times\prod_{k=3}^n \prod_{1\leq l<m\leq k-1} { (q^{\sum_{b=k+1}^n (\theta_{l,b}-\theta_{m,b})}t s_m/s_l;q)_{\theta_{l,k}} \over (q^{\sum_{b=k+1}^n (\theta_{l,b}-\theta_{m,b})}q s_m/s_l;q)_{\theta_{l,k}} } 
{ (q^{-\theta_{m,k}} q^{\sum_{b=k+1}^n (\theta_{l,b}-\theta_{m,b})}q s_m/t s_l;q)_{\theta_{l,k}} \over (q^{-\theta_{m,k}}q^{\sum_{b=k+1}^n (\theta_{l,b}-\theta_{m,b})} s_m/s_l;q)_{\theta_{l,k}} }. 
\end{split}
\end{equation}

\begin{fact}[\cite{NS,S,BFS}]\label{fact: eigen fn of D}
The function $p_n(\boldsymbol{x};\boldsymbol{s}|q,t)$ is an unique formal solution to the eigenfunction equation 
\begin{align} 
D^1_n(\boldsymbol{s};q,t)p_n(\boldsymbol{x};\boldsymbol{s}|q,t) =(s_1+\cdots +s_n)p_n(\boldsymbol{x};\boldsymbol{s}|q,t),
\end{align} 
up to scalar multiples. 
Equivalently, $f_n(\boldsymbol{x};\boldsymbol{s}|q,t)$ is an unique function (up to scalar multiples) such that 
\begin{align}
&\widetilde{D^1_n}(\boldsymbol{s};q,t)f_n(\boldsymbol{x};\boldsymbol{s}|q,t) =(s_1+\cdots +s_n)f_n(\boldsymbol{x};\boldsymbol{s}|q,t).
\end{align}
\end{fact}

\begin{remark}
The function $p_n(\boldsymbol{x};\boldsymbol{s}|q,t)$ and $f_n(\boldsymbol{x};\boldsymbol{s}|q,t)$ are dual to each other 
in the sense of Lemma \ref{Lem_delta1}. 
Furthermore, it is known that the function $p_n(\boldsymbol{x};\boldsymbol{s}|q,t)$ can also be regarded 
as a holomorphic function, and its analyticity is clarified in \cite{NS}. 
We explain its analyticity in Section \ref{sec: analyticity} (Fact \ref{fact_analyticity of p}).  
\end{remark}

In the beginning of this subsection, 
we stated to give an explicit formula for generalized Macdonald functions. 
For that purpose, 
let us introduce the following notations with respect to the screened vertex operators.

\begin{notation}
For  an $N$-tuple of nonnegative integers $\vnn=(n_1,\ldots ,n_N)$, 
we write 
\begin{align}
&|\vnn|:=\sum_{s=1}^N n_s, \qquad 
\ik=\lbl{i}{k}{n}:=\sum_{s=1}^{i-1}n_s+k 
\quad (1\leq i \leq N, \, \leq k \leq n_i), \\
& t^{\pm \vnn} \cdot \vu:=(t^{\pm n_1}u_1, \ldots, t^{\pm n_{N}} u_N). 
\end{align}
\end{notation}

\begin{definition}\label{def_Vn}
Let $\vnn=(n_1,\ldots ,n_N)$ be an $N$-tuple of nonnegative integers. 
Define the operator 
$V^{(\vnn)}(x_1,\ldots,x_{|\vnn|}): \cF_{t^{-\vnn}\cdot\vu} \to \cF_{\vu}$  by 
\begin{align}
\begin{split}
    &V^{(\vnn)}(x_1,\ldots,x_{|\vnn|})=\\
    &\quad\Phi^{(0)}(x_1)\cdots \Phi^{(0)}(x_{n_{1}})
\Phi^{(1)}(x_{n_1+1})\cdots \Phi^{(1)}(x_{n_{1}+n_{2}})\cdots 
\Phi^{(N-1)}(x_{\lbl{N}{1}{n}})\cdots \Phi^{(N-1)}(x_{|\vnn|}).
\end{split}
\end{align}
We occasionally write 
$V^{(\vnn)}(x_1,\ldots,x_{|\vnn|})=\Vweight{(\vnn)}{\vu\\ t^{-\vnn}\cdot\vu}{x_1,\ldots,x_{|\vnn|}}$
in order that the spectral parameters of representations are clear. 
\end{definition}

Now, we obtain an explicit formula for the generalized Macdonald functions $\ket{Q_{\vl}}$ 
in terms of the screened vertex operators. 
The action of the screened vertex operators is controlled 
by the function $f_{n}(\boldsymbol{x};\boldsymbol{s}|q,t)$. 
Further, we can also obtain an expression of $p_n$ by using $\bra{P_{\vl}}$ and 
the operator $V^{(\vnn)}$.  
These expressions of $\ket{Q_{\vl}}$ and $p_n$ are symmetric to each other in some sense. 

\begin{theorem}\label{thm: GM}
Let  $\vnn=(n_1, \ldots, n_N)$ be an $N$-tuple of integers 
satisfying that $n_i \geq \ell(\lambda^{(i)})$ for all $i$. 
Then under the identification $s_{\nik}=q^{\lambda^{(i)}_k} t^{1-k}u_i$ ($1 \leq k < n_i, \, i=1,\ldots, N$), 
the followings hold: 
\begin{align}
\left[ \boldsymbol{x}^{-\lambda} f_{|\vnn|}(\boldsymbol{x};\boldsymbol{s}|q,q/t) \Vweight{(\vnn)}{\vu\\  t^{-\vnn}\cdot\vu }{x_1,\ldots,x_{|\vnn|}}
\ketzero\right]_{\boldsymbol{x},1}
=\cR^{\vnn}_{\boldsymbol{\lambda}}(\vu)\ket{Q_{\boldsymbol{\lambda}}}\,,\label{GM_formula}\\
 \boldsymbol{x}^{-\vl}\bra{P_{\vl}} \Vweight{(\vnn)}{\vu\\ t^{-\vnn}\cdot\vu}{x_1,\ldots,x_{|\vnn|}}\ketzero
=
\mathcal{R}^{\vnn}_{\vl} (\vu)\,p_{|\vnn|}(\boldsymbol{x};\boldsymbol{s}|q,q/t). 
\label{GM_formula 2}
\end{align}
Here, we set 
\begin{align}
\boldsymbol{x}^{-\vl}:=\prod_{i=1}^N\prod_{k=1}^{n_i}x_{\nik}^{-\lambda^{(i)}_k}. 
\end{align}
$[\cdots]_{\boldsymbol{x},1}$ means to take the constant term in $\cdots$ with respect to $\vx$, and 
$\cR^{\vnn}_{\boldsymbol{\lambda}}(\vu) \in \mathbb{C}(u_1,\ldots , u_N)$ is some coefficient.
\begin{proof}

By Lemma \ref{lem: rel of X and Phi} and the operator product (\ref{eq: Psi+ Phi0 OP}) 
we can get the relation 
\begin{align}
&\Xo(z)\Vweight{(\vnn)}{\vu\\ t^{-\vnn}\cdot\vu}{x_1,\ldots,x_{|\vnn|}}
= \prod_{k=1}^{|\vnn|} \frac{1-qz/t^2x_k}{1-z/tx_k} \cdot
V^{(\vnn)}(x_1,\ldots,x_{|\vnn|})\Xo(z) \nonumber \\
&+(1-t^{-1}) \sum_{i=1}^N\sum_{k=1}^{n_i} 
u_{i}t^{1-k}\delta(tx_{\ik}/z )
\prod_{1\leq \ell < \ik} \frac{1-qx_{\ik}/tx_{\ell}}{1-x_{\ik}/x_{\ell}} \cdot
\prod_{\ik < \ell \leq |\vnn|} \frac{1-tx_{\ell}/qx_{\ik}}{1-x_{\ell}/x_{\ik}}\nonumber \\
& \qquad \qquad \qquad \qquad \times T_{q,x_{\ik}}\left( V^{(\vnn)}(x_1,\ldots,x_{|\vnn|})\right) \Psi^+(x_{\ik}). 
\end{align}
By taking the constant term of $z$, 
we have 
\begin{equation}\label{eq: X_0 V |0>}
\begin{split}
&\Xo_0 V^{(\vnn)}(x_1,\ldots,x_{|\vnn|}) \ketzero= 
V^{(\vnn)}(x_1,\ldots,x_{|\vnn|})\ketzero \left( \sum_{i=1}^N t^{-n_{i}}u_i \right) \\
&\qquad +(1-t^{-1}) D^1_{|\vnn|}(\boldsymbol{s}\big|_{\lambda_k^{(i)} = 0};q, q/t)  V^{(\vnn)}(x_1,\ldots,x_{|\vnn|})\ketzero\,.
\end{split}
\end{equation}
Noting the integrals of the total difference vanish, integrating by parts gives the following equality.
\begin{align}
&\left[ \boldsymbol{x}^{-\boldsymbol\lambda} f(\boldsymbol{x};\boldsymbol{s}|q,t) D^1_n(\boldsymbol{s}|_{\lambda_k^{(i)} = 0};q, q/t) V^{(\vnn)}(x_1,\ldots,x_{|\vnn|})\ketzero\right]_{\boldsymbol{x},1}\nonumber \\
&\qquad =\left[\left(\widetilde{D^1_{|\vnn|}}(\boldsymbol{s}\big|_{\lambda_k^{(i)} = 0};q, q/t) \boldsymbol{x}^{-\boldsymbol\lambda} f(\boldsymbol{x};\boldsymbol{s}|q, q/t)\right) V^{(\vnn)}(x_1,\ldots,x_{|\vnn|})\ketzero\right]_{\boldsymbol{x},1} \nonumber\\
&\qquad = \left[\left(\boldsymbol{x}^{-\boldsymbol\lambda}\widetilde{D^1_{|\vnn|}}(\boldsymbol{s};q, q/t) f(\boldsymbol{x};\boldsymbol{s}|q, q/t)\right) V^{(\vnn)}(x_1,\ldots,x_{|\vnn|})\ketzero\right]_{\boldsymbol{x},1}\nonumber \\
&\qquad = \left(\sum_{i = 1}^N \sum_{k=1}^{n_i} s_{\nik} \right)\left[\boldsymbol{x}^{-\vl} f(\boldsymbol{x};\boldsymbol{s}|q, q/t) V^{(\vnn)}(x_1,\ldots,x_{|\vnn|})\ketzero\right]_{\boldsymbol{x},1}\,. 
\end{align}
Here, in the third equality, we made use of Fact \ref{fact: eigen fn of D}. 
This equality identify the LHS of (\ref{GM_formula}) with the generalized Macdonald functions up to the normalization. 

From (\ref{eq: X_0 V |0>}), we can also have 
\begin{align}
D^1_{|\vnn|} (\vs;q,q/t)\vx^{-\vl}\bra{P_{\vl}} V^{(\vnn)}(x_1,\ldots,x_{|\vnn|})\ketzero 
= (s_1+\cdots +s_{|\vnn|}) \vx^{-\vl}\bra{P_{\vl}} V^{(\vnn)}(x_1,\ldots,x_{|\vnn|})\ketzero.
\end{align}
Therefore, $\vx^{-\vl}\bra{P_{\vl}} V^{(\vnn)}\ketzero$ is an eigenfunction 
of $D^1_{|\vnn|}$. This implies (\ref{GM_formula 2}). 
Since the constant terms in $p_n$ and $f_{n}$ are $c_1(-;s_1|q,t)=1$, 
it is clear that (\ref{GM_formula}) and (\ref{GM_formula 2}) have 
the same proportionality constant $\mathcal{R}^{\vnn}_{\vl}(\vu)$. 
\end{proof}
\end{theorem}

The following proposition gives the explicit form of $\cR^{\vnn}_{\boldsymbol{\lambda}}(\vu)$.
\begin{proposition}\label{Prop_R}
The coefficient $\cR^{\vnn}_{\boldsymbol{\lambda}}(\vu)$ is of the form,
\begin{equation}
\mathcal{R}^{\vnn}_{\vl}(\vu) =
\gamma^{\sum_{i=1}^N (i-1)|\lambda^{(i)}|}
\prod_{k=2}^N\prod_{i=1}^{n_k}\prod_{l=1}^{k-1} 
\frac{(t^{-n_l+i}u_l/u_k;q )_{-\lambda^{(k)}_i}}{(qt^{-n_l+i-1}u_l/u_k;q )_{-\lambda^{(k)}_i}}\,.
\end{equation}
\end{proposition}

The proof of Proposition \ref{Prop_R} is in subsection \ref{Proof_R}.

\subsection{Analyticity of matrix elements of $V^{(\boldsymbol{n})}$}\label{sec: analyticity}
In this subsection, we describe the fact that the Macdonald functions $p_n$ can be treated as 
holomorphic functions. 
Moreover, we show the matrix elements of $V^{(\boldsymbol{n})}$ can meromorphically extend to the whole $\mathbb{C}^{|\boldsymbol{n}|-1}$, 
and thus we can deal with the variables $\boldsymbol{x}$ as those in that space, not as formal variables.

\begin{definition}
Define the projection for the canonical coordinates $\pi_n : (\mathbb{C}^*)^n \to \mathbb{C}^{n-1}$ such that for each $a = (a_1,\dots,a_n)\in (\mathbb{C}^*)^n$, 
\begin{equation}
    \pi_n(a) =  (a_2/a_1, \dots, a_n/a_{n-1}) \in \mathbb{C}^{n-1}\,.
\end{equation}
\end{definition}

\begin{notation}
In what follows in this subsection, we use the following notations.
\begin{align*}
    &z := (z_1,\dots,z_{|\vnn|-1}) = \pi_{|\vnn|}((x_1,\dots,x_{|\vnn|}))\,, \\
    &\tilde{z} := (\tilde z_1,\dots,\tilde z_{|\vnn|+|\vmm|-1}) = \pi_{|\vnn|+|\vmm|}((x_1,\dots,y_1,\dots))\,,\\
    & w := (w_1,\dots,w_{|\vnn|-1}) = \pi_{|\vnn|}((s_1,\dots,s_{|\vnn|}))\,, \\
    &\tilde{w} := (\tilde w_1,\dots,\tilde w_{|\vnn|+|\vmm|-1}) = \pi_{|\vnn|+|\vmm|}((s_1,\dots,s_{|\vnn|+|\vmm|}))\,.
\end{align*}
\end{notation}

\begin{definition}
Define the open subset $D_w \subset \mathbb{C}^{n-1}$ by
\begin{equation}
    D^n_w = \{w = (w_1,\dots,w_{n-1})\in\mathbb{C}^{n-1}\big|\,\, w_i\cdots w_{j-1} \notin q^{-\mathbb{Z}}\cup \{0\},\,\, 1\leq i< j\leq n  \},
\end{equation}
so that
\begin{equation}
       \pi^{-1}(D^n_w(r)) = \{s = (s_1,\dots,s_{n})\in(\mathbb{C}^*)^n\big|\,\, s_j/s_i \notin q^{-\mathbb{Z}},\,\, 1\leq i< j\leq n \}\,.
\end{equation}
\end{definition}
\begin{definition}
Define the subsets $U^n_z(r), B^n_z(r) \subset \mathbb{C}^{n-1}$ by
\begin{align}
    & U^n_z(r) = \{z = (z_1,\dots,z_{n-1})\in\mathbb{C}^{n-1}\big|\,\, |z_i|< r,\,\, i = 1,\dots , n-1 \}\,,\\
    &B^n_z(r) = \{z = (z_1,\dots,z_{n-1})\in\mathbb{C}^{n-1}\big|\,\, |z_i|\leq r,\,\, i = 1,\dots , n-1\}\,,
\end{align}
so that
\begin{equation}
    \pi^{-1}(B^n_z(r)) = \{x = (x_1,\dots,x_{n})\in(\mathbb{C}^*)^n\big|\,\, |x_j/x_i|\leq r^{j-i},\,\, 1\leq i< j\leq n \}\,.
\end{equation}
\end{definition}

In order to see the matrix element above can be analytically continued 
except for its singularities, we make use of the following fact, proved in \cite{NS}.
\begin{fact}[\cite{NS}]\label{fact_analyticity of p}
Let $\tau$ be a generic complex parameter. 
For $n=2,3,\dots$, we regard $p_n(\vx;\vs|q,\tau)$ as a formal power series in $z = (z_1,\dots,z_{n-1})$ with coefficients in $\mathcal{O}(D^n_w)$, the ring of holomorphic functions on $D^n_w$:
\begin{equation}\label{temp}
    p_n(\vx;\vs|q,\tau) =\sum_{\theta \in M_n} c_n(\theta;\vs|q,\tau) \prod_{1\leq i<j\leq n}(x_j/x_i)^{\theta_{i,j}}\,\, \in \mathcal{O}(D^n_w)[[z]].
\end{equation}
Set $r_0 = |q/\tau|^{\frac{n-2}{n-1}}$ if $|q/\tau|\leq 1 $, and $r_0 = |\tau/q|$ if $|q/\tau|\geq 1$. Then for any compact subset $K\subset D_w$ and for any $r< r_0$, this series \ref{temp} is absolutely convergent,
uniformly on $B^n_z(r) \times K$.
Hence $p_n(\vx;\vs|q,\tau)$ defines a holomorphic function on $U^n_z(r_0) \times D^n_w$.
\end{fact}

Then, we have the following theorem.
\begin{theorem}\label{thm_Analytic V}
The operator $V^{(\vnn)}(x_1,\ldots,x_{|\vnn|})$ is well-defined on $\pi^{-1}_{|\vnn|}(U^{|\vnn|}_z(r_0))$ with $r_0 = |t^{-1}|$,
in the sense that its matrix elements are the holomorphic functions there.
\end{theorem}

We prepare the following lemma. 

\begin{lemma}\label{lem: <P|VV|0>}
Let $\vu=t^{-\vnn}\cdot\vv$, $\vw=t^{-\vmm}\cdot\vu$ 
and $n_i\geq \ell(\lambda^{(i)})$ ($\forall i$). Then 
\begin{equation}\label{pr_3.19}
\vx^{-\vl}\bra{P_{\vl}} 
\Vweight{(\vnn)}{\vv \\ \vu}{x_1,\ldots,x_{|\vnn|}}
\Vweight{(\vmm)}{\vu \\ \vw}{y_1,\ldots,y_{|\vmm|}}\ketzero  
= \mathcal{R}_{\vl}(\vv) \,p_{|\vnn|+|\vmm|}((\vx,\boldsymbol y);\vs|q,q/t) 
\end{equation}
under the identification 
\begin{align}
&s_{\lbl{i}{k}{n}}=q^{\lambda^{(i)}_k} t^{1-k}v_i 
\quad (1 \leq k \leq n_i, \, i=1,\ldots, N), \label{eq: form of s 1}\\
&s_{|\vnn|+[i,k]_{\vmm}}= t^{1- n_{i} -k}v_i 
\quad (1 \leq k \leq m_i, \, i=1,\ldots, N)\label{eq: form of s 2}. 
\end{align}
\end{lemma}

This lemma can be derived as a corollary of Theorem \ref{thm: GM} 
by noting the normalization of the screened vertex operators (Remark \ref{rem: expand Phik}). 

\vspace{0.5eM}

\noindent 
\textit{Proof of Theorem \ref{thm_Analytic V}. }
First, by Fact \ref{fact_analyticity of p}, the left hand side of (\ref{pr_3.19}) is a holomorphic function on $U^{|\vnn|+|\vmm|}_{\tilde z}(r_0)$.
Then, through the pull-back $\pi^*_{|\vnn|+|\vmm|}$, we regard it as the holomorphic function in $\pi^{-1}_{|\vnn|+|\vmm|}(U^{|\vnn|+|\vmm|}_{\tilde z}(r_0))$.
Thus, 
once we fix $y \in (\mathbb{C}^*)^{|\vmm|} $ such that $|y_j/y_i| < r^{j-i}_0 \,\, (1\leq i< j\leq |\vmm|)$, 
it becomes a holomorphic function of $x \in \widetilde{\pi}^{-1}_{|\vnn|,y_1}(U^{|\vnn|}_{z}(r_0))$, where 
\begin{equation}
    \widetilde{\pi}^{-1}_{|\vnn|,y_1}(U^{|\vnn|}_{ z}(r_0)) := \{ (x_1,\dots,x_{|\vnn|})\in(\mathbb{C}^*)^{|\vnn|}\big|\,\, |x_j/x_i|< r_0^{j-i}\, (1\leq i< j\leq |\vnn|), \,\, |y_1/x_{|\vnn|}| < r_0 \}\,.
\end{equation}

Second, by multiplying both sides of (\ref{pr_3.19}) by $\boldsymbol y^{-\mu} f_{|\vmm|}(\boldsymbol y;\vs|q,q/t) $ and taking the constant term in $\boldsymbol y$,
the left hand side of (\ref{pr_3.19}) becomes $  \vx^{-\vl}\bra{P_{\vl}} 
\Vweight{(\vnn)}{\vv \\ t^{-\vnn}\cdot \vv}{x_1,\ldots,x_{|\vnn|}}\ket{Q_{\vm}}$ with the help of Theorem \ref{thm: GM}.
Note that, in this case, taking the constant term means the usual contour integral in $y$, and the contour is chosen to separate $x$ and $y$.
This procedure does not affect the holomorphicity in $x$.
At this point, $x^{-\vl}\bra{P_{\vl}} 
\Vweight{(\vnn)}{\vv \\ t^{-\vnn}\cdot \vv}{x_1,\ldots,x_{|\vnn|}}\ket{Q_{\vm}}$ can holomorphically extend to $\pi^{-1}_{|\vnn|}(U^{|\vnn|}_z(r_0))$, because we can take arbitrary small $y_1$. 
As a result, Theorem \ref{thm_Analytic V} is proved.
\qed

\subsection{Integral forms}
In this subsection, we define the renormalization of generalized Macdonald functions, 
which correspond to the integral form of ordinary Macdonald functions up to monomials.
In this basis, 
the main theorem for the matrix elements of a main vertex operator can be written in smarter way than 
in the basis formed by $\ket{P_{\boldsymbol{\lambda}}}$, defined above.  

\begin{definition}\label{def_NekFac}
Define
\begin{align}
N_{\lambda \mu}(u) := \prod_{(i,j)\in \lambda} \left( 1- u q^{a_{\lambda}(i,j)}t^{\ell_{\mu}(i,j)+1} \right)  \prod_{(i,j)\in \mu} \left( 1- u q^{-a_{\mu}(i,j)-1} t^{-\ell_{\lambda}(i,j)} \right).
\end{align}
$N_{\lambda \mu}(u)$ is called the Nekrasov factor. 
\end{definition}
This factor appears in the instanton partition functions in five dimensional $\mathcal{N}=1$ supersymmetric gauge theories.

\begin{definition}
Set
\begin{align}
&\mathcal{C}_{\vl}^{(+)}(\vu) := \xi_{\vl}^{(+)}(\vu) \times 
\prod_{1\leq i<j \leq N} N_{\lambda^{(i)}, \lambda^{(j)}}(qu_i/tu_j) \cdot
\prod_{k=1}^N c_{\lambda^{(k)}}, \label{eq: ren. const. of K} \\
&\mathcal{C}_{\vl} ^{(-)}(\vu) :=  \xi_{\vl}^{(-)}(\vu) \times
\prod_{1\leq i<j \leq N} N_{\lambda^{(j)}, \lambda^{(i)}}(qu_j/tu_i) \cdot
\prod_{k=1}^N c_{\lambda^{(k)}},\\
& \xi_{\vl}^{(+)}(\vu):=
\prod_{i=1}^N(-1)^{(N-i+1)|\lambda^{(i)}|}
u_{i}^{(-N+i)|\lambda^{(i)}|+\sum_{k=1}^{i}|\lambda^{(k)}|}\nonumber \\
&\qquad \quad \times \prod_{i=1}^N
(q/t)^{\left(\frac{1-i}{2}\right)|\lambda^{(i)}|} 
q^{\left(i-N\right)\left(n(\lambda^{(i)'})+|\lambda^{(i)}|\right)}
t^{\left(N-i-1\right) \left( n(\lambda^{(i)})+|\lambda^{(i)}| \right)},\\
& \xi_{\vl}^{(-)}(\vu):=\prod_{i=1}^N(-1)^{i|\lambda^{(i)}|}
u_{i}^{(-i+1)|\lambda^{(i)}|+\sum_{k=i}^{N}|\lambda^{(k)}|}\nonumber \\
&\qquad \quad \times \prod_{i=1}^N
(q/t)^{\left(\frac{i-1}{2}\right)|\lambda^{(i)}|} t^{|\lambda^{(i)}|}
q^{\left(1-i\right)\left(n(\lambda^{(i)'})+|\lambda^{(i)}|\right)}
t^{\left(i-2\right) \left( n(\lambda^{(i)})+|\lambda^{(i)}| \right)},
\end{align}
where $c_{\lambda}$ is defined in (\ref{eq: c and c'}) and 
$n(\lambda)=\sum_{j\geq1} (j-1)\lambda_j$. 

\end{definition}

\begin{definition}\label{def: K}
Define $\ket{K_{\vl}(\vu)} \in \mathcal{F}_{\vu}$ and 
$\bra{K_{\vl}(\vu)} \in \mathcal{F}^*_{\vu}$ by
\begin{align}
\ket{K_{\vl}}=\ket{K_{\vl}(\vu)}:=\mathcal{C}_{\vl}^{(+)}(\vu) \ket{P_{\vl}(\vu)}, \quad 
\bra{K_{\vl}}=\bra{K_{\vl}(\vu)}:=\mathcal{C}_{\vl} ^{(-)}(\vu)\bra{P_{\vl}(\vu)}.
\end{align}
\end{definition}

This normalization arises from the following conjecture with 
respect to the expansion coefficient in the PBW-type basis $\ket{X_{\vl}}$ and $\bra{X_{\vl}}$. 
\begin{conjecture}
\begin{align}
&\ket{K_{\boldsymbol{\lambda}}}= \sum_{\vm} \alpha^{(+)}_{\vl \vm} \ket{X_{\vm}}, \quad 
\alpha^{(+)}_{\vl, ((1^{|\vl|}), \emptyset,\ldots,\emptyset)} =1, \\
&\bra{K_{\boldsymbol{\lambda}}}= \sum_{\vm} \alpha^{(-)}_{\vl \vm} \bra{X_{\vm}}, \quad 
\alpha^{(-)}_{\vl, ((1^{|\vl|}), \emptyset,\ldots, \emptyset)} =1.
\end{align}
\end{conjecture}

In Appendix \ref{sec: ex of K}, we give examples of the transition matrices $\alpha^{(\pm)}_{\vl \vm}$. 

\begin{definition}[(Taki's) flaming factors]\label{def: flaming factor}
Define 
\begin{align}
&f_{\lambda}:=(-1)^{|\lambda|}q^{n(\lambda')+|\lambda|/2}t^{-n(\lambda)-|\lambda|/2},\\
&g_{\lambda}:=q^{n(\lambda')}t^{-n(\lambda)}.
\end{align}
\end{definition}

\begin{proposition}
\begin{align}
\braket{K_{\vl}|K_{\vl}}=
&\left((-1)^{N}\gamma^2 e_N(\vu)\right)^{|\vl|}
\prod_{i=1}^N 
\left( u_i^{|\lambda^{(i)}|}
\gamma^{-2|\lambda^{(i)}|}
g_{\lambda^{(i)}}
\right)^{\left(2-N\right)}
\cdot 
\prod_{i,j=1}^N N_{\lambda^{(i)}, \lambda^{(j)}}(qu_i/tu_j), 
\end{align}
where we put $e_{N}(\vu):=\prod_{i=1}^Nu_i$.
\end{proposition}

\begin{proof} 
This follows from 
\begin{align}
&c_{\lambda}c'_{\lambda}=(-1)^{(|\lambda|)} q^{n(\lambda') +|\lambda|}t^{n(\lambda)}N_{\lambda,\lambda}(1), \\ 
&N_{\lambda,\mu}(\gamma^{-1}x)=N_{\mu, \lambda}(\gamma^{-1}x^{-1})x^{|\lambda|+|\mu|}\frac{f_{\lambda}}{f_{\mu}}. 
\end{align}
\end{proof}

\section{Proof of Main Theorem}\label{Sec_matrix element}
This section is organized as follows.
In Section \ref{sec_4.1}, before going into the details of the proof of the main theorem, 
we recall what is  our main claim in this paper. 
In Section \ref{sec_4.2}, we derive a transformation formula for the multiple basic hypergeometric series, which is one of the key ingredients
in the proof in Section \ref{sec_4.3} of the main theorem.

\subsection{Statement of main theorem}\label{sec_4.1}
First, we define the vertex operator $\cV(x)$ by relations with the algebra $\WA$. 
We call $\cV(x)$ the Mukad\'{e} operator as explained in Introduction.

\begin{definition}\label{def_V_gen}
Define the linear operator $\cV(x)=\cVweight{\vv \\ \vu}{x} : \cF_{\boldsymbol{u}} \to \cF_{\boldsymbol{v}}$ 
by the following relations: 
\begin{equation}\label{def_V}
\left(1-\frac{x}{z}\right)X^{(i)}(z) \cV(x) 
=\left(1- (t/q)^i\frac{x}{z}\right)\cV(x) X^{(i)}(z), 
\qquad i \in \{1,2,\dots,N\}\,
\end{equation}
and $\bra{\boldsymbol{0}}\cV(x) \ket{\boldsymbol{0}}=1$. 
\end{definition}

Then we have the following proposition.

\begin{proposition}\label{prop: uniqueness of V}
The $\cV(x)$ exists uniquely.
\begin{proof}
If the operator $\cV(x)$ exists, 
the uniqueness is clear by definition of $\cV(w)$ and the fact that 
the vectors $\ket{X_{\vl}}$ form a basis (Fact \ref{fact: PBW basis}). 
(See also Appendix \ref{sec: Ex of Mat el}) 
The existence of $\cV(w)$ will be shown late in the Proposition \ref{prop: existance_V}.
\end{proof}
\end{proposition}

\begin{remark}
The realization of $\cV(x)$ given in Section \ref{Sec_Sdual}
shows that  $\cV(x)$
satisfies the relation with the algebra $\mathcal{U}$ 
introduced in \cite[Definition 3.12]{AFHKSY1} 
after some simple renormalization. 
However, if we adopt its relation as definition of the vertex operator, 
the uniqueness is nontrivial.  
\end{remark}

In Appendix \ref{sec: Ex of Mat el}, 
we demonstrate the calculation of the matrix elements of $\cV(x)$ 
by using the defining relation (\ref{def_V}). 
The matrix elements with respect to generalized Macdonald functions are factorized 
and reproduce the Nekrasov factor: 

\begin{theorem}\label{thm: matrix elements of V}
\begin{align}
\bra{K_{\boldsymbol{\lambda}}(\vv)}\cV (x)\ket{K_{\boldsymbol{\mu}}(\vu)}
=&
\frac{\left((-\gamma^2)^N e_{N}(\vu)x \right)^{|\vl|}}
{\left( \gamma^2 x \right)^{|\vm|}}
\prod_{i=1}^N
\frac{u_i^{|\mu^{(i)}|}g_{\mu^{(i)}}}
{\left(v_i^{|\lambda^{(i)}|}g_{\lambda^{(i)}}\right)^{N-1}} \cdot
\prod_{i,j=1}^N N_{\lambda^{(i)},\mu^{(j)}}(qv_i/tu_j). \label{eq: matrix element} 
\end{align}
\end{theorem}

As a direct result of this theorem, we obtain an algebraic description of the 5D (K-theoretic) AGT correspondence \cite{AGT, AY1, AY2}. 
Theorem \ref{thm: matrix elements of V} makes us possible to calculate the multi-point functions. 
In particular, the two-point function, 
which can be regarded as $q$-analogue of the four-point conformal block in 2D CFT
\footnote{
As $X^{(1)}_0 \ketzero=(u_1+\cdots+u_N)\ketzero$, 
the vectors $\brazero$ and $\ketzero$ are not the ordinary vacuum states for the algebra $\WA$ 
but highest weight vectors of highest weight $\vu$. 
Hence, 
the function $\brazero \cV (x_1)\cdots \cV (x_n) \ketzero$ 
can be regarded as $q$-analogue of the $(n+2)$-point conformal block.}, 
corresponds to the Nekrasov partition function 
of the 5-dimensional $\mathcal{N}=1$ $U(N)$ gauge theory 
with $2N$ fundamental matters \cite{AK1, AK2}. (See also Remark 3.14, 3.15 in \cite{AFHKSY1} for more detail.): 
\begin{align}
\brazero \cVweight{\vw \\ \vv}{z_1}\cVweight{\vv \\ \vu}{z_2}\ketzero
&=\sum_{\vl}\frac{\brazero\cVweight{\vw \\ \vv}{z_1}\ket{K_{\vl}} \bra{K_{\vl}} \cVweight{\vv \\ \vu}{z_2} \ketzero}
{\braket{K_{\vl}|K_{\vl}}}\nonumber \\ 
&=\sum_{\vl}\left( \frac{e_N(\vu)z_2}{e_N(\vv)z_1 }\right)^{|\vl|}
\prod_{i,j=1}^N 
\frac{N_{\emptyset, \lambda^{(j)}} (qw_i/tv_j) N_{\lambda^{(i)}, \emptyset} (qv_i/tu_j)}
{N_{\lambda^{(i)}, \lambda^{(j)}} (qv_i/tv_j)}. 
\end{align}

In what follows in this section, we prove Theorem \ref{thm: matrix elements of V}.

\subsection{Transformation formula}\label{sec_4.2}
\label{sec: trf foumula}

First of all, 
we show a certain transformation formula 
(Proposition \ref{prop: trf formula}) 
for multiple series with respect to Macdonald functions $p_n(x;s|q,t)$, 
which plays an essential role in our proof 
of the main theorem. 
That formula is based on the Euler transformation 
of Kajihara and Noumi's multiple basic hypergeometric series, 
which is defined as follows.

\begin{definition}\label{def_phi_m_n}
Define 
\begin{align}
\Mphiu{m,n}
{a_1,\ldots,a_m\\ x_1,\ldots,x_m}
{b_1,\ldots,b_n\\ c_1,\ldots,c_n}{u} =
\sum_{\mu \in \mathbb{Z}_{\geq 0}^m} 
u^{\sum_{i=1}^m \mu_i}
\Mphi{m,n}{\mu}
{a_1,\ldots,a_m\\ x_1,\ldots,x_m}
{b_1y_1,\ldots,b_ny_n\\ c_1y_1,\ldots,c_ny_n}
\end{align}
with 
\begin{align}
&\Mphi{m,n}{\mu}
{a_1,\ldots,a_m\\ x_1,\ldots,x_m}
{b_1,\ldots,b_n\\ c_1,\ldots,c_n} =
\prod_{i<j}{}
\frac{q^{\mu_i}x_i-q^{\mu_j}x_j}{x_i-x_j}\ 
\prod_{i,j}{}
\frac{(a_jx_i/x_j;q)_{\mu_i}}
{(qx_i/x_j;q)_{\mu_i}}\ 
\prod_{i,k}{}
\frac{(b_kx_i;q)_{\mu_i}}
{(c_kx_i;q)_{\mu_i}}. 
\end{align}
\end{definition}

Let us also mention that the elliptic analogue is studied in \cite{KN}. 
Kajihara and Noumi gave the following $q$-Euler transformation formula. 

\begin{fact}[\cite{K,KN}] \label{fact: Kajihara Noumi}
\begin{align}
&
\Mphiu{m,n}
{a_1,\ldots,a_m\\ x_1,\ldots,x_m}
{b_1y_1,\ldots,b_ny_n\\ c_1y_1,\ldots,c_ny_n}{u}  \nonumber \\
&=\frac{(a_1\cdots a_m b_1\cdots b_n u / c^n; q)_\infty}
{(u; q)_\infty} \\ 
& \quad \times 
\Mphiu{n,m}
{c/b_1&,\ldots,&c/b_n\\ y_1&,\ldots,&y_n}
{cx_1/a_1,&\ldots,& cx_m/a_m\\c x_1,&\ldots,&c x_m}
{a_1\cdots a_m b_1\cdots b_n u / c^n}. \nonumber
\end{align}
\end{fact}

We prepare the following notation. 

\begin{definition}\label{Def_N_mu}
For nonnegative integers $n$, $m$ and 
$\mu=(\mu_i)_{1\leq i \leq m} \in \mathbb{Z}^m$, 
we introduce 
\begin{align}
\mathsf{N}^{n,m}_{\mu}(s_1,\ldots,s_{n+m}) :=
\prod_{k=1}^m\left( \prod_{i=1}^{n+k} \frac{(qs_{n+k}/ts_i;q)_{\mu_k}}{(qs_{n+k}/s_i;q)_{\mu_k}} \right)\cdot 
\prod_{1\leq i <j \leq m} \frac{(t\, q^{-\mu_i}s_{n+j}/s_{n+i};q)_{\mu_j}}{(q^{-\mu_i}s_{n+j}/s_{n+i};q)_{\mu_j}}. 
\end{align}
\end{definition}

By using Fact \ref{fact: Kajihara Noumi}, 
we can show the following formula 
which transforms the Macdonald function $p_{n+m}(\vx;\vs|q,t)$ 
to another series that contains $p_m$ as inner summation. 

\begin{proposition}\label{prop: trf formula}
Let $s_i$ ($i=1,\ldots, n+m$) be generic complex parameters and $|t|>|q|^{-(n-2)}$. 
We put $x_i = x t^{n-i}$ for $i=1,\ldots, n$ and $x_{n+k}=y_k$ for $k=1,\ldots, m$. 
Then 
\begin{align}\label{eq: trf formula}
&\prod_{k=1}^m \frac{(qy_k/t^nx;q)_{\infty}}{(ty_k/x;q)_{\infty}} \cdot
p_{n+m}(x_1,\ldots, x_{n+m}; s_1,\ldots,s_{n+m}|q,t)  \nonumber \\
&=\prod_{i=1}^n \frac{(q/t;q)_{\infty}}{(q/t^i;q)_{\infty}} \cdot
\prod_{1\leq i<j\leq n} \frac{(qs_j/ts_i;q)_{\infty}}{(qs_j/s_i;q)_{\infty}} \\
&\quad  \times \sum_{\mu\in \mathbb{Z}_{\geq 0}^m }
\mathsf{N}^{n,m}_{\mu}(s_1,\ldots,s_{n+m}) \,
p_m(y_1,\ldots, y_{m};q^{\mu_1} s_{n+1},\ldots,q^{\mu_m} s_{n+m} | q,t) 
\prod_{k=1}^m (ty_k/x)^{\mu_k}. \nonumber
\end{align}
\end{proposition}

\begin{remark}
Throughout this subsection, 
$s_i$ are treated as generic parameters 
with $s_i\neq 0$ and $s_i \neq q^{r_1} t^{r_2}s_j$ ($r_1\in \mathbb{Z}$, $r_2=0, \pm1$, $\forall i,j$). 
The above transformation formula holds 
even if $s_i$ are not specialized as (\ref{eq: form of s 1}) and (\ref{eq: form of s 2}). 
\end{remark}

In particular, if $m=0$, 
this proposition is nothing but the following specialization formula.

\begin{fact}[\cite{NS}]\label{fact:princ. spec.}
Let $|t|>|q|^{-(n-2)}$. Then 
\begin{align}
p_n(\vx;\vs|q,t) \Big|_{x_i \rightarrow t^{n-i}} = \prod_{i=1}^n \frac{(q/t;q)_{\infty}}{(q/t^i;q)_{\infty}}\cdot 
\prod_{1\leq i<\leq <n} \frac{(qs_j/ts_i;q)_{\infty}}{(qs_j/s_i;q)_{\infty}}.
\end{align}
\end{fact}

For the proof of Proposition \ref{prop: trf formula}, 
we prepare two lemmas. 
They are tools for attributing our transformation formula (\ref{eq: trf formula}) 
to the one of the multiple basic hypergeometric series. 
The proofs of these lemmas are by direct calculation of factorials (See Appendix \ref{sec: trf formula lem1} and \ref{sec: trf formula lem2}).

\begin{lemma}\label{lem: trf formula 1}
Let $\sigma=(\sigma_k)_{1\leq k \leq m-1} \in \mathbb{Z}^{m-1}$ 
and $\theta =(\theta_{i})_{1\leq i \leq n+m-1}\in \mathbb{Z}_{\geq 0}^{n+m-1}$. 
Then under the transformation $\rho_k=\sigma_k -\theta_{n+k}$, 
we have 
\begin{align}\label{eq: trf. rho=sigma-theta}
&\prod_{1\leq i<j\leq n} 
\frac{(q\ts{j}/t\ts{i};q)_{\infty}}{(q\ts{j}/\ts{i};q)_{\infty}} \cdot
d_{n+m}(\theta;s|q,t) 
\mathsf{N}^{n,m-1}_{\sigma}(\ts{1},\ldots,\ts{n+m-1}) \nonumber \\
&= \prod_{1\leq i<j\leq n} 
\frac{(q s_j/t s_i;q)_{\infty}}{(q s_j/ s_i;q)_{\infty}} \cdot 
\lim_{h \rightarrow 1}  \widetilde{\mathsf{N}}^{n,m-1}_{\rho}(h; s_1,\ldots, s_{n+m-1}) \nonumber\\
&\times \Mphi{n+m-1,m}{\theta}
{t&, \ldots, &t \\ h s_1^{-1} &, \ldots ,& h s_{n+m-1}^{-1}}
{q q^{\rho_1} s_{n+1}/t &,\ldots ,&q q^{\rho_{m-1}} s_{n+m-1}/t &, ts_{n+m} \\ 
q q^{\rho_1} s_{n+1}&, \ldots ,&q q^{\rho_{m-1}} s_{n+m-1} &, qs_{n+m}} \nonumber \\ 
&\times \prod_{i=1}^n q^{\theta_i} t^{-i\theta_i} 
\prod_{i=n+1}^{m-1} q^{\theta_i} t^{-n\theta_i}.
\end{align}
Here, $h$ is a parameter which goes to $1$ in the limit, 
satisfying $q^{r_1}t^{r_2}h\neq 1$ ($\forall r_1, r_2 \in \mathbb{Z}$). 
$\widetilde{\mathsf{N}}^{n,m}_{\mu}$ is defined by 
\begin{align}
\widetilde{\mathsf{N}}^{n,m}_{\mu}(h;s_1,\ldots,s_{n+m}) := 
\prod_{k=1}^m\left( \prod_{i=1}^{n+k} \frac{(qs_{n+k}/ts_i;q)_{\mu_k}}{(h qs_{n+k}/s_i;q)_{\mu_k}} \right)\cdot 
\prod_{1\leq i <j \leq m} \frac{(t\, q^{-\mu_i}s_{n+j}/s_{n+i};q)_{\mu_j}}{(q^{-\mu_i}s_{n+j}/s_{n+i};q)_{\mu_j}}. 
\end{align}
\end{lemma}

\begin{remark}
If $\rho_{k}<0$ and $\rho_k +\theta_{n+k}>0$ for some $k$, 
then $\widetilde{\mathsf{N}}^{n,m-1}_{\rho}$ is $0$ 
and $ \phi^{n+m-1,m}_{\theta}$ diverges in the limit $h \rightarrow 1$. 
However, (\ref{eq: trf. rho=sigma-theta}) converges as a result. 
In order to avoid the problem with respect to this divergence, 
we inserted the parameter $h$. 
\end{remark}

\begin{lemma}\label{lem: trf formula 2}
Let $\rho =(\rho_k)_{1\leq k \leq m-1} \in \mathbb{Z}^{m-1}$ and 
$\nu=(\nu_k)_{1\leq k \leq m} \in \mathbb{Z}_{\geq 0}^{m}$. 
Then 
\begin{align}
&\lim_{h\rightarrow 1} \widetilde{\mathsf{N}}^{n,m-1}_{\rho}(h; s_1,\ldots, s_{n+m-1}) \nonumber \\ 
&\times \Mphi{m,n+m-1}{\nu}
{t&, \ldots, &t , &q/t\\ q^{\rho_1} s_{n+1}& ,\ldots, &q^{\rho_{m-1}} s_{n+m-1}, &s_{n+m}}
{h q/ts_1 &, \ldots ,& h q/ts_{n+m-1} \\ h q/s_1 &, \ldots ,& h q/s_{n+m-1}} \nonumber \\ 
&= \mathsf{N}^{n,m}_{\mu}(s_1,\ldots, s_{n+m}) 
\times d_m((\theta_i); (q^{\mu_i}s_{n+i})|q,t) 
\end{align}
under the transformation 
\begin{align}
&\rho_k=\mu_k-\theta_k \quad (k=1,\ldots, m-1), \\
&\nu_k=\theta_k \quad (k=1,\ldots, m-1), \\ 
&\nu_{m}=\mu_m.
\end{align}
\end{lemma}

Now, we prove Proposition \ref{prop: trf formula} 
by using these lemmas and the $q$-Euler transformation formula. 

\vspace{1em}

\noindent
\textit{Proof of Proposition \ref{prop: trf formula}}
The proof is done by induction on $m$. 
If $m=0$, (\ref{eq: trf formula}) follows from Fact \ref{fact:princ. spec.}. 
By assumption that it holds for $m-1$, 
it can be shown that the left hand side of (\ref{eq: trf formula}) is 
\begin{align}\label{eq: pf of trf formula 1}
&\sum_{\theta \in \mathbb{Z}_{\geq 0}^{n+m-1}}
\frac{(qy_m/t^nx;q)_{\infty}}{(ty_m/x;q)_{\infty}} 
d_{n+m}(\theta,\vs|q,t)
\prod_{i=1}^{n+m-1}(x_{n+m}/x_i)^{\theta_i}\nonumber \\ 
&\qquad \qquad \times \prod_{k=1}^{m-1} \frac{(qy_k/t^nx;q)_{\infty}}{(ty_k/x;q)_{\infty}} \cdot
p_{n+m-1}(x_1,\ldots,x_{n+m-1};\ts{1},\ldots,\ts{n+m-1}|q,t)\nonumber \\ 
&=
\sum_{\substack{\theta \in \mathbb{Z}_{\geq 0}^{n+m-1}\\ \sigma \in \mathbb{Z}_{\geq 0}^{m-1}}}
\prod_{i=1}^n \frac{(q/t;q)_{\infty}}{(q/t^i;q)_{\infty}} \cdot
\frac{(qy_m/t^nx;q)_{\infty}}{(ty_m/x;q)_{\infty}}
\prod_{1\leq i<j\leq n} \frac{(q\ts{j}/t\ts{i};q)_{\infty}}{(q\ts{j}/\ts{i};q)_{\infty}} \nonumber \\
& \qquad  \qquad \times 
d_{n+m}(\theta,\vs|q,t)
\mathsf{N}^{n,m-1}_{\sigma}(\ts{1},\ldots,\ts{n+m}) 
p_{m-1}((y_i);(q^{\sigma_i} \ts{n+i}) | q,t) \nonumber \\ 
& \qquad \qquad \times \prod_{k=1}^{m-1} (ty_k/x)^{\sigma_k}. 
\end{align}
By the contribution of the factor $\prod_{k=1}^{m-1} 1/(q;q)_{\sigma_k}$ 
in $\mathsf{N}^{n,m-1}_{\sigma}$, 
we can extend the range which $\theta$ and $\sigma$ ran over 
to 
\begin{equation}
\theta\in \mathbb{Z}_{\geq 0}^{n+m-1},\quad  \sigma \in \mathbb{Z}^{m-1}. 
\end{equation}
Under this range, by applying Lemma \ref{lem: trf formula 1}, 
we can rewrite (\ref{eq: pf of trf formula 1}) as 
\begin{align}\label{eq: pf of trf formula 2}
&\lim_{h\rightarrow 1}
\sum_{\substack{\theta\in \mathbb{Z}_{\geq 0}^{n+m-1}\\ \rho \in \mathbb{Z}^{m-1}}}
\prod_{i=1}^n \frac{(q/t;q)_{\infty}}{(q/t^i;q)_{\infty}} \cdot
\frac{(qy_m/t^nx;q)_{\infty}}{(ty_m/x;q)_{\infty}}
\prod_{1\leq i<j\leq n} \frac{(qs_j/t s_i;q)_{\infty}}{(q s_j/ s_i;q)_{\infty}} \nonumber\\ 
&\times \widetilde{\mathsf{N}}^{n,m-1}_{\rho}(h; s_1,\ldots, s_{n+m-1}) 
p_{m-1}(y;(q^{\rho_i} s_{n+i}) | q,t)
\prod_{k=1}^{m-1} (ty_k/x)^{\rho_k} \nonumber \\
&\times \Mphi{n+m-1,m}{\theta}
{t&, \ldots, &t \\ h s_1^{-1} &, \ldots ,& h s_{n+m-1}^{-1}}
{q q^{\rho_1} s_{n+1}/t &,\ldots ,&q q^{\rho_{m-1}} s_{n+m-1}/t &, ts_{n+m} \\ 
	q q^{\rho_1} s_{n+1}&, \ldots ,&q q^{\rho_{m-1}} s_{n+m-1} &, qs_{n+m}} \nonumber \\ 
&\times \prod_{i=1}^{n+m-1}(qy_m/t^nx)^{\theta_i}. 
\end{align}
By using the Euler transformation formula for the multiple hypergeometric series 
(Fact \ref{fact: Kajihara Noumi}), 
(\ref{eq: pf of trf formula 2}) can be written as 
\begin{align}\label{eq:  pf of trf formula 3}
&\lim_{h\rightarrow 1}
\sum_{\substack{\nu\in \mathbb{Z}_{\geq 0}^{m}\\ \rho \in \mathbb{Z}^{m-1}}}
\prod_{i=1}^n \frac{(q/t;q)_{\infty}}{(q/t^i;q)_{\infty}} \cdot
\prod_{1\leq i<j\leq n} \frac{(qs_j/t s_i;q)_{\infty}}{(q s_j/ s_i;q)_{\infty}} \nonumber\\ 
&\times \widetilde{\mathsf{N}}^{n,m-1}_{\rho}(h; s_1,\ldots, s_{n+m-1}) 
p_{m-1}(y;(q^{\rho_i} s_{n+i}) | q,t)
\prod_{k=1}^{m-1} (ty_k/x)^{\rho_k} \nonumber \\
&\times \Mphi{m,n+m-1}{\nu}
{t&, \ldots, &t , &q/t\\ q^{\rho_1} s_{n+1}& ,\ldots, &q^{\rho_{m-1}} s_{n+m-1}, &s_{n+m}}
{h q/ts_1 &, \ldots ,& h q/ts_{n+m-1} \\ h q/s_1 &, \ldots ,& h q/s_{n+m-1}} \nonumber \\ 
&\times \prod_{k=1}^m(ty_m/x)^{\nu_k}.
\end{align}
Finally, Lemma \ref{lem: trf formula 2} shows 
that (\ref{eq:  pf of trf formula 3}) is equal to 
\begin{align}
&\prod_{i=1}^n \frac{(q/t;q)_{\infty}}{(q/t^i;q)_{\infty}} \cdot
\prod_{1\leq i<j\leq n} \frac{(qs_j/ts_i;q)_{\infty}}{(qs_j/s_i;q)_{\infty}} \label{eq:  pf of trf formula 4}\\
&\quad  \times 
\sum_{\substack{(\mu_i)_{i=1}^{m-1} \in \mathbb{Z}^{m-1}  \\ \mu_m \geq 0}}
\mathsf{N}^{n,m}_{\mu}(s_1,\ldots,s_{n+m}) \,
p_m(y_1,\ldots, y_{m};q^{\mu_1} s_{n+1},\ldots,q^{\mu_m} s_{n+m} | q,t) 
\prod_{k=1}^m (ty_k/x)^{\mu_k}. \nonumber
\end{align}
Noting that the summation is restricted to $\mu\in \mathbb{Z}_{\geq 0}^m$ 
by the contribution of the factor in $\mathsf{N}^{n,m}_{\mu}$, 
we can show that 
(\ref{eq:  pf of trf formula 4}) coincides with 
the right hand side of (\ref{eq: trf formula}). 
\qed

\subsection{Proof of Theorem \ref{thm: matrix elements of V}} \label{sec_4.3}

We are now in a position to show Theorem \ref{thm: matrix elements of V}. 
This subsection is devoted to its proof.

Since the defining relation of $\cV(w)$ and  the expansion coefficients of $\ket{K_{\vl}}$ in the basis of $\ket{X_{\vl}}$ are 
written by rational functions of $q$, $t$, $u_i$ and $v_i$, 
the matrix elements 
$\bra{K_{\boldsymbol{\lambda}}}\cV (z)\ket{K_{\boldsymbol{\mu}}}$ 
determine the unique rational function of them by Proposition \ref{prop: uniqueness of V}. 
Therefore, 
it suffices to prove (\ref{eq: matrix element}) in the case $v_i=t^{n_i}u_i$ 
($i=1,\ldots, N$) 
for all sufficiently large $n_i \in \mathbb{Z}$, i.e., 
$n_i\geq \ell(\lambda^{(i)})$ by analytic continuation.

\subsubsection*{Step1: Realization of $\cV(x)$} 

Firstly, 
we give a realization of $\cV(x)$ in the case $\vv=t^{\vnn}\cdot \vu$ 
as follows.

\begin{definition}\label{def_Vspecial}
Let $|t|>|q|^{-(n-2)}$. 
Define $\widetilde V^{(\vnn)}(x)=\tVweight{(\vnn)}{\vv \\ \vu}{x} 
: \mathcal{F}_{\vu} \rightarrow \mathcal{F}_{\vv}$ 
with $\vu=t^{-\vnn}\cdot \vv$ 
by 
\begin{align}
\widetilde V^{(\vnn)}(x)=
\lim_{x_i \rightarrow t^{|\vnn|-i} x}
\prod_{1\leq i< j\leq |\vnn|} \frac{(tx_j/x_i;q)_{\infty}}{(qx_j/tx_i;q)_{\infty}} \cdot 
V^{(\vnn)} (x_1,\ldots ,x_{|\vnn|})A^{-1}_{(|\vnn|)}(x),
\end{align}
where
\begin{align}
A_{(r)}(x) = \exp\left( \sum_{n>0} \frac{(1-(q/t)^r)(1-t^{(1-r)n})t^{2r}}{n(1-q^n)(1-t^{-n})}
\sum_{i=1}^{N} \gamma^{(i-1)n}a^{(i)}_{n}x^{-n} \right). 
\end{align}
\end{definition}

\begin{remark}
Note that $\sum_{i=1}^{N} \gamma^{(i-1)n}a^{(i)}_{n}$ is the boson 
corresponding to the Cartan part $\Delta^{(N)}(\psi^+(z))$. 
\end{remark}

\begin{proposition}\label{prop: welldef tilde V}
$\widetilde V^{(\vnn)}(x)$ is well-defined on $\mathbb{C}^*$, i.e., its arbitrary matrix elements are holomorphic functions there.
\end{proposition}

Before the proof, we prepare the following fact. 
This fact tells us the duality of the Macdonald functions under exchanging $t$ and $q/t$.
\begin{fact}[\cite{NS}]\label{fact_t_q/t}
The formal series $p_n(\vx;\vs|q,t)$ with the leading coefficient 1 satisfies the symmetry relation
\begin{equation}
     p_n(\vx;\vs|q,t) = \prod_{1\leq i < j \leq n} \frac{(t x_j/x_i ; q)_\infty}{(qx_j/tx_i;q)_\infty} \cdot p_n(\vx;\vs|q,q/t)\,.
\end{equation}
\end{fact}

\noindent
\textit{Proof of Proposition \ref{prop: welldef tilde V}. }
We introduce the following product of currents,
\begin{equation}
    \widetilde{A}_{(r)}(x_1,\dots,x_r) := \prod_{k=1}^r \exp\left( \sum_{n>0} \frac{(1-(q/t)^r)t^{2r}}{n(1-q^n)}
\sum_{i=1}^{N} \gamma^{(i-1)n}a^{(i)}_{n}x^{-n}_k \right)\,.
\end{equation}
We have the following equality.
\begin{equation}\label{temp2}
\begin{split}
    &\prod_{1\leq i< j\leq |\vnn|} \frac{(tx_j/x_i;q)_{\infty}}{(qx_j/tx_i;q)_{\infty}} \cdot  \prod_{1\leq i< j\leq |\vmm|} \frac{(ty_j/y_i;q)_{\infty}}{(qy_j/ty_i;q)_{\infty}} \\
    &\times \vx^{-\vl}\bra{P_{\vl}}
V^{(\vnn)} (x_1,\ldots ,x_{|\vnn|})\widetilde{A}^{-1}_{(|\vnn|)}(x_1,\dots,x_{|\vnn|})
\Vweight{(\vmm)}{\vu \\ \boldsymbol{w}}{y_1,\ldots,y_{|\vmm|}}\ketzero  
\\
&=\prod_{1\leq i< j\leq |\vnn|+|\vmm|} \frac{(tx_j/x_i;q)_{\infty}}{(qx_j/tx_i;q)_{\infty}}\cdot \mathcal{R}_{\vl}(\vv) \,p_{|\vnn|+|\vmm|}((\vx,\vy);\vs|q,q/t)
= \mathcal{R}_{\vl}(\vv) \,p_{|\vnn|+|\vmm|}((\vx,\vy);\vs|q,t) \,.
\end{split}
\end{equation}
Here, for brevity of notation, we set $x_{|\vnn|+i} = y_i$, and used Fact \ref{fact_t_q/t} and Lemma \ref{lem: <P|VV|0>}.

Then, by the same argument as the proof of Theorem \ref{thm_Analytic V}, we can show the matrix elements of 
\begin{equation*}
    \prod_{1\leq i< j\leq |\vnn|} \frac{(tx_j/x_i;q)_{\infty}}{(qx_j/tx_i;q)_{\infty}}\cdot V^{(\vnn)} (x_1,\ldots ,x_{|\vnn|})\widetilde{A}^{-1}_{(|\vnn|)}(x_1,\dots,x_{|\vnn|})
\end{equation*}
are the holomorphic functions on $\pi^{-1}_{|\vnn|}(U^{|\vnn|}_z(\bar{r}_0))$ with $\bar{r}_0 = |q/t|^{\frac{|\vnn|-2}{|\vnn|-1}}$.
Note that in this case, $\boldsymbol y^{-\vm} f_{|\vmm|}(\boldsymbol y;\vs|q,t)$ is multiplied before the integration in $\boldsymbol y$. 
Under the assumption $|t|>|q|^{-(n-2)}$,  $|t^{-1}|<\bar{r}_0$, and thus we can safely take the limit $x_i \to t^{|\vnn|-i} x$. This limit ends with the operator $\widetilde{V}^{(\vnn)}$, and this completes the proof.
\qed

This operator $\widetilde V^{(\vnn)}(x)$ is a realization of $\mathcal{V}$ 
in the case $\vv= t^{\vnn}\cdot \vu$. 
This follows from the following relation which are the essentially same as (\ref{def_V}). 

\begin{proposition}\label{prop:realize V}
For $r= 1,\dots, N$, the $\tVweight{(\vnn)}{\vv \\ \vu}{x}$ satisfies
\begin{align}\label{eq: rel X tilV}
\left(1-t^{|\vnn|}\frac{x}{z} \right)X^{(r)}(z) \tVweight{(\vnn)}{\vv \\ \vu}{x}
=(q/t)^r\left(1-(t/q)^rt^{|\vnn|} \frac{x}{z}\right)\tVweight{(\vnn)}{\vv \\ \vu}{x}X^{(r)}(z).
\end{align}
\end{proposition}
\begin{proof}
By Lemma \ref{lem: rel of X and Phi}, 
we can get 
\begin{align}\label{eq: for pf XV=VX 1}
&X^{(r)}(z)\Vweight{(\vnn)}{\vv \\ \vu}{x_1,\ldots,x_{|\vnn|}}
= \prod_{k=1}^{|\vnn|} \frac{1-(q/t)^rz/tx_k}{1-z/tx_k} \cdot
V^{(\vnn)}(x_1,\ldots,x_{|\vnn|})X^{(r)}(z) \nonumber \\
&+(1-t^{-1}) \sum_{i=1}^N\sum_{k=1}^{n_i} 
v_{i}t^{1-k}\delta(tx_{\nik}/z )
\prod_{1\leq \ell < \nik} \frac{1-(q/t)^rx_{\nik}/x_{\ell}}{1-x_{\nik}/x_{\ell}} \cdot
\prod_{\nik < \ell \leq |\vnn|} \frac{1-tx_{\ell}/qx_{\nik}}{1-x_{\ell}/x_{\nik}}\nonumber \\
& \qquad \qquad \qquad \qquad \times 
\widetilde U^{(\vnn)}_{\nik}(x_1,\ldots , x_{|\vnn|}) \Psi^+(x_{\nik}),
\end{align}
where $\widetilde U^{(\vnn)}_{\nik}$ is the operator obtained by replacing $\Phi^{(i)}(x_{\nik})$ in 
$V^{(\vnn)}$ with $Y^{(r)}(x_{\nik}) \Phi^{(i)}(qx_{\nik})$, 
i.e., 
\begin{align}
\widetilde U^{(\vnn)}_{\nik}(x_1,\ldots , x_{|\vnn|}):= \Phi^{(0)}(x_1)\cdots \Phi^{(i)}(x_{\nik-1})
Y^{(r)}(x_{\nik})\Phi^{(i)}(q x_{\nik})\Phi^{(i)}(x_{\nik+1}) \cdots 
\Phi^{(N-1)}(x_{|\vnn|}). 
\end{align}
Because of the $q$-difference of the operator $\Phi^{(i)}(x_{\nik})$ in $\widetilde U^{(\vnn)}_{\nik}$, 
if $\nik \neq 1$, the product of $\widetilde U^{(\vnn)}_{\nik}$ and 
$\prod \frac{(tx_j/x_i;q)_{\infty}}{(qx_j/tx_i;q)_{\infty}}$ 
vanishes under the principal specialization, i.e., 
\begin{align}\label{eq: for pf XV=VX 2}
\lim_{x_i \rightarrow t^{|\vnn|-i} x}
\prod_{1\leq i< j\leq |\vnn|} \frac{(tx_j/x_i;q)_{\infty}}{(qx_j/tx_i;q)_{\infty}} \cdot
\widetilde U^{(\vnn)}_{\nik}(x_1,\ldots , x_{|\vnn|}) =0 \quad (\nik \neq 1). 
\end{align}
Furthermore, 
by the operator product (\ref{eq: A Lam OPE}), 
we have 
\begin{align}\label{eq: for pf XV=VX 3}
A_{(s)}(x)X^{(r)}(z)=\prod_{k=1}^{r-1}\frac{1-t^{-k}z/x}{1-t^{-k}(q/t)^rz/x} \cdot X^{(r)}(z)A_{(s)}(x). 
\end{align}
By (\ref{eq: for pf XV=VX 1}), (\ref{eq: for pf XV=VX 2}) and  (\ref{eq: for pf XV=VX 3}), 
it can be shown that 
\begin{align}
&X^{(r)}(z)\widetilde V^{(\vnn)}(x)=
\frac{1-(q/t)^rz/t^{|\vnn|}x}{1-z/t^{|\vnn|}x} 
\widetilde V^{(\vnn)}(x)X^{(r)}(z)\\
&\qquad +(1-t^{-1})
v_{1}t^{1-k}
\lim_{x_i \rightarrow t^{|\vnn|-i} x}
\delta(tx_{1}/z )
\prod_{1 < \ell \leq |\vnn|} \frac{1-tx_{\ell}/qx_{1}}{1-x_{\ell}/x_{1}}
\prod_{1\leq i< j\leq |\vnn|} \frac{(tx_j/x_i;q)_{\infty}}{(qx_j/tx_i;q)_{\infty}} \cdot
\widetilde U^{(\vnn)}_{1}\Psi^+(x_1). 
\end{align}
By multiplying the both hand sides by $\left(1-t^{-|\vnn|}\frac{z}{x} \right)$, we have
\begin{align}
\left(1-t^{-|\vnn|}\frac{z}{x} \right)X^{(r)}(z) \tVweight{(\vnn)}{\vv \\ \vu}{x}
=\left(1-(q/t)^rt^{-|\vnn|} \frac{z}{x}\right)  \tVweight{(\vnn)}{\vv \\ \vu}{x}X^{(r)}(z).
\end{align}
After some simple calculation, 
we obtain Proposition \ref{prop:realize V}. 
\end{proof}

\subsubsection*{Step2: Evaluation of the Matrix Elements of $\widetilde V^{(\vnn)}(x)$} 

Next, we evaluate the matrix elements of $\widetilde V^{(\vnn)}(x)$. 
Let 
$\vs':=(s'_i)_{1\leq i\leq |\vmm|}$ 
with 
\begin{align}
s'_{[i,k]_{\vmm}}:=q^{\mu^{(i)}_k}t^{1-n_i-k}v_i \quad (1\leq k \leq m_i, i=1,\ldots, N) 
\end{align} 
and 
$\vs=(s_i)_{1\leq i \leq |\vnn|+|\vmm|}$ 
be the same one given in (\ref{eq: form of s 1}), (\ref{eq: form of s 2}), i.e., 
\begin{align}
&s_{\lbl{i}{k}{n}}=q^{\lambda^{(i)}_k} t^{1-k}v_i 
\quad (1 \leq k \leq n_i, \, i=1,\ldots, N), \\
&s_{|\vnn|+[i,k]_{\vmm}}= t^{1- n_{i} -k}v_i 
\quad (1 \leq k \leq m_i, \, i=1,\ldots, N). 
\end{align}
Further, we write $x_{|\boldsymbol{n}|+i}=y_i$. 
By the explicit formula for $\ket{Q_{\vl}}$ (Theorem \ref{thm: matrix elements of V}) 
and Lemma \ref{lem: <P|VV|0>}, 
we have 
\begin{align}
\braket{P_{\vl}|\widetilde V^{(\vnn)}(x)|Q_{\vm}}&=
\frac{1}{\mathcal{R}_{\vm}^{\vmm}(\vu)}
\left[ 
\vy^{-\vm} f_{|\vmm|}(\vy;\vs'|q,q/t)
\bra {P_{\vl}}\widetilde V^{(\boldsymbol{n})}(x) V^{(\boldsymbol{m})}(y_1,\ldots,y_{|\boldsymbol{m}|})\ketzero
\right]_{\vy,1}\nonumber \\ 
&=\frac{\mathcal{R}_{\vl}^{\vnn}(\vv)}{\mathcal{R}_{\vm}^{\vmm}(\vu)}
\Bigg[ \lim_{\substack{ x_i \rightarrow t^{|\vnn|-i} x\\ (1\leq i \leq |\boldsymbol{n}|)}}
\vy^{-\vm}\vx^{\vl}  f_{|\vmm|}(\vy;\vs'|q,q/t)
\prod_{1\leq i<j\leq |\boldsymbol{m}|} \frac{(qy_j/ty_i;q)_{\infty}}{(ty_j/y_i;q)_{\infty}} \nonumber \\
& \qquad\qquad\qquad\qquad \times 
\prod_{k=1}^{|\boldsymbol{m}|} \frac{(qy_k/t^{|\boldsymbol{n}|}x;q)_{\infty}}{(ty_k/x;q)_{\infty}}  \cdot
p_{|\boldsymbol{n}|+|\boldsymbol{m}|}((x_i)_{1\leq i \leq |\boldsymbol{n}|+|\boldsymbol{m}|};\vs|q,t) 
\Bigg]_{\vy,1}. \label{eq:step2_1}
\end{align}
By virtue of the transformation formula of Proposition \ref{prop: trf formula}, 
the Macdonald function $p_{|\boldsymbol{n}|+|\boldsymbol{m}|}$ in (\ref{eq:step2_1}) 
can be transformed to the summation of $p_{|\boldsymbol{m}|}$ as 
\begin{align}
&\frac{\mathcal{R}_{\vl}^{\vnn}(\vv)}{\mathcal{R}_{\vm}^{\vmm}(\vu)} 
\prod_{i=1}^{|\boldsymbol{n}|} \frac{(q/t;q)_{\infty}}{(q/t^i;q)_{\infty}} \cdot 
\prod_{1\leq i<j\leq |\boldsymbol{n}|} \frac{(qs_j/ts_i;q)_{\infty}}{(qs_j/s_i;q)_{\infty}} \cdot
 \lim_{\substack{ x_i \rightarrow t^{|\vnn|-i} x\\ (1\leq i \leq |\boldsymbol{n}|)}}\vx^{\vl}  
\Bigg[ 
\vy^{-\vm} f_{|\vmm|}(\vy;\vs'|q,q/t)\nonumber \\ 
& \quad \times \prod_{1\leq i<j\leq |\boldsymbol{m}|} \frac{(qy_j/ty_i;q)_{\infty}}{(ty_j/y_i;q)_{\infty}} \cdot
\sum_{\nu \in \mathbb{Z}_{\geq 0}^{|\vmm|} }
\mathsf{N}^{|\boldsymbol{n}|,|\boldsymbol{m}|}_{\nu}(s_1,\ldots,s_{|\boldsymbol{n}|+|\boldsymbol{m}|}) \,
p_{|\boldsymbol{m}|}(\vy;(q^{\nu_i}s_{|\vnn|+i}) | q,t) 
\prod_{k=1}^{|\boldsymbol{m}|} (ty_k/x)^{\nu_k} \Bigg]_{\vy,1}\nonumber \\
&=
\frac{\mathcal{R}_{\vl}^{\vnn}(\vv)}{\mathcal{R}_{\vm}^{\vmm}(\vu)} 
\prod_{i=1}^{|\boldsymbol{n}|} \frac{(q/t;q)_{\infty}}{(q/t^i;q)_{\infty}} \cdot 
\prod_{1\leq i<j\leq |\boldsymbol{n}|} \frac{(qs_j/ts_i;q)_{\infty}}{(qs_j/s_i;q)_{\infty}} \cdot
 \lim_{\substack{ x_i \rightarrow t^{|\vnn|-i} x\\ (1\leq i \leq |\boldsymbol{n}|)}}\vx^{\vl} \\
&\quad \times 
\sum_{\nu \in \mathbb{Z}_{\geq 0}^{|\vmm|} }
\mathsf{N}^{|\boldsymbol{n}|,|\boldsymbol{m}|}_{\nu}(s_1,\ldots,s_{|\boldsymbol{n}|+|\boldsymbol{m}|})
\Bigg[ 
\vy^{-\vm} f_{|\vmm|}(\vy;\vs'|q,q/t)
p_{|\boldsymbol{m}|}(\vy;(q^{\nu_i}s_{|\vnn|+i}) | q,q/t) 
\prod_{k=1}^{|\boldsymbol{m}|} (ty_k/x)^{\nu_k} \Bigg]_{\vy,1}.\nonumber 
\end{align}
Here, we used Fact \ref{fact_t_q/t}. 
Note that since $s_{|\vnn| + [i,k]_{\vmm}}= t^{1- n_{i} -k}v_i$, 
if $\nu$ does not satisfy 
$\nu_{[i,1]_{\vmm}}\geq \nu_{[i,2]_{\vmm}} \geq \cdots \geq \nu_{[i,m_i]_{\vmm}}$, i.e., 
cannot be regarded as an $N$-tuple of partitions, 
then 
$\mathsf{N}^{|\boldsymbol{n}|,|\boldsymbol{m}|}_{\nu}=0$.  
By Lemma \ref {Lem_delta1} and 
\begin{align}
\vx^{\vl}\Big|_{x_i \rightarrow t^{|\vnn|-i} x}=x^{|\vl|}t^{(|\vnn|-1)|\vl|} 
\prod_{i=1}^N t^{-n(\lambda^{(i)})-|\lambda^{(i)}|\sum_{k=1}^{i-1} n_k},
\end{align}
we obtain 
\begin{align}
\braket{P_{\vl}|\widetilde V^{(\vnn)}(x) |Q_{\vm}}
=&x^{|\vl|-|\vm|}  t^{|\vm|+(|\vnn|-1)|\vl|} 
\prod_{i=1}^N t^{-n(\lambda^{(i)})-|\lambda^{(i)}|\sum_{k=1}^{i-1} n_k} \nonumber \\
&\times \frac{\mathcal{R}^{\vnn}_{\vl}(\vv)}{\mathcal{R}^{\vmm}_{\vm}(\vu)}
\prod_{i=1}^{|\boldsymbol{n}|} \frac{(q/t;q)_{\infty}}{(q/t^i;q)_{\infty}} \cdot
\prod_{1\leq i<j\leq |\boldsymbol{n}|} \frac{(qs_j/ts_i;q)_{\infty}}{(qs_j/s_i;q)_{\infty}}\cdot
\mathsf{N}^{|\boldsymbol{n}|,|\boldsymbol{m}|}_{[\vm]^{\vmm}}(s_1,\ldots,s_{|\boldsymbol{n}|+|\boldsymbol{m}|}). \label{eq:step2_2}
\end{align}
Here, we put 
\begin{equation}
[\vm]^{\vmm}=([\vm]^{\vmm}_{i})_{1\leq i\leq |\vmm|}
:=(\mo_1,\ldots,\mo_{m_1},\mt_1,\ldots,\mt_{m_2},\ldots,\mN_1,\ldots,\mN_{m_N}). 
\end{equation} 
Moreover, (\ref{eq:step2_2}) leads to 
\begin{align}
\frac{\braket{K_{\vl}|\widetilde V^{(\vnn)}(x) |K_{\vm}}}{\brazero \widetilde V^{(\vnn)}(x) \ketzero}
=&
\mathcal{C}_{\vl}^{(-)} \, \mathcal{C}^{(+)}_{\vm}
\prod_{i=1}^N \frac{c'_{\mu^{(i)}}}{c_{\mu^{(i)}}} \cdot 
x^{|\vl|-|\vm|} t^{|\vm|+(|\vnn|-1)|\vl|} 
\prod_{i=1}^N t^{-n(\lambda^{(i)})-|\lambda^{(i)}|\sum_{k=1}^{i-1} n_k} \nonumber
\\
&\times 
\frac{\mathcal{R}^{\vnn}_{\vl}(\vv)}{\mathcal{R}^{\vmm}_{\vm}(\vu)}
\prod_{1\leq i<j\leq|\boldsymbol{n}|}\frac{(q s_j/ts_i;q)_{[\boldsymbol{\lambda}]^{\vnn}_i-[\boldsymbol{\lambda}]^{\vnn}_j}}{(q s_j/s_i;q)_{[\boldsymbol{\lambda}]^{\vnn}_i-[\boldsymbol{\lambda}]^{\vnn}_j}}  \cdot 
\mathsf{N}^{|\vnn|,|\vmm|}_{[\vm]^{\vmm}}(s_1,\ldots,s_{|\vnn|+|\vmm|}). 
\label{eq: <K|til V|K>}
\end{align}
Note that 
\begin{align}
\prod_{1\leq i<j\leq|\boldsymbol{n}|}\frac{(q s_j/ts_i;q)_{[\boldsymbol{\lambda}]^{\vnn}_i-[\boldsymbol{\lambda}]^{\vnn}_j}}{(q s_j/s_i;q)_{[\boldsymbol{\lambda}]^{\vnn}_i-[\boldsymbol{\lambda}]^{\vnn}_j}}
=&\prod_{k=1}^{N} \prod_{1 \leq i<j \leq n_k}
\frac{(q^{\lambda^{(k)}_j-\lambda^{(k)}_i+1}t^{-j+i-1};q)_{-\lambda^{(k)}_j+\lambda^{(k)}_i}}
{(q^{\lambda^{(k)}_j-\lambda^{(k)}_i+1}t^{-j+i};q)_{-\lambda^{(k)}_j+\lambda^{(k)}_i}}  \\
& \times \prod_{1\leq k<l\leq N} \prod_{\substack{1 \leq i \leq n_k \\ 1 \leq j \leq n_l}} 
\frac{(q^{\lambda^{(l)}_j-\lambda^{(k)}_i+1}t^{-j+i-1}v_l/v_k;q)_{-\lambda^{(l)}_j+\lambda^{(k)}_i}}
{(q^{\lambda^{(l)}_j-\lambda^{(k)}_i+1}t^{-j+i}v_l/v_k;q)_{-\lambda^{(l)}_j+\lambda^{(k)}_i}}. \nonumber
\end{align}
Finally, it can be shown that the expression (\ref{eq: <K|til V|K>}) coincides with the Nekrasov factors
\begin{align}
\frac{\braket{K_{\vl}|\widetilde V^{(\vnn)}(x) |K_{\vm}}}{\brazero \widetilde V^{(\vnn)}(x) \ketzero} 
=&
\left((-1)^N  e_N(\boldsymbol{v}) x \right)^{|\vl|} \left((t/q)t^{|\vnn|} x\right)^{-|\vm|}
\prod_{i=1}^N v_i^{-(N-1)|\lambda^{(i)}|} 
((q/t) u_{i})^{|\mu^{(i)}|} 
g_{\lambda^{(i)}}^{-N+1}g_{\mu^{(i)}}\nonumber \\
&\times
\prod_{i=1}^N (q/t)^{(N-i)|\mu^{(i)}|-\sum_{k=1}^i|\mu^{(k)}|}
\prod_{ i, j=1 }^N N_{\lambda^{(i)}, \mu^{(j)}}(t^{n_j}v_i/v_j). \label{eq: final eq}
\end{align}
The coincidence between (\ref{eq: <K|til V|K>}) and (\ref{eq: final eq}) can be proved 
by induction on $\vl$ and $\vm$. 
We give some formulas to show this coincidence in Appendix \ref{App:Useful_Induction}. 
The difference between (\ref{def_V}) and (\ref{eq: rel X tilV}) 
can be modified by the transformation 
\begin{align}
x \rightarrow t^{|\vnn|}x, \quad u_i \rightarrow (q/t) u_i \quad (i=1,\ldots, N).
\end{align}
Noting that the renormalization constant (\ref{eq: ren. const. of K}) are also modified, 
we can see that the equation (\ref{eq: final eq}) shows Theorem \ref{thm: matrix elements of V} in the case $\vv= t^{\vnn} \cdot \vu$. 
Therefore, by analytic continuation, 
Theorem \ref{thm: matrix elements of V} holds for the general case.

\section{Refined Topological Vertex and Changing Preferred Direction}\label{Sec_Sdual}

In \cite{AFS}, the intertwiners among the $\cU$-modules are introduced, and their matrix elements are identical to the refined topological vertex of \cite{IKV}.
In this section, we compute the matrix elements of the ladder diagrams, which are obtained by gluing the intertwiners.
The result shows the invariance under changing the preferred directions of the diagrams, which is a natural consequence of the S-duality in the string theory. 

\subsection{Trivalent intertwiner and refined topological vertex}
We introduced the $\cF^{(1,M)}$-module in Sec \ref{subsec_F1M}.
In order to introduce the intertwiners, we need the other one, referred to as the $\cF^{(0,1)}_v$.
\begin{fact}[\cite{FT, FFJMM}]\label{Fact_vertical_rep}
Let $u$ be an indeterminate.
We can endow a $\cU$-module structure
to $\cF$ by setting
\begin{align}
&
c^{1/2} P_\lambda=P_\lambda,\\
&
x^+(z) P_\lambda=
\sum_{i=1}^{\ell(\lambda)+1} 
A^+_{\lambda,i}\,
\delta(q^{\lambda_i}t^{-i+1}u/z) 
P_{\lambda+{\bf 1}_i},\\
&
x^-(z) P_\lambda=q^{1/2}t^{-1/2}
\sum_{i=1}^{\ell(\lambda)} 
A^-_{\lambda,i}\,
\delta(q^{\lambda_i-1}t^{-i+1}u/z)
P_{\lambda-{\bf 1}_i},\\
&
\psi^+(z)P_\lambda=q^{1/2}t^{-1/2}
B^+_\lambda(u/z)P_\lambda,\\
&
\psi^-(z)P_\lambda=q^{-1/2}t^{1/2}
\,B^-_\lambda(z/u)P_\lambda\,,
\end{align}
with
\begin{align}
&A^+_{\lambda,i}=(1-t)
\prod_{j=1}^{i-1}
{(1-q^{\lambda_i-\lambda_j}t^{-i+j+1})
(1-q^{\lambda_i-\lambda_j+1}t^{-i+j-1})\over 
(1-q^{\lambda_i-\lambda_j}t^{-i+j})
(1-q^{\lambda_i-\lambda_j+1}t^{-i+j})},\label{A+}\\
&A^-_{\lambda,i}=(1-t^{-1})
{1-q^{\lambda_{i+1}-\lambda_i} \over 1-q^{\lambda_{i+1}-\lambda_i+1} t^{-1}}
\prod_{j=i+1}^{\infty}
{(1-q^{\lambda_j-\lambda_i+1}t^{-j+i-1})
(1-q^{\lambda_{j+1}-\lambda_i}t^{-j+i})\over 
(1-q^{\lambda_{j+1}-\lambda_i+1}t^{-j+i-1})
(1-q^{\lambda_j-\lambda_i}t^{-j+i})} ,\label{A-}\\
&
B^+_\lambda(z)=
{1-q^{\lambda_{1}-1} t z \over 1-q^{\lambda_{1}} z}
\prod_{i=1}^{\infty}
{(1-q^{\lambda_i}t^{-i}z)
(1-q^{\lambda_{i+1}-1}t^{-i+1}z)\over 
(1-q^{\lambda_{i+1}}t^{-i}z)
(1-q^{\lambda_i-1}t^{-i+1}z)},\label{B+}\\
&
B^-_\lambda(z)=
{1-q^{-\lambda_{1}+1} t^{-1} z \over 1-q^{-\lambda_{1}} z}
\prod_{i=1}^{\infty}
{(1-q^{-\lambda_i}t^{i}z)
(1-q^{-\lambda_{i+1}+1}t^{i-1}z)\over 
(1-q^{-\lambda_{i+1}}t^{i}z)
(1-q^{-\lambda_i+1}t^{i-1}z)}.\label{B-}
\end{align}
We denote this module as $\cF^{(0,1)}$-module.
\end{fact}

Then, we introduce the intertwiners among the modules defined above. 
\begin{fact}[\cite{AFS}]\label{Fact_intertwiner}
When $w = -uv$, 
there exists a unique intertwiner which satisfies
    \begin{equation}
        \Phi\left[  (1,M+1),w \atop (0,1),v; (1,M) ,u\right]:\cF^{(0,1)}_v\otimes \cF^{(1,M)}_u\longrightarrow
\cF^{(1,M+1)}_{w}\,;\qquad a \Phi= \Phi \Delta(a) \qquad (\forall a\in \cU)
    \end{equation}
and the normalization condition $\bra{0}\Phi(1 \otimes \ket{0})=1$. Moreover, its component $\Phi_\lambda$, defined by
\begin{equation}
\Phi_\lambda (\alpha)=\Phi(P_\lambda \otimes \alpha )
\qquad (\forall P_\lambda \otimes \alpha \in \cF^{(0,1)}_v\otimes \cF^{(1,M)}_u)\,,
\end{equation}
has the following realization,
\begin{equation}
\Phi_\lambda\left[  (1,M+1),-vu \atop (0,1),v; (1,M) ,u\right]  = \hat{t}(\lambda,u,v,M) \widehat{\Phi}_\lambda(v)\,,
\end{equation}
where
\begin{equation}
\begin{split}
&\hat t(\lambda,u,v,M)=(-vu)^{|\lambda|} (-v)^{-(M+1)|\lambda|}f_\lambda^{-M-1}q^{n(\lambda')}/ c_\lambda\,,\\
&\widehat{\Phi}_\lambda(v) = :\Phi_{\emptyset}(v) \eta_\lambda(v):\,,\\
&\Phi_{\emptyset}(v) =
 \exp \Bigl(
 -\sum_{n=1}^{\infty} \dfrac{1}{n}\dfrac{1}{1-q^n} a_{-n}v^n 
      \Bigr)
 \exp\Bigl(
-  \sum_{n=1}^{\infty} \dfrac{1}{n}\dfrac{q^n}{1-q^{n}} a_{n}v^{-n}
     \Bigr),\\
&\eta_\lambda(v)=\,
:\prod_{i=1}^{\ell(\lambda)}\prod_{j=1}^{\lambda_i}
\eta(q^{j-1}t^{-i+1} v):\,.
\end{split}    
\end{equation}
Similarly, the following intertwiner exists uniquely,
\begin{equation}
\Phi^*\left[  (1,M) ,u;(0,1),v \atop (1,M+1),-vu\right]: \cF^{(1,M+1)}_{-uv}\longrightarrow
\cF^{(1,M)}_u\otimes \cF^{(0,1)}_v,\qquad
\Delta(a) \Phi^*= \Phi^* a \qquad (\forall a\in \cU)\,,
\end{equation}
with normalization $\Phi^\ast(\ket{0}) = \ket{0}\otimes 1+ \cdots$, and its component, defined by
\begin{equation}
\Phi^*( \alpha )=\sum_\lambda \Phi^*_\lambda (\alpha)\otimes Q_\lambda
\qquad (\forall \alpha \in \cF^{(1,M+1)}_{-uv})\,,
\end{equation}
 is realized by
\begin{equation}
\Phi^*_\lambda\left[  (1,M) ,v;(0,1),u \atop (1,M+1),-vu\right] = \hat{t}^*(\lambda,u,v,M) \wPhi^*_\lambda(u) \,,
\end{equation}
where
\begin{equation}
     \begin{split}
         &\hat{t}^*(\lambda,u,v,M)=(q^{-1} v)^{-|\lambda|} (-u)^{M|\lambda|}f_\lambda^{M}q^{n(\lambda')}/ c_\lambda\,,\\
         &\wPhi^*_\lambda (u) =:\Phi^*_{\emptyset}(u) \xi_\lambda(u): \,,\\
         &\Phi^*_{\emptyset}(u) =\exp \Bigl(\sum_{n=1}^{\infty} \dfrac{1}{n}\dfrac{1}{1-q^n} q^{-n/2}t^{n/2}a_{-n}u^n 
      \Bigr)
 \exp\Bigl(
  \sum_{n=1}^{\infty} \dfrac{1}{n}\dfrac{q^n}{1-q^{n}} q^{-n/2}t^{n/2}a_{n}u^{-n}
     \Bigr)\,,\\
     &\xi_\lambda(u)=\,
:\prod_{i=1}^{\ell(\lambda)}\prod_{j=1}^{\lambda_i}\xi(q^{j-1}t^{-i+1} u):.
     \end{split}
 \end{equation}
\end{fact}

\begin{notation}
In what follows, we mainly consider the $M=0$ case, and we introduce the simplified notations for the intertwiners.
\begin{equation}
            \Phi[u,v] := \Phi\left[  (1,1),-vu \atop (0,1),v; (1,0) ,u\right]\,,\qquad\Phi^*[u,v]: = \Phi^*\left[  (1,0) ,u;(0,1),v \atop (1,1),-vu\right]\,,
\end{equation}
and their components,
\begin{equation}
        \Phi_\lambda[u,v] := \Phi_\lambda\left[  (1,1),-vu \atop (0,1),v; (1,0) ,u\right]\,,\qquad\Phi_\lambda^*[u,v]: = \Phi^*_\lambda\left[  (1,0) ,u;(0,1),v \atop (1,1),-vu\right].
\end{equation}

We also assign the trivalent diagrams to each intertwiner as follows. The arrows stand for the $\cF^{(0,1)}$-modules, and we refer to this direction as the preferred direction, following the terminology of the refined topological vertex in \cite{IKV}.
\begin{center}
	\begin{tikzpicture}[scale = 1]
		\draw[MRarrow] (0,0)--(0,1);
		\draw (0,0)--(1,0);
		\draw (0,0)--(-0.71,-0.71);
		\node at (0,-1.7) {$\Phi^*[u,v]$};
		\node[right] at (1,0) {$u$};
		\node[above] at (0,1) {$v$};
		\node[left, below] at (-0.71,-0.71) {$-uv$};
	\end{tikzpicture}
	\hspace{1cm}
	\begin{tikzpicture}[scale = 1]
		\draw[MRarrow] (0,-1)--(0,0);
		\draw (0,0)--(-1,0);
		\draw (0,0)--(0.71,0.71);
		\node[left] at (-1,0) {$u$};
		\node[below] at (0,-1) {$v$};
		\node[right, above] at (0.71,0.71) {$-uv$};
		\node at (0,-1.7) {$\Phi[u,v]$};
	\end{tikzpicture}
\end{center}
\end{notation}

These intertwiners can be identified with the refined topological vertex, invented in \cite{IKV}.
To see this, we define the refined topological vertex.
\begin{definition}
The refined topological vertex is defined by
\begin{align}
C^{{\rm (IKV)}}_{\lambda\mu\nu}(t,q)=
\left(q\over t\right)^{||\mu||^2\over 2} t^{\kappa(\mu)\over 2} q^{||\nu||^2\over 2}
{1\over c_\lambda}
\sum_\eta 
\left(q\over t\right)^{|\eta|+|\lambda|-|\mu|\over 2} 
s_{\lambda'/\eta}(t^{-\rho}q^{-\nu})s_{\mu/\eta}(t^{-\nu'}q^{-\rho}),
\end{align}
where $c_\lambda$ is defined in (\ref{eq: c and c'}), $||\lambda||^2=\sum_i\lambda_i^2$, $\rho = ( -1/2, -3/2, -5/2, \dots)$ and
$\kappa(\lambda)=\sum_i \lambda_i(\lambda_i+1-2 i)$.
\end{definition}
The following fact shows the intertwiners can be regarded as the refined topological vertex.
\begin{fact}[\cite{AFS}]
\begin{align}
&{1\over \langle P_\lambda,P_\lambda\rangle_{q,t}}
\bra{S_\mu(q,t)}\Phi_\lambda\left[  (1,M+1),-vu \atop (0,1),v; (1,M) ,u\right] 
\ket{s_\nu}\\
&\qquad
=\left( q^{-1/2}u\over (-v)^{M} \right)^{|\lambda|}
f_\lambda^{-M}\cdot \,\,
(-q^{-1/2}v)^{-|\nu|} f_\nu \cdot 
(t^{-1/2}v)^{|\mu|} \cdot 
 (-1)^{|\mu|+|\nu|+|\lambda|}
\,\,C^{\rm(IKV)}_{\mu\nu' \lambda'}(q,t),\nonumber\\
&
\bra{S_\nu(q,t)}\Phi^*_\lambda\left[  (1,M) ,v;(0,1),u \atop (1,M+1),-vu\right] 
\ket{s_\mu}\\
&\qquad
=\left( (-u)^M\over q^{-1/2} v \right)^{|\lambda|}
f_\lambda^M \cdot \,\,
(-q^{-1/2}u)^{|\nu|} f_\nu^{-1} \cdot 
(t^{-1/2}u)^{-|\mu|} \cdot 
 \,C^{\rm (IKV)}_{\mu'\nu \lambda}(t,q).\nonumber
\end{align}
\end{fact}

\subsection{Changing preferred direction of ladder diagrams}
In this subsection, we will show the  matrix elements of the following ladder diagrams are identical up to some monomial factors.

\begin{figure}[H]
    \centering
    \includegraphics{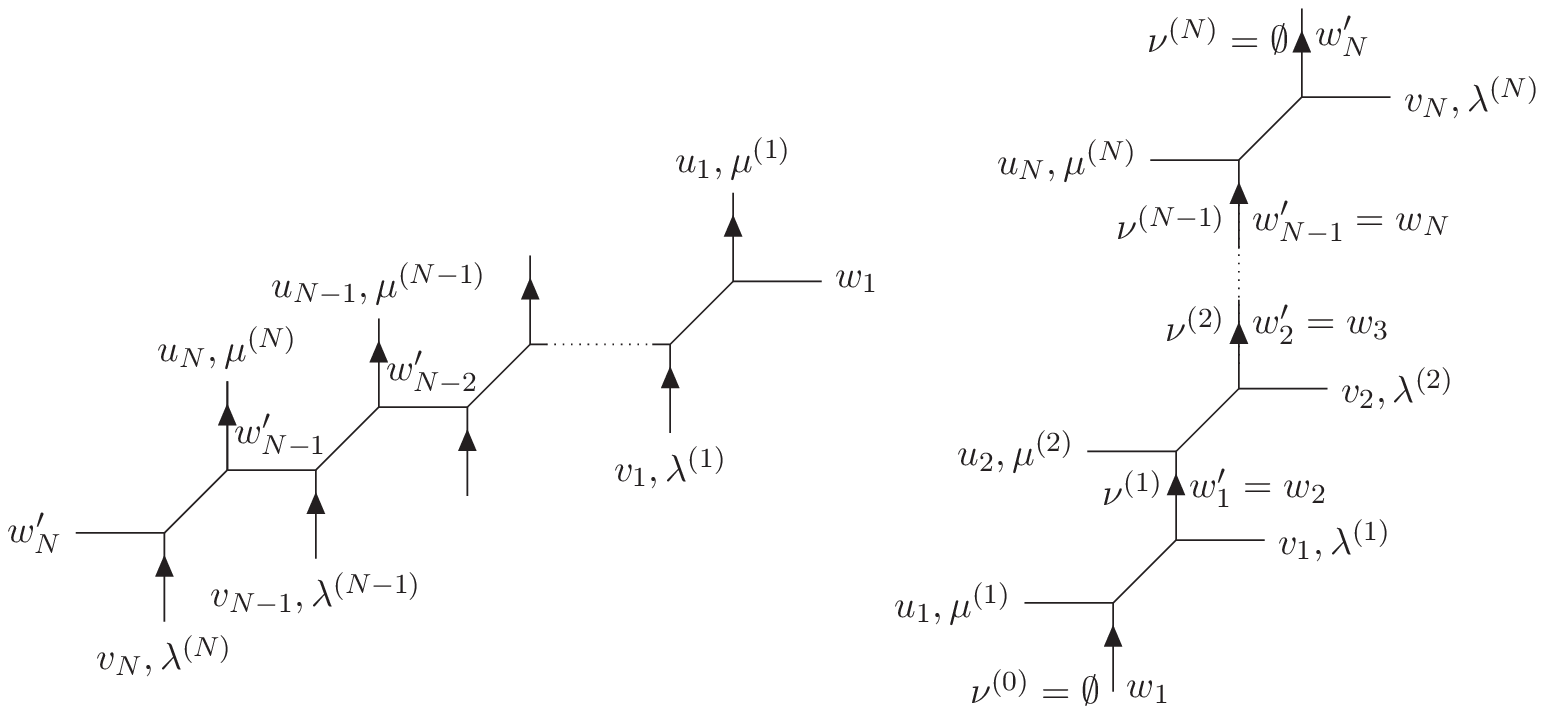}
	\captionof{figure}{Left:$\cTH_{\boldsymbol{\lambda},\boldsymbol{\mu}}(\boldsymbol{u}, \boldsymbol{v};w)$  and Right:$\cTV_{\boldsymbol{\lambda},\boldsymbol{\mu}}(\boldsymbol{u}, \boldsymbol{v};w)$ (up to the overall factor)}
	\label{fig_TH}
\end{figure}

\begin{notation}
We introduce some notations $w_i$ and $w'_i$ related to the spectral parameters $\boldsymbol{u}, \boldsymbol{v}$ of the modules.
\begin{equation}
    \begin{split}
         &w_1 = w, \\
         &w'_i = \frac{u_i}{v_i}w_i, \quad \text{for} \quad i=1,\dots,N,\\
         &w_{i+1} = w'_i, \quad \text{for} \quad i=1,\dots,N-1.
    \end{split}
\end{equation}
\end{notation}

\begin{notation}\label{not: aotimes}
For integers $n\leq m$, we write 
\begin{align}
\aotimes{i=n}{m}A_i:=A_{n} \otimes \cdots \otimes A_{m}.
\end{align}
\end{notation}

\begin{definition}\label{def_TV_TH}
Define the map $\cTH(\boldsymbol{u}, \boldsymbol{v};w) : \left(\aotimes{i=1}{N} \cF^{(0,1)}_{v_i}\right) \to \left(\aotimes{i=1}{N} \cF^{(0,1)}_{u_i}\right) $ by the following composition
\begin{align*}
   & \cF^{(0,1)}_{v_1}\otimes\cdots\otimes\cF^{(0,1)}_{v_{N-1}}\otimes\cF^{(0,1)}_{v_N}\otimes \ket{0}  \xrightarrow{{\rm id}\otimes \cdots\otimes \Phi} \cF^{(0,1)}_{v_1}\otimes\cdots\otimes\cF^{(0,1)}_{v_{N-1}}\otimes\cF^{(1,1)}_{-v_N w'_N}\\
    \xrightarrow{{\rm id}\otimes \cdots\otimes \Phi^*} &\cF^{(0,1)}_{v_1}\otimes\cdots\otimes\cF^{(0,1)}_{v_{N-1}}\otimes\cF^{(1,0)}_{w_N} \otimes\cF^{(0,1)}_{u_N} 
     \xrightarrow{{\rm id}\otimes \cdots\otimes \Phi\otimes {\rm id}}
     \cF^{(0,1)}_{v_1}\otimes\cdots\otimes\cF^{(1,1)}_{-v_{N-1}w'_{N-1}}\otimes\cF^{(0,1)}_{u_N}
     \\
     \xrightarrow{{\rm id}\otimes \cdots\otimes \Phi^*\otimes {\rm id}} 
    &\cdots\cdots
     \xrightarrow{\Phi^*\otimes \cdots\otimes {\rm id}} \ket{0}\otimes \cF^{(0,1)}_{u_1}\otimes\cdots\otimes\cF^{(0,1)}_{u_{N-1}}\otimes\cF^{(0,1)}_{u_N}\,.
\end{align*}
Here, 
$\cdots \otimes \ket{0}$ and 
$\ket{0}\otimes \cdots$ mean 
taking the vacuum expectation value at the level (1,0) modules. 
For simplicity, we introduce the following notation,
\begin{equation}
    \cTH(\boldsymbol{u}, \boldsymbol{v};w) = \bra{0} \Phi^*[w_1, u_1]\Phi[w'_1,  v_1]\Phi^*[w_2, u_2]\Phi[w'_2,v_2]\cdots\Phi^*[w_N, u_N]\Phi[w'_N,v_N]\ket{0} \,.
\end{equation}
We denote its matrix elements by 
\begin{equation}
\begin{split}
    \cTH_{\boldsymbol{\lambda},\boldsymbol{\mu}}(\boldsymbol{u}, \boldsymbol{v};w) = \bra{0} \Phi^*_{\mu^{(1)}}[w_1,  u_1]\Phi_{\lambda^{(1)}}[w'_1, v_1]&\Phi^*_{\mu^{(2)}}[w_2,  u_2]\Phi_{\lambda^{(2)}}[w'_2,v_2]\cdots
    \\
    &\cdots\Phi^*_{\mu^{(N)}}[w_N,  u_N]\Phi_{\lambda^{(N)}}[w'_N,v_N]\ket{0}\,. 
\end{split}
\end{equation}

Also, define the vertex operator $\cTV(\boldsymbol{u}, \boldsymbol{v};w) : \cF_{\boldsymbol{u}} \to \cF_{\boldsymbol{v}}$ by
\begin{equation}
    \cTV(\boldsymbol{u}, \boldsymbol{v};w) :=  T(\boldsymbol{u}, \boldsymbol{v};w)\big/\bra{\boldsymbol{0}}T(\boldsymbol{u}, \boldsymbol{v};w)\ket{\boldsymbol{0}}\,,
\end{equation}
with
\begin{equation}
\begin{split}
&T(\boldsymbol{u}, \boldsymbol{v};w) := \sum_{\nu^{(1)},\dots,\nu^{(N-1)}} \prod_{i=1}^{N-1}\frac{c_{\nu^{(i)}}}{c'_{\nu^{(i)}}}\,\,\Phi^*_{\nu^{(1)}}[v_1, w'_1]\Phi_{\emptyset}[u_1, w_1]\otimes \Phi^*_{\nu^{(2)}}[v_2, w'_2]\Phi_{\nu^{(1)}}[u_2, w_2]\otimes \cdots 
    \\
    & \hspace{7cm}\cdots\otimes \Phi^*_{\emptyset}[v_N, w'_N]\Phi_{\nu^{(N-1)}}[u_N, w_N]\,,
\end{split}
\end{equation}
and its matrix elements are denoted by
\begin{equation}
        \cTV_{\boldsymbol{\lambda},\boldsymbol{\mu}}(\boldsymbol{u}, \boldsymbol{v};w) =  \bra{P_{\boldsymbol{\lambda}}} \cTV(\boldsymbol{u}, \boldsymbol{v};w) \ket{P_{\boldsymbol{\mu}}}\,.
\end{equation}
\end{definition}

The operator $\cTV$ guarantees the existence of the Mukad\'{e} operator, defined in Def. \ref{def_V}, and the following proposition completes the proof of Proposition \ref{prop: uniqueness of V}.
\begin{proposition}\label{prop: existance_V}
For arbitrary $i \in \{1,2,\dots,N\}$, the operator $ \cTV(\boldsymbol{u}, \boldsymbol{v};w)$ satisfies
\begin{equation}\label{eq_5.28}
    \left(1- \frac{w}{z}\right)X^{(i)}(z) \cTV(\boldsymbol{u}, \boldsymbol{v};w) = \gamma^{-i}\left(1- \gamma^{2 i} \frac{w}{z}\right)\cTV(\boldsymbol{u}, \boldsymbol{v};w) X^{(i)}(z) \,.
\end{equation}
\begin{proof}
The proof is done by the direct calculation. For more details, see Sec \ref{proof of existence of V}.
\end{proof}
\end{proposition}

\begin{remark}
The factor $\gamma^{-i}$ in the right hand side of (\ref{eq_5.28}) can be compensated by redefining the spectral parameters as $\gamma^{-1} u_i$. Under this redefinition, the matrix elements of $\cTV(\vu,\vv;w)$ agree with those of $\cV(w)$.
\end{remark}

\begin{theorem}\label{prop: changing PD}
The following equality between the two matrix elements holds:
\begin{equation}
     \cTH_{\boldsymbol{\lambda},\boldsymbol{\mu}}(\boldsymbol{u}, \boldsymbol{v};w)  \sim \cTV_{\boldsymbol{\lambda},\boldsymbol{\mu}}(\boldsymbol{u}, \boldsymbol{v};w) \,.
\end{equation}
Here, $\sim$ means the both sides are identical up to some monomial factor and $\mathcal{G}$-factors $\mathcal{G}(z) := \prod_{i,j=0}^\infty (1 - z q^i t^{-j})$, 
which appear from the normal orderings among $\Phi_\emptyset$'s and $\Phi^*_\emptyset$'s.
\begin{proof}
We can show, up to the monomial factors and $\mathcal{G}$-factors, the both sides are equal to
\begin{equation}
 \frac{\prod_{ i, j=1 }^N N_{\lambda^{(i)}, \mu^{(j)}}(v_i/\gamma u_j)}{\prod_{1\leq i<j \leq N} N_{\mu^{(i)}, \mu^{(j)}}(qu_i/tu_j) 
\prod_{k=1}^N c_{\mu^{(k)}} \prod_{1\leq i<j \leq N} N_{\lambda^{(j)}, \lambda^{(i)}}(qv_j/tv_i) 
\prod_{k=1}^N c_{\lambda^{(k)}}
}\,.
\end{equation}
For the right hand side, it is obvious from Theorem \ref{thm: matrix elements of V} and Definition \ref{def: K}.
For the left hand side, it can be shown by a direct computation, using some formulas in Appendix \ref{App: Useful Formula}.
\end{proof}
\end{theorem}

\begin{remark}
This proposition implies the invariance under changing preferred directions of the refined topological vertex.
\end{remark}

\subsection{Proof of Proposition \ref{prop: existance_V}}\label{proof of existence of V}

First, we give the proof for $X^{(1)}(z)$. For $X^{(k)}(z)$ with $k>1$, the strategy of the proof is similar to that of $k=1$ case, and thus we just show the sketch of the proof.

\paragraph{\bf{$k=1$ case}}
In order to simplify the notation, we rewrite the operator $\cTV(\boldsymbol{u}, \boldsymbol{v};w)$ as
\begin{align}
\cTV(\boldsymbol{u}, \boldsymbol{v};w) &= 
\sum_{\nu^{(1)},\dots,\nu^{(N-1)}} \mathcal{C}_{\boldsymbol{\nu}} \,\,
     :\wPhi_{\emptyset}(w_1)\wPhi^*_{\nu^{(1)}}(w'_1):\otimes :\wPhi_{\nu^{(1)}}(w_2)\wPhi^*_{\nu^{(2)}}(w'_2):\nonumber\\
     &\hspace{5cm}\otimes \cdots \otimes :\wPhi_{\nu^{(N-1)}}(w_N)\wPhi^*_{\emptyset}(w'_N):\nonumber 
\\
&=: \sum_{\nu^{(1)},\dots,\nu^{(N-1)}}  \cTV_{\vn}(\boldsymbol{u}, \boldsymbol{v};w),
\end{align}
with
\begin{equation}
\begin{split}
    \mathcal{C}_{\boldsymbol{\nu}} &:= \prod_{i=1}^{N}\frac{c_{\nu^{(i)}}}{c'_{\nu^{(i)}}}\,\,\mathcal{G}(w_i/\gamma w'_i)^{-1}\\
    &\qquad\times\hat t(\nu^{(i-1)},u_i ,w_i,0)\,\,\hat t^*(\nu^{(i)} ,w'_i ,v_i,0)\,\, N_{\nu^{(i-1)}\nu^{(i)}}(w_i/\gamma w'_i)\big/\bra{\boldsymbol{0}}\cTV(\boldsymbol{u}, \boldsymbol{v};w)\ket{\boldsymbol{0}},
\end{split}
\end{equation}
which appear from the normal orderings.
Here, we put $\nu^{(0)} = \nu^{(N)} = \emptyset$ for convenience.
We also omit the parameters $\boldsymbol{u}, \boldsymbol{v}$ in $\cTV(\boldsymbol{u}, \boldsymbol{v};w)$ in what follows.

We first compute the commutation relations between $\Lambda^{(j)}, j= 1,\dots,N$ and $\cTV(w)$.
Putting $z_j = \gamma^{j-1} z$, we have
\begin{align}
& (w_j/\gamma w'_j) \,\, v_j \Lambda^{(j)}(z)\,\, \cTV_{\vn}(w) - \frac{1- z/\gamma^2 w_1}{1-z/w_1}\cTV_{\vn}(w)\,\, u_j\Lambda^{(j)}(z) \nonumber
\\
&= -u_j\left(\text{delta functions from $A(\nu^{(j-1)})$}\right)\nonumber\\ 
&- 
u_j\sum_{y \in R(\nu^{(j)})}\delta(\gamma z_j/w'_j \chi_y)\,\, \mathcal{C}_{\boldsymbol{\nu}} \,\,\frac{\prod_{x\in R(\nu^{(j-1)})}1-\gamma w'_j \chi_y/ w_j \chi_x}{\prod_{x\in A(\nu^{(j-1)})}1- w'_j \chi_y/\gamma w_j \chi_x}
\frac{\prod_{x\in A(\nu^{(j)})}1-  \chi_y/\gamma^2  \chi_x}{\prod_{x\in R(\nu^{(j)}), x\neq y}1-\chi_y/\chi_x}
\nonumber \\
&\times
\cdots \otimes :\eta(\gamma^{-1}w'_j\chi_y)\wPhi_{\nu^{(j-1)}}\wPhi^*_{\nu^{(j)}}:\otimes:\wPhi_{\nu^{(j)}}\wPhi^*_{\nu^{(j+1)}}:\otimes \cdots\nonumber
\\
&= -u_j\left(\text{delta functions from $A(\nu^{(j-1)})$}\right)\\ 
&- 
u_j\sum_{y \in R(\nu^{(j)})}\delta(\gamma z_j/w'_j \chi_y)\,\, \mathcal{C}_{\boldsymbol{\nu}} \,\,\frac{\prod_{x\in R(\nu^{(j-1)})}1-\gamma w'_j \chi_y/ w_j \chi_x}{\prod_{x\in A(\nu^{(j-1)})}1- w'_j \chi_y/\gamma w_j \chi_x}
\frac{\prod_{x\in A(\nu^{(j)})}1-  \chi_y/\gamma^2  \chi_x}{\prod_{x\in R(\nu^{(j)}), x\neq y}1-\chi_y/\chi_x}
\nonumber \\
&\times
\cdots \otimes :\varphi^-(\gamma^{-1/2}w'_j \chi_y)\wPhi_{\nu^{(j-1)}}(w_j)\wPhi^*_{\nu^{(j)}-y}(w'_j):\otimes:\wPhi_{\nu^{(j)}}(w_{j+1})\wPhi^*_{\nu^{(j+1)}}(w'_{j+1}):\otimes \cdots\,,\nonumber
\end{align}
where $A(\nu)$ and $R(\nu)$ is defined in Section \ref{sec: Intro}. 
Through the computation, we used the equality
\begin{equation}
\eta (\gamma^{-1}u) =\,\, :\varphi^{-}(\gamma^{-1/2} u)\xi^{-1}(u):\,.
\end{equation}

Now, for $\nu^{(j)}\neq \emptyset$, fix a $y \in R(\nu^{(j)})$ and pick up the corresponding term from the above equation. Then, we rewrite $\nu^{(j)}-y$ as $\nub^{(j)}$, and express the term using $\nub^{(j)}$. This procedure ends with the following expression: 
\begin{align}
&-u_j\delta(\gamma z_j/w'_j \chi_y)\,\, \mathcal{C}_{\boldsymbol{\nu}} \,\,\frac{\prod_{x\in R(\nu^{(j-1)})}1-\gamma w'_j \chi_y/ w_j \chi_x}{\prod_{x\in A(\nu^{(j-1)})}1- w'_j \chi_y/\gamma w_j \chi_x}
\frac{\prod_{x\in A(\nu^{(j)})}1-  \chi_y/\gamma^2  \chi_x}{\prod_{x\in R(\nu^{(j)}), x\neq y}1-\chi_y/\chi_x}
\nonumber \\
&\times
\cdots \otimes :\varphi^-(\gamma^{-1/2}w'_j \chi_y)\wPhi_{\nu^{(j-1)}}(w_j)\wPhi^*_{\nu^{(j)}-y}(w'_j):\otimes:\wPhi_{\nu^{(j)}}(w_{j+1})\wPhi^*_{\nu^{(j+1)}}(w'_{j+1}):\otimes \cdots\nonumber 
\\
&=u_{j+1}\delta(\gamma z_j/w'_j \chi_y)\,\, \mathcal{C}_{(\dots,\nub^{(j)},\dots)} \,\,
\frac{\prod_{x\in R(\nu^{(j)})}1-\gamma^2  \chi_y/\chi_x}{\prod_{x\in A(\nu^{(j)}), x\neq y}1- \chi_y/\chi_x}
\frac{\prod_{x\in A(\nu^{(j+1)})}1-w_{j+1}\chi_y/\gamma w'_{j+1}\chi_x}{\prod_{x\in R(\nu^{(j+1)})}1-\gamma w_{j+1}\chi_y / w'_{j+1}\chi_x}\nonumber 
\\
&\times
\cdots \otimes :\varphi^-(\gamma^{-1/2}w'_j \chi_y)\wPhi_{\nu^{(j-1)}}(w_j)\wPhi^*_{\nub^{(j)}}(w'_j):\otimes:\wPhi_{\nub^{(j)}+y}(w_{j+1})\wPhi^*_{\nu^{(j+1)}}(w'_{j+1}):\otimes \cdots\,.
\end{align}
While taking summation over $\vn$, these terms are canceled by the terms which appear from the commutation relation between $\Lambda^{(j+1)}$ and $\cTV_{\vn'}(w)$ with $\nu^{(j)'} = \nub^{(j)}$, 
\begin{align}
& (w_{j+1}/\gamma w'_{j+1}) \,\, v_{j+1}\Lambda^{(j+1)}(z) \,\,\cTV_{\vn'}(w) - \frac{1- z/\gamma^2 w_1}{1-z/w_1}\cTV_{\vn'}(w)\,\, u_{j+1}\Lambda^{(j+1)}(z) \nonumber
\\
&= -u_{j+1}\left(\text{delta functions from $R(\nu^{(j+1)'})$}\right)\nonumber \\
&- u_{j+1}\sum_{y \in A(\nu^{(j)'})}
\delta(z_{j+1}/w_{j+1} \chi_y)
\frac{\prod_{x\in R(\nu^{(j)'})}1-\gamma^2  \chi_y/\chi_x}{\prod_{x\in A(\nu^{(j)'}), x\neq y}1- \chi_y/\chi_x}
\times 
\frac{\prod_{x\in A(\nu^{(j+1)'})}1-w_{j+1} \chi_y/\gamma w'_{j+1} \chi_x }{\prod_{x\in R(\nu^{(j+1)'})}1-\gamma w_{j+1} \chi_y/ w'_{j+1} \chi_x}\nonumber
\\
&\times
\cdots \otimes :\varphi(\gamma^{-1/2}w_{j+1}\chi_y)\wPhi_{\nu^{(j-1)'}}(w_j)\wPhi^*_{\nu^{(j)'}}(w'_j):\otimes:\wPhi_{\nu^{(j)'}+y}(w_{j+1})\wPhi^*_{\nu^{(j+1)'}}(w'_{j+1}):\otimes \cdots\,.
\end{align}

While summing up for $j$ from $1$ to $N$, the delta function which is not cancelled by the above mechanism exists at $A(\nu^{(0)}) = A(\emptyset)$.
This delta function vanishes when we multiply the both sides by $1-w_1/z$, and we finally obtain the main claim of Proposition \ref{prop: existance_V} for the $k=1$ case.

\paragraph{\bf{$k>1$ case}}

We use the following simplified notation,
\begin{equation}
    \Lambda^{(i_1,\dots,i_k)}(z) : = : \Lambda^{(i_1)}(z) \cdots \Lambda^{(i_k)}(\gamma^{2(1-k)}z):\,.
\end{equation}

We can compute the commutation relations between $ \Lambda^{(i_1,\dots,i_k)}(z)$ and $\cTV_{\vn}(w)$ as,
\begin{align}
        &\prod_{j=1}^k (w_{i_j}/ \gamma w'_{i_j})\,\, v_{i_1}\cdots v_{i_k}\Lambda^{(i_1,\dots,i_k)}(z) \,\,\cTV_{\vn}(w)  -   \frac{1-z/\gamma^{2k} w_1}{1-z/w_1}  \cTV_{\vn}(w)\,\,  u_{i_1}\cdots u_{i_k}\Lambda^{(i_1,\dots,i_k)}(z)\nonumber 
        \\
        & \quad  =  -u_{i_1}\cdots u_{i_k}\Big( \text{delta functions from } A(\nu^{(i_1-1)}), R(\nu^{(i_1)}), A(\nu^{(i_2-1)}), R(\nu^{(i_2)}),\nonumber\\
        &\hspace{5cm}\dots, A(\nu^{(i_k-1)}) \text{ and } R(\nu^{(i_k)})   \Big)
\end{align}

The delta functions related to $R(\nu^{(i_j)})$ (for $j = 1,\dots,k$) cancel those related to $A(\nu^{(i_j)})$ which appear in the commutation relations between $ \Lambda^{(\dots,i_j+1,\dots)}(z)$ and $\cTV_{\vn}(w)$.
This sequence of the cancellation begins when $i_j = i_{j-1}+1$ and terminates when $i_j = i_{j+1}-1$, because in those cases, the poles and zeros related to $\nu^{(i_j)}$ cancel each other, and no delta functions related to them appear.

As a result, the only delta function which survives the above cancellation mechanism is that related to $A(\nu^{(0)})$.
Similarly to the $k=1$ case, it vanishes when $1-w_1/z$ is multiplied, and we obtain the expected commutation relation between $X^{(k)}(z)$ and $\cTV(w)$.

\section*{Acknowledgment}
The authors would like to thank H.~Awata, B.~Feigin, A.~Hoshino, M.~Kanai, H.~Kanno, Y.~Matsuo, M.~Noumi and S.~Yanagida for valuable discussions. 
The research of J.S.~is supported by JSPS KAKENHI (Grant Numbers 15K04808 and 16K05186).
Y.O. and M.F.~are partially supported by Grant-in-Aid for JSPS Research Fellow 
(Y.O.: 18J00754, M.F.: 	17J02745).

\appendix

\section{Construction of Macdonald Symmetric Functions in terms of Topological Vertex}\label{App: Mac from top vertex}

We can obtain Macdonald functions as matrix elements of 
some compositions of the intertwining operators. 
The similar derivation and its supersymmetric version are given in \cite{Z}. 
Set 
\begin{align}
&\widetilde{\mathcal{T}}_i(x)=\widetilde{\mathcal{T}}_i(\vu; x)\nonumber\\
&:=\prod_{k=1}^N\mathcal{G}(u_k/\gamma v_k)\cdot\bra{\boldsymbol{0}}\cTV( \vv, \vu ;x)\ket{\boldsymbol{0}}\cdot\mathcal{T}^V\big(\vv, \vu ;x\big)
\Bigg|_{ \substack{ v_k \to \gamma^{-1} t^{-\delta_{k,i}}u_k \\ (1\leq k\leq N)}},
\end{align}
where the overall factor is chosen for the later convenience.
Though the prefactor consisting of the $\mathcal{G}$-factors gives zeros under the restriction of $\vv$, but the $\mathcal{G}$-factors which appear from the normal orderings among $\Phi$'s and $\Phi^*$'s in $\cTV$, cancel those zeros, and thus the operator is well-defined.
By operator products (\ref{eq: Phi* Phi OPE}) and Lemma \ref{lem: nonzero cond of N}, 
it can be seen that the Young diagrams 
are restricted to only one row 
when gluing intertwiners over the vertical representations (Fig. \ref{fig: til Ti}), i.e., 
\begin{align}
\widetilde {\mathcal{T}}_i(x) &=
\sum_{0\leq m_1\leq m_2\leq \cdots \leq m_{i-1} <\infty} 
t^{-m_{i-1}} 
\prod_{k=1}^{i-1} \frac{q^{2n((1^{m_k}))}}{c'_{(m_k)} c_{(m_k)}}
\left(\frac{qu_{k+1}}{\gamma u_k}\right)^{m_k}f^{-1}_{(m_k)}  \nonumber \\
&\qquad\qquad \times N_{\emptyset,(m_{i-1})}(t^{-1})\prod_{k=1}^{i-1} N_{(m_{k-1}),(m_k)}(1)  
\nonumber\\
&\qquad\qquad\times:\wPhi^*_{(m_1)}(\gamma^{-1}x)\wPhi_{(\emptyset)}(x): \otimes  
\aotimes{k=2}{i-1} :\wPhi^*_{(m_k)}(\gamma^{-k}x)\wPhi_{(m_{k-1})}(\gamma^{-k+1}x):\nonumber \\
&\qquad\qquad \otimes  :\wPhi^*_{\emptyset}(\gamma^{-i}t^{-1}x)\wPhi_{(m_{i-1})}(\gamma^{-i+1}x):
\otimes \aotimes{k=i+1}{N} 
:\wPhi^*_{\emptyset}(\gamma^{-k}t^{-1}x)\wPhi_{\emptyset}(\gamma^{-k+1}t^{-1}x):,
\end{align}
where we put $m_0 = 0$.
Here, $\aotimes{}{}$ is introduced in Notation \ref{not: aotimes}. 
Then, we have $\widetilde{\mathcal{T}}_1(x)=\Phi^{(0)}(t^{-1}x)$. 
Note that the operators $\Phi^{(k)}(x)\,\,(k=0,\dots,N-1)$ in this section slightly differ from those in Sec. \ref{SSec_VO}. In this section, $\Phi^{(k)}(x)$ is a map 
$\cF_{t^{-\delta_{k+1}}\cdot \tilde{\vu}} \to \cF_{\vu}$ with  $\tilde{\vu}=(\gamma^{-1}u_1, \ldots, \gamma^{-1}u_N)$, 
that is, the spectral parameter differs by the factor $\gamma^{-1}$.

\begin{figure}[H]
\begin{center}
\includegraphics[width=0.6\hsize]{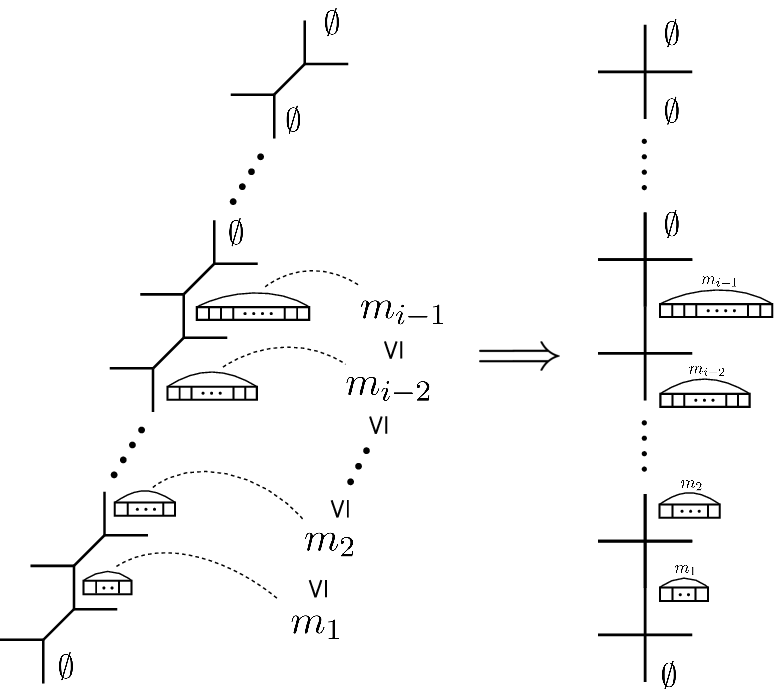}
\caption{The operator $\widetilde{\mathcal{T}}_i(x)$. 
($\Longrightarrow$ stands for a simplification of the diagram for convenience) }\label{fig: til Ti}
\end{center}
\end{figure} 

We obtain the expression of Macdonald functions in terms of intertwiners of the DIM algebra. 
The following proposition says the vacuum expectation value of Fig. \ref{fig: sharp mark} gives the Macdonald function of the $A$ type.  
\begin{figure}[H]
\begin{center}
\includegraphics[height=8cm]{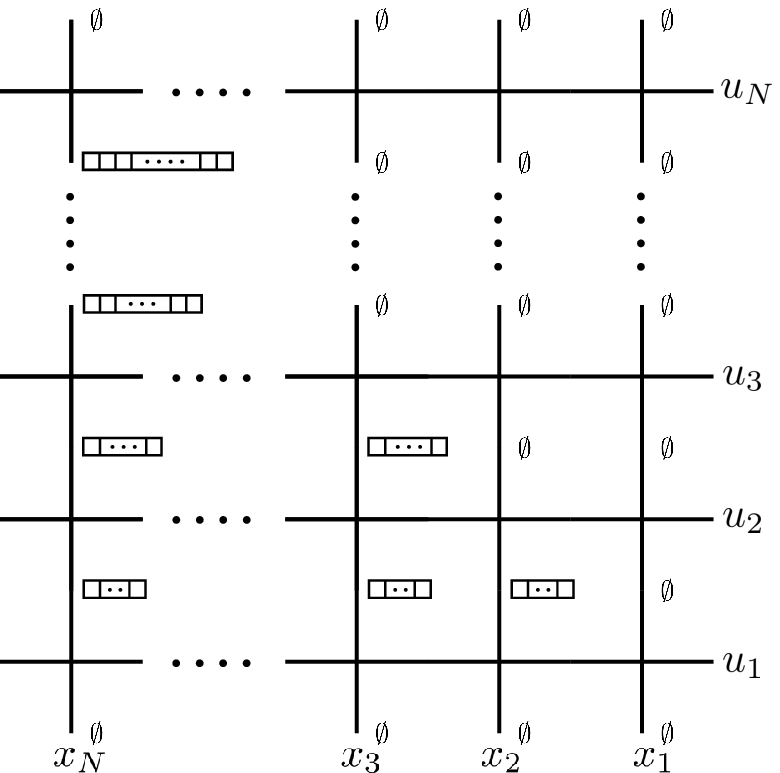}
\caption{The diagram for the Macdonald function $p_{N}(\vx;\vu|q,q/t)$. }\label{fig: sharp mark}
\end{center}
\end{figure} 

\begin{proposition}
\begin{align}
\bra{\boldsymbol{0}}\widetilde{\mathcal{T}}_1(\vu;x_1) \widetilde{\mathcal{T}}_2(x_2) \cdots 
\widetilde{\mathcal{T}}_N(x_N)\ket{\boldsymbol{0}} 
\propto 
p_{N}(\vx;\vu|q,q/t).
\end{align}

\begin{proof}
The operators $\wPhi_{(m)}(z)$ and $\wPhi^*_{(m)}(z)$ 
can be decomposed as 
\begin{align}
\wPhi_{(m)}(z)=:\wPhi_{\emptyset}(t^{-1}z) \mathcal{A}(q^mz):, \quad 
\wPhi^*_{(m)}(z)=:\wPhi_{\emptyset}^*(t^{-1}z) \mathcal{A}^*(q^mz):, 
\end{align}
where 
\begin{align}
&\mathcal{A}(z)=
\exp\left( -\sum_{n>0}\frac{1-t^{-n}}{n(1-q^n)} a_{-n}z^n\right) 
\exp\left( \sum_{n>0}\frac{1-t^n}{n(1-q^{-n})} a_{n} z^{-n}\right), \\
&\mathcal{A}^*(z)=
\exp\left( \sum_{n>0}\frac{1-t^{-n}}{n(1-q^n)}\gamma^n a_{-n}z^n\right) 
\exp\left( -\sum_{n>0}\frac{1-t^n}{n(1-q^{-n})} \gamma^{n}a_{n} z^{-n}\right).  
\end{align}
Then the tensor product of $\mathcal{A}(z)$ and $\mathcal{A}^*(z)$ 
corresponds to the screening current (Definition \ref{def: screening}): 
\begin{align}
\mathcal{A}^*(z)\otimes \mathcal{A}(z)=\phi^{sc}(\gamma^{-1}t^{-1}z). 
\end{align}
Therefore, 
we have 
\begin{align}\label{eq: Jackson til Ti}
\widetilde{\mathcal{T}}_i(\vu; x) =&
 \left( \frac{(q/t;q)_{\infty}}{(q;q)_{\infty}}\right)^{i-1} \\
&\times \sum_{0\leq m_1\leq m_2\leq \cdots \leq m_{i-1} <\infty} 
\Phi^{(0)}(t^{-1}x) \widetilde{S}^{(1)}(q^{m_1}x)  \cdots  \widetilde{S}^{(i-1)}(q^{m_{i-1}}x) 
\prod_{k=1}^{i-1}(u_{k+1}/u_{k})^{m_k}. \nonumber 
\end{align}
Here, we put 
\begin{align}
\widetilde{S}^{(k)} (z)=S^{(k)}(\gamma^{-2k}t^{-1} z). 
\end{align}
As in the proof of Proposition \ref{prop: commutativity of S} (Appendix \ref{sec: pf of screening}), 
$X^{(1)}(z)$ commutes with $\widetilde{S}^{(k)}(w)$ up to 
$q$-difference:  
\begin{align}
\left[X^{(1)}(z),\widetilde{S}^{(k)}(w)\right]
=(t-1)(u_{k+1}T_{q,w}-u_k) \left( \delta\left( \frac{\gamma^{-2k}w}{qz} \right):\Lambda^{(k)}(\gamma^{-2k }w/q)\widetilde{S}^{(k)}(w):\right). 
\end{align}
Further,  
by the property of the operator products 
\begin{align}
\widetilde{S}^{(k-1)}(x)\Lambda^{(k)}(\gamma^{-2k }x/q)
=\Phi^{(0)}(t^{-1}x) \Lambda^{(1)}(\gamma^{-2}x/q)=0, 
\end{align}
we obtain 
\begin{align}
&\Phi^{(0)}(t^{-1}x) \cdot \left[\Xo_{0}, \sum_{m=0}^{\infty} \widetilde{S}^{(1)}(q^{m}x)(u_{2}/u_{1})^{m}\right]=0, \\
&\widetilde{S}^{(k-1)}(x)\cdot \left[\Xo_{0}, \sum_{m=0}^{\infty} \widetilde{S}^{(k)}(q^{m}x)(u_{k+1}/u_{k})^{m}\right]=0 \quad (k\geq 2). 
\end{align}
This leads to 
\begin{align}
&\Xo(z) \widetilde{\mathcal{T}}_i(\vu;w)
- \gamma \frac{1-qz/tw}{1-z/w} \widetilde{\mathcal{T}}_i(\vu;w)\Xo(z) =u_i(1-t^{-1}) \widetilde{\mathcal{T}_i}(\vu;qw)\Psi^+(t^{-1}w) \delta(w/z). 
\end{align}
By this relation, we get  
\begin{align}
&\bra{\boldsymbol{0}}\Xo(z)\widetilde{\mathcal{T}}_1(x_1) \cdots \widetilde{\mathcal{T}}_N(x_N)\ket{\boldsymbol{0}} 
= \gamma^N \prod_{k=1}^N \frac{1-qz/tx_k}{1-z/x_k} \cdot 
\bra{\boldsymbol{0}}\widetilde{\mathcal{T}}_1(x_1) \cdots \widetilde{\mathcal{T}}_N(x_N)\Xo(z)\ket{\boldsymbol{0}} \nonumber \\
&\qquad +(1-t^{-1}) \sum_{i=1}^{N} \delta(x_i/z) u_i 
\prod_{k=1}^{i-1}\frac{1-qx_i/tx_k}{1-x_i/x_k}
\prod_{k=1}^{i-1}\frac{1-tx_k/qx_i}{1-x_k/x_i}T_{q,x_i} 
\bra{\boldsymbol{0}}\widetilde{\mathcal{T}}_1(x_1)  \cdots \widetilde{\mathcal{T}}_N(x_N)\ket{\boldsymbol{0}}. 
\end{align}
Thus, the considered matrix elements are eigenfunctions of the difference operator $D_{N}^1$: 
\begin{align}
D_{N}^1(u;q, q/t)
\bra{\boldsymbol{0}}\widetilde{\mathcal{T}}_1(x_1) \cdots \widetilde{\mathcal{T}}_N(x_N)\ket{\boldsymbol{0}} =(u_1+\cdots +u_N)\bra{\boldsymbol{0}}\widetilde{\mathcal{T}}_1(x_1) \cdots \widetilde{\mathcal{T}}_N(x_N)\ket{\boldsymbol{0}}. 
\end{align}
\end{proof}
\end{proposition}

\begin{remark}
By using the expression (\ref{eq: Jackson til Ti}), 
it is shown that the operator $\widetilde{\mathcal{T}}_i(x)$ also represents the screened vertex operator $\Phi^{(i-1)}(x)$. 
Whereas $\widetilde{\mathcal{T}}_i(x)$ is written by a sort of Jackson integrals, 
$\Phi^{(i-1)}(x)$ is introduced by counter integrals. 
In order to justify the equivalence of the two expressions, see the following argument.
\end{remark}
We first note the following deformation of the formal series,
\begin{equation}\label{eq_Jack_to_Cont}
\begin{split}
       \text{(Jackson integral)}\quad \sum_{m\in\mathbb{Z}_{\geq 0}} \delta(q^m z) \alpha^m &=\sum_{m\in\mathbb{Z}_{\geq 0}, k \in \mathbb{Z}} (q^m z)^k \alpha^m \\
        &= \sum_{k\in \mathbb{Z}} \frac{1}{1-q^k \alpha} z^k = \frac{\theta_q(\alpha)}{(q;q)_\infty (q^{-1};q)_\infty}\frac{\theta_q (\alpha z )}{\theta_q (z)}\quad \text{(Contour integral)}\,.
\end{split}
\end{equation}
Of course, this series does not converge for arbitrary $q$ and $\alpha$. However, by the following reason, we can choose some domains for $q$ and $\alpha$ so that the matrix elements of $\widetilde{\mathcal{T}}_i(x)$ converge there, and thus they are identical to those of $\Phi^{(i-1)}(x)$. We discuss the example of the $i=2$ case. The matrix elements of $\widetilde{\mathcal{T}}_2(x)$ are of the form,
\begin{equation}
\begin{split}
        \bra{P_{\vl}} \widetilde{\mathcal{T}}_2(x)\ket{Q_{\vm}} &\sim \oint \frac{dw}{w} \sum_{m\geq 0}\delta(q^m x/w) \alpha^m \bra{P_{\vl}} \Phi^{(0)}(x) \widetilde{S}^{(1)}(w)\ket{Q_{\vm}}\\
        &\sim \oint \frac{dw}{w} \sum_{m\geq 0}\sum_{k\in \mathbb{Z}}(q^m x/w)^k \alpha^m \sum_{n\geq 0} \frac{(t;q)_n}{(q;q)_n}(w/x)^n \bra{P_{\vl}} :\Phi^{(0)}(x) \widetilde{S}^{(1)}(w):\ket{Q_{\vm}}\,,
\end{split}
\end{equation}
where $\alpha = u_2/u_1$, and $\sim$ means the both sides are equal up to some overall factors.
The factor $\bra{P_{\vl}} :\Phi^{(0)}(x) \widetilde{S}^{(1)}(w):\ket{Q_{\vm}}$ is some finite sum of the monomial of the form $x^a w^b$ with the maximums and minimums of $a$'s and $b$'s depending on $\vl$ and $\vm$.
Then, by taking the constant term in $w$ in the above, the running index $k$ is fixed as $k = n + b \geq b$.
Therefore, when we choose the $q$ and $\alpha$ so that
\begin{equation}
    |\alpha q^{\mathrm{min}(b)}| \ll 1\,,
\end{equation}
the deformation (\ref{eq_Jack_to_Cont}) is justified by putting $k< \mathrm{min} (b)$ terms zero, and thus 
\begin{equation}
    \begin{split}
        \bra{P_{\vl}} \widetilde{\mathcal{T}}_2(x)\ket{Q_{\vm}} &\sim \oint \frac{dw}{w} \sum_{m\geq 0}\sum_{k\in \mathbb{Z}}(q^m x/w)^k \alpha^m \sum_{n\geq 0} \frac{(t;q)_n}{(q;q)_n}(w/x)^n \bra{P_{\vl}} :\Phi^{(0)}(x) \widetilde{S}^{(1)}(w):\ket{Q_{\vm}}\\
        &\sim \oint \frac{dw}{w} \frac{\theta_q (\alpha x/w )}{\theta_q (x/w)}  \bra{P_{\vl}} \Phi^{(0)}(x) \widetilde{S}^{(1)}(w)\ket{Q_{\vm}} \sim \bra{P_{\vl}} \Phi^{(1)}(x) \ket{Q_{\vm}}\,.
    \end{split}
\end{equation}
This discussion can extend to the general $i$ case, and the equivalence between the Jackson integral and contour integral is justified.

\section{Basic Facts for Ordinary Macdonald Functions}\label{App: ord. Macdonald}

In this section, we give the definition and basic facts for ordinary Macdonald functions. 
Let $\Lambda_n=\mathbb{C}[x_1,\ldots,x_n]^{\mathfrak{S}_n}$ be the ring of symmetric polynomials, 
and $\Lambda=\lim_{\leftarrow}\Lambda_n$ be the projective limit in the category of graded rings, 
i.e., the ring of symmetric functions. 
$m_{\lambda}$ denotes the monomial symmetric function. 
We denote the power sum symmetric functions by $\pp_n=\pp_n(x)=\sum_{i \geq 1} x^n_i$. 
For a partition $\lambda$, set $\pp_{\lambda}=\prod_{i\geq 1} \pp_{\lambda_{i}}$. 
Macdonald functions are defined as orthogonal functions with respect to  
the following scalar product. 

\begin{definition}
Define the bilinear form $\langle -,- \rangle_{q,t}:\Lambda \otimes \Lambda \rightarrow \mathbb{C}$ 
by 
\begin{align}
\langle \pp_{\lambda}, \pp_{\mu} \rangle_{q,t}
=\delta_{\lambda, \mu}z_{\lambda} \prod_{i=1}^{\ell(\lambda)} 
\frac{1-q^{\lambda_i}}{1-t^{\lambda_i}}, \quad 
z_{\lambda}:=\prod_{i\geq 1}i^{m_i}\cdot m_i!,
\end{align}
where $m_i$ is the number of entries in $\lambda$ equal to $i$. 
\end{definition}

\begin{fact}[\cite{M}]
There exists an unique function $P_{\lambda}\in \Lambda$ such that 
\begin{align}
&P_{\lambda}=m_{\lambda}+\sum_{\mu< \lambda} \alpha_{\lambda,\mu} m_{\mu}\quad (\alpha_{\lambda,\mu}  \in \mathbb{C});\\ 
&\langle P_{\lambda}, P_{\mu} \rangle_{q,t}=0 \quad (\lambda \neq \mu). 
\end{align}
\end{fact}

The symmetric functions $P_{\lambda}$ are called Macdonald symmetric functions. 
It is known that 
\begin{align}
\langle P_{\lambda}, P_{\lambda} \rangle_{q,t}=c'_{\lambda}/c_{\lambda}. 
\end{align} 
Here, $c_{\lambda}$ and $c'_{\lambda}$ are defined in (\ref{eq: c and c'}). 
Set $Q_{\lambda}:=\langle P_{\lambda}, P_{\lambda} \rangle_{q,t}^{-1}\cdot P_{\lambda}$, so that 
$\langle P_{\lambda}, Q_{\lambda} \rangle_{q,t}=1$. 
In this paper, 
we regard power sum symmetric functions as variables of Macdonald functions 
and write $P_{\lambda}=P_{\lambda}(\pp_n)$, $Q_{\lambda}=Q_{\lambda}(\pp_n)$. 
These are abbreviation for $P_{\lambda}(\pp_1,\pp_2,\ldots)$ and $Q_{\lambda}(\pp_1,\pp_2,\ldots)$. 
We often substitute the generators $a_{n}$ 
in the Heisenberg algebra into Macdonald functions 
as $P_{\lambda}(a_{-n})$. 
Note that this substitution preserves the properties of Macdonald functions as follows. 
The space of symmetric functions $\Lambda$ 
and the Fock space $\mathcal{F}$ are isomorphic 
as graded vector spaces, 
and they can be identified by 
\begin{align}
\Lambda \rightarrow \mathcal{F}, \quad \pp_{\lambda} \mapsto \ket{a_{\lambda}}. 
\end{align}
Further, the bilinear form on the Fock spaces preserves the structure of the scalar product 
$\langle -, - \rangle_{q,t}$ 
under this identification, i.e., 
\begin{align}
\langle \pp_{\lambda}, \pp_{\mu} \rangle=\braket{a_{\lambda}|a_{\mu}}. 
\end{align}

We describe some facts for Macdonald functions that are used in this paper. 

\subsection*{Kernel function}

The following function $\Pi(x,y;q,t)$ is called the kernel function: 
\begin{align}
\Pi(x,y;q,t):=\prod_{n\geq 1}\exp \left( \frac{1-t^n}{1-q^n}  \pp_n(x)\pp_n(y) \right). 
\end{align}
This function $\Pi(x,y;q,t)$ can be expanded in terms of dual bases parametrized by partitions 
as follows. 

\begin{fact}[\cite{M}]\label{fact: kernel}
Let $u_{\lambda}$, $v_{\lambda} \in \Lambda$ 
be homogeneous symmetric functions of level $|\lambda|$ 
and $\{u_{\lambda}\}$, $\{v_{\lambda}\}$ form $\mathbb{C}$-bases on $\Lambda$. 
Then, the followings are equivalent: 
\begin{align}
&\cdot\; \langle u_{\lambda}, v_{\lambda} \rangle_{q,t} =\delta_{\lambda, \mu} \quad \mbox{for all $\lambda, \mu$}; \\
&\cdot\; \Pi(x,y;q,t) =\sum_{\lambda}u_{\lambda}(x) v_{\lambda}(y). 
\end{align}
\end{fact}

Note that especially, we have
\begin{align}
\Pi(x,y;q,t)=\sum_{\lambda} P_{\lambda}(\pp_n(x)) Q_{\lambda}(\pp_n(y)). 
\end{align}

\subsection*{Pieri formula}

Define the symmetric function $g_r\in \Lambda$ to be the expansion coefficient 
of $y^r$ in the following series: 
\begin{align}
\exp \left( \sum_{n\geq 1} \frac{1}{n}\frac{1-t^n}{1-q^n} \pp_n y^n\right) =\sum_{r\geq 0} g_r \, y^r. 
\end{align}
For a partition $\lambda$ and a coordinate $s$,  we write 
\begin{align}
b_{\lambda}(s)=
\begin{cases}
\dfrac{1-q^{a_{\lambda}(s)}t^{l_{\lambda}(s)+1}}{1-q^{a_{\lambda}(s)+1}t^{l_{\lambda}(s)}}, & s \in \lambda; \\
1, & \mathrm{otherwise}.
\end{cases}
\end{align}

\begin{fact}[\cite{M}]\label{fact: Pieri}
We have the Pieri rules:
\begin{align}
g_rQ_{\mu}=\sum_{\lambda}\prod_{s\in R_{\lambda/\mu}- C_{\lambda/\mu}}\frac{b_{\mu}(s)}{b_{\lambda}(s)} \cdot Q_{\lambda}\,.
\end{align}
Here, the summation is over the partitions $\lambda$ such that 
$\lambda/\mu$ is a horizontal $r$-strip, i.e.,  
$\lambda/\mu$ has at most one box in each column. 
$R_{\lambda/\mu}$ (resp. $C_{\lambda/\mu}$) is the union of the rows (resp. columns) that intersect 
$\lambda-\mu$. 
\end{fact}
Note that by this formula, 
we can see that 
$Q_{\mu}$ ($\mu\ngeq\lambda$) does not appear in the expansion of the product $\prod_{i \geq 1} g_{\lambda_i}$ in the basis of Macdonald functions.

\subsection*{Another scalar product}
There is another scalar product for which Macdonald functions $P_{\lambda}$'s are pairwise orthogonal. 
Let $L_r=\mathbb{C}[x_{1}^{\pm 1},\ldots, x_{r}^{\pm}]$ be the $\mathbb{C}$-algebra of Laurent polynomials in $r$ variables. 
For $f(x_1,\ldots, x_r) \in L_r$, we put $\bar f= f(x_1^{-1},\ldots, x_r^{-1})$.

\begin{definition}
The scalar product $\langle -,- \rangle'_r$ on $L_r$ is defined by 
\begin{align}
\langle f,g \rangle'_r:=\oint\prod_{i=1}^r\frac{dz_i}{2 \pi \sqrt{-1}} f \bar g \Delta(x), 
\end{align}
where 
\begin{align}
\Delta(x)&:=\prod_{ i<j} 
\frac{(x_i/ x_j;q)_{\infty}(x_j/ x_i;q)_{\infty}}{(tx_i/x_j;q)_{\infty}(tx_j/x_i;q)_{\infty}}. 
\label{eq: Delta(z)}
\end{align}
\end{definition}

\begin{fact}[\cite{M}]\label{fact: <P,Q>'}
Let $P_{\lambda}^{(r)}=P_{\lambda} \big|_{x_i \rightarrow 0\, (i>r)}, \,
Q_{\lambda}^{(r)}=Q_{\lambda} \big|_{x_i \rightarrow 0\, (i>r)} \in \Lambda_r$ be the Macdonald symmetric polynomials in $r$ variables. 
Then 
\begin{align}
\langle P_{\lambda}^{(r)},Q_{\mu}^{(r)} \rangle'_r
=\delta_{\lambda, \mu} 
\langle 1,1 \rangle'_r \prod_{(i,j) \in \lambda}\frac{1-q^{j-1}t^{r-i+1} }{1-q^{j}t^{r-i} },
\end{align}
where $\ell(\lambda), \ell(\mu) \leq r$. 
\end{fact}

\section{Some Useful Formulas}\label{App: Useful Formula}

\subsection{List of operator products}
In this subsection, we list some formulas for the normal ordering among the various operators appeared in the main text. 
We have 
\begin{align}
& \Lambda^{(i)}(z) S^{(i)}(w) = \frac{1-t^2w/qz}{1-tw/qz}:\Lambda^{(i)}(z) S^{(i)}(w):\,,\label{eq: Lam S OPE 1}\\ 
&\Lambda^{(i+1)}(z) S^{(i)}(w) = \frac{1-w/z}{1-tw/z}:\Lambda^{(i+1)}(z) S^{(i)}(w):\,,\\
& \Lambda^{(j)}(z) S^{(i)}(w) = :\Lambda^{(j)}(z) S^{(i)}(w): \quad \text{for} \quad j<i\quad\text{and}\quad j>i+1\,,\\
& S^{(i)}(w)\Lambda^{(i)}(z)  = \frac{1-qz/t^2w}{1-qz/tw}:S^{(i)}(w)\Lambda^{(i)}(z):\,,\\ 
&S^{(i)}(w)\Lambda^{(i+1)}(z)  = \frac{1-z/w}{1-z/tw}:S^{(i)}(w)\Lambda^{(i+1)}(z):\,,\\
& S^{(i)}(w)\Lambda^{(j)}(z)   = :S^{(i)}(w)\Lambda^{(j)}(z): \quad \text{for} \quad j<i\quad\text{and}\quad j>i+1\,,\label{eq: Lam S OPE 2}
\end{align}

\begin{align}
&\Phi^{(0)}(z) S^{(1)}(w)=\frac{(t w/ z ; q)_\infty}{(w/z ; q)_\infty}:\Phi^{(0)}(z) S^{(1)}(w):, \label{eq: Phi0 S OPE} \\ 
&\Phi^{(0)}(z) S^{(i)}(w)=:\Phi^{(0)}(z) S^{(i)}(w): \quad (i\geq 2), \\
&S^{(1)}(z)\Phi^{(0)}(w)=\frac{(qw/  z ; q)_\infty}{(q w/t z ; q)_\infty}:S^{(1)}(z)\Phi^{(0)}(w):,\\
&S^{(i)}(z)\Phi^{(0)}(w)=:S^{(i)}(z)\Phi^{(0)}(w): \quad (i\geq 2), \\
&\Phi^{(0)}(z)\Phi^{(0)}(w)=\frac{(qw/tz;q)_{\infty}}{(tw/z;q)_{\infty}} :\Phi^{(0)}(z)\Phi^{(0)}(w):,
\label{eq:  Phi0 Phi0 OPE}
\end{align}

\begin{align}
& S^{(i)}(z)S^{(i)}(w) = (1- w/z) \frac{(q w/ t z ; q)_\infty}{(t w/z ; q)_\infty}:S^{(i)}(z)S^{(i)}(w):\,,\\
& S^{(i)}(z)S^{(i+1)}(w) = \frac{(t w/  z ; q)_\infty}{( w/z ; q)_\infty}:S^{(i)}(z)S^{(i+1)}(w):, \\
& S^{(i+1)}(z)S^{(i)}(w)= \frac{(qw/  z ; q)_\infty}{(q w/t z ; q)_\infty}:S^{(i+1)}(z)S^{(i)}(w): \quad (\forall i),\label{eq: Si+1 Si OPE}\\
&S^{(i)}(z)S^{(j)}(w) = :S^{(i)}(z)S^{(j)}(w): \quad \text{for} \quad |i-j|>2, \label{eq: S S OPE1} 
\end{align}

\begin{align}
& \Lambda^{(1)}(z) \Phi^{(0)}(x) = \frac{1-x/z}{1-tx/z}:\Lambda^{(1)}(z) \Phi^{(0)}(x) :\,,\label{eq: Lam Phi0 OPE1}\\
&\Phi^{(0)}(x) \Lambda^{(1)}(z) = \frac{1-z/x}{1-qz/t^2x}:\Phi^{(0)}(x) \Lambda^{(1)}(z):\,, \\
& \Lambda^{(i)}(z) \Phi^{(0)}(x) = : \Lambda^{(i)}(z) \Phi^{(0)}(x):\,, \\ 
& \Phi^{(0)}(x) \Lambda^{(i)}(z) = \frac{1-z/tx}{1-qz/t^2x}:\Phi^{(0)}(x) \Lambda^{(i)}(z):\quad (i>1)\,, \label{eq: Lam Phi0 OPE2}
\end{align}

\begin{align}
&\Psi^+(z)S^{(i)}(w)=S^{(i)}(w)\Psi^+(z)=:\Psi^+(z)S^{(i)}(w): 
\quad (\forall i), \\
&\Psi^+(z)\Phi^{(0)}(w) = \frac{1-tw/qz}{1-w/z}:\Psi^+(z)\Phi^{(0)}(w):\,, \label{eq: Psi+ Phi0 OP}
\end{align}

\begin{align}
&A_{(s)}(x)\Lambda^{(i)}(z) 
= \prod_{k=1}^{r-1}\frac{1-t^{-k}z/x}{1-t^{-k-1}qz/x} :\Lambda^{(i)}(z) A_{(s)}(x):, \label{eq: A Lam OPE} \\
&A^{(r)}(x) \Phi^{(0)}(y)=\prod_{k=0}^{r-2}\frac{(t^{-k}qy/tx;q)_{\infty}}{(t^{-k}y/x;q)_{\infty}}:A^{(r)}(x) \Phi^{(0)}(y):, \label{eq: Ar Phi0}\\
&A^{(r)}(x) S^{(i)}(y)=:A^{(r)}(x) S^{(i)}(y):, \label{eq: Ar Si}
\end{align}
and with $\mathcal{G}(z) = \prod_{i,j=0}^\infty (1 - z q^i t^{-j})$,
\begin{align}
&\Phi_{\lambda}(v_i) \Phi^*_{\mu}(u_j) = \mathcal{G}(u_j/\gamma v_i)^{-1} N_{\mu \lambda}(u_j/\gamma v_i):\Phi_{\lambda}(v_i) \Phi^*_{\mu}(u_j):,
\\
 &\Phi^*_{\mu}(u_j)\Phi_{\lambda}(v_i)
= \mathcal{G}(v_i/\gamma u_j)^{-1} N_{ \lambda \mu}(v_i/\gamma u_j):\Phi^*_{\mu}(u_j)\Phi_{\lambda}(v_i):, \label{eq: Phi* Phi OPE}
\\
&\Phi_{\lambda^{(i)}}(v_i) \Phi_{\lambda^{(j)}}(v_j) 
= \frac{\mathcal{G}(v_j/\gamma^2 v_i)}{N_{\lambda^{(j)}\lambda^{(i)}}(v_j/\gamma^2 v_i)}:\Phi_{\lambda^{(i)}}(v_i) \Phi_{\lambda^{(j)}}(v_j):,
\\
&\Phi^*_{\mu^{(i)}}(u_i) \Phi^*_{\mu^{(j)}}(u_j) 
= \frac{\mathcal{G}(u_j/u_i)}{N_{\mu^{(j)}\mu^{(i)}}(u_j/u_i)}:\Phi^*_{\mu^{(i)}}(u_i) \Phi^*_{\mu^{(j)}}(u_j): .    
\end{align}

\subsection{On Nekrasov factors}

For a partitions $\lambda$ and non-negative integers $r, s \in \mathbb{Z}_{\geq 0}$, 
define $\mathsf{B}_{r,s}(\lambda)$ to be 
the partition obtained by removing the 1st to $s$-th rows 
and the 1st to $r$-th columns, i.e., 
\begin{align}
\mathsf{B}_{r,s}(\lambda) :=(\mathsf{P}(\lambda_{s+i}-r) )_{i\geq 1}, 
\quad \mathsf{P}(n)=
\begin{cases}
n, \quad n\geq 0;\\
0, \quad n<0. 
\end{cases}
\end{align}
For example, if $\lambda=(5,5,4,4,4,1,1)$, then 
$\mathsf{B}_{2,1}(\lambda)=(3,2,2,2)$. 
We have the following vanishing condition for the Nekrasov factor. 

\begin{lemma}\label{lem: nonzero cond of N}
If $m\geq 0$ and $n\leq 0$, then 
\begin{align}
N_{\lambda,\mu}(q^n t^m) \neq 0\quad \Longleftrightarrow \quad 
\mu \supset \mathsf{B}_{-n,m}(\lambda) .
\end{align}
If $m\leq -1$ and $n\geq 1$, then 
\begin{align}
N_{\lambda,\mu}(q^n t^m) \neq 0 \quad \Longleftrightarrow \quad 
\lambda \supset \mathsf{B}_{n-1,-m-1}(\mu).
\end{align}
\end{lemma}

Note that in particular, 
$N_{\emptyset,\mu}(t)\neq 0$ if and only if $\mu=(m)$ for some $m \in \mathbb{Z}_{\geq 0}$.

\subsection{Duality of functions $p_n(x;s|q,t)$ and $f_n(x;s|q,t)$}

\begin{lemma}\label{Lem_delta1}
Let $\vl$ and $\vm$ satisfy $\ell (\lambda^{(i)}), \ell(\mu^{(i)})\leq n_i$.  
\begin{align}
\left[ x^{-\vl}x^{\vm} f_{|\vnn|}(x;s(\vl)|q,q/t)p_{|\vnn|}(x;s(\vm)|q,q/t) \right]_{x, 1}=\delta_{\vl, \vm}.
\end{align}
Here, $s(\vl)=(s_j(\vl))_{1\leq j\leq |\vnn|}$, 
$s_{[i,k]_{\vnn}}(\vl)=q^{\lambda^{(i)}_{k}}t^{1-k}u_i$. 
\begin{proof}
We denote the LHS by $F(\vl|\vm)$.
Inserting the Macdonald operator $D_n^1(s;q,q/t)$ and integrating by parts give
 \begin{equation}
     \begin{split}
        &\left[ x^{-\vl} f_{|\vnn|}(x;s(\vl)|q,q/t) \left(D_{|\vnn|}^1(\bar{s};q,q/t) x^{\vm} p_{|\vnn|}(x;s(\vm)|q,q/t)\right) \right]_{x, 1}  \\
         &=\left[\left(\widetilde{D^1_{|\vnn|}}(\bar{s};q,q/t) x^{-\vl} f_{|\vnn|}(x;s(\vl)|q,q/t)  \right) x^{\vm} p_{|\vnn|}(x;s(\vm)|q,q/t) \right]_{x, 1}\,,
     \end{split}
 \end{equation}
 where $\bar{s}_{\lbl{i}{k}{n}} = u_i t^{1-k}$.
 The LHS gives $\epsilon_{\vm}F(\vl|\vm)$, and  the RHS gives $\epsilon_{\vl} F(\vl|\vm)$.
 Thus we can show $F(\vl|\vm) = C(\vl) \delta_{\vl, \vm}$ with the coefficient $C(\vl) = F(\vl|\vl)$.
 We can show $C(\vl) = 1$ because
 both $f_{|\vnn|}(x;s(\vl)|q,q/t)$ and $p_{|\vnn|}(x;s(\vm)|q,q/t)$ are in $\mathbb{C}(s_1,\ldots, s_{|\vnn|})[[x_2/x_1, x_3/x_2,\ldots , x_{|\vnn|}/x_{|\vnn|-1}]]$ and so is their product. 
\end{proof}
\end{lemma}

\subsection{Some formulas to prove (\ref{eq: final eq})}\label{App:Useful_Induction}

The coincidence between (\ref{eq: <K|til V|K>}) and (\ref{eq: final eq}) can be identified 
with the following equation. 

\begin{proposition}\label{prop: step3}
\begin{align}
&\frac{\mathcal{R}^{\vnn}_{\vl}}{\mathcal{R}^{\vmm}_{\vm}}
\mathsf{N}^{|\vnn|,|\vmm|}_{[\vm]}(s_1,\ldots,s_{|\vnn|+|\vmm|}) \prod_{1\leq i<j\leq|\boldsymbol{n}|}\frac{(q s_j/ts_i)_{[\boldsymbol{\lambda}]_i-[\boldsymbol{\lambda}]_j}}{(q s_j/s_i)_{[\boldsymbol{\lambda}]_i-[\boldsymbol{\lambda}]_j}}\nonumber 
\\
&=
\prod_{i=1}^N t^{-(|\vnn|+n_i)|\lambda^{(i)}|}
(-\gamma^{-2})^{|\mu^{(i)}|}(t^{n_i})^{-(N-i)|\mu^{(i)}|} t^{(|\lambda^{(i)}|-|\mu^{(i)}|)\sum_{s=1}^in_s}
\gamma^{-(N-1)(|\lambda^{(i)}|+|\mu^{(i)}|)}t^{2n(\lambda^{(i)})+|\lambda^{(i)}|}q^{n(\mu^{(i)'})}\nonumber 
\\
&\times
\prod_{i=1}^N \left(\frac{f_{\mu^{(i)}}}{f_{\lambda^{(i)}}}\right)^{N-i}
\prod_{1\leq i<j\leq N}(v_i/v_j)^{-|\lambda^{(i)}|+|\mu^{(i)}|}\nonumber 
\\
&\times
        \prod_{i=1}^N\frac{1}{c_{\lambda^{(i)}}c'_{\mu^{(i)}}}\frac{\prod_{ i, j=1 }^N N_{\lambda^{(i)}, \mu^{(j)}}(t^{n_j}v_i/v_j)}{\prod_{1\leq i< j\leq N}N_{\lambda^{(j)}\lambda^{(i)}}(q v_j/t v_i)\prod_{1\leq i< j\leq N}N_{\mu^{(i)}\mu^{(j)}}(qt^{-n_i+n_j}v_i/t v_j)}.\label{eq_Step3}
\end{align}
\end{proposition}

The following lemmas are formulas for proving (\ref{eq_Step3}). 

\begin{lemma}
With  $n = \ell(\lambda)$ and $m = \ell(\mu)$, the following holds: 
\begin{align}
        \frac{N_{\lambda\mu }(t^{n})}{c_{\lambda}c'_{\mu}} \nonumber
    &= (-t/q)^{|\mu|}t^{n(\mu)-2n(\lambda)-|\lambda|}(t^{n})^{|\lambda|+|\mu|} q^{-n(\mu')}\\
        &\times\prod_{1\leq i<j\leq n}\frac{(q s_j/t s_i)_{\lambda_i-\lambda_j}}{(q s_j/ s_i)_{\lambda_i-\lambda_j}}\prod_{i=1}^{m}\frac{(q/t)_{\mu_{i}}}{(q)_{\mu_i}}\prod_{j=1}^{m}\prod_{i=1}^{n+j-1}
    \frac{(q s_{j+n}/t s_i)_{\mu_j}}{(q s_{j+n}/s_i)_{\mu_j}}\prod_{1\leq i<j\leq m}\frac{(q\sigma_i/t\sigma_j)_{\mu_j}}{(q\sigma_i/\sigma_j)_{\mu_j}}\,,
\end{align}
where
\begin{equation}
    s_{i} = q^{\lambda_i}t^{1-i}\,\, (i=1,\dots,n)\,,\quad s_{n+i} = t^{1-n-i}\,\, (i=1,\dots,m)\,,\quad\text{and}\quad \sigma_i = q^{\mu_i}t^{1-i}\,.\nonumber
\end{equation}
\end{lemma}

The equality that we obtain by removing the factors of this type from the both sides of (\ref{eq_Step3}), can be shown by using following relations.

\begin{lemma}
Under the $\lambda_i \to \lambda_i+1$ or $\mu_i \to \mu_i +1$,
we also have the following induction relations,
\begin{align}
	&\frac{N_{\lambda+1_i, \mu}(u)}{N_{\lambda, \mu}(u)} =(1-ut^{\ell(\mu)}\chi_x)\prod_{j=1}^{\ell(\mu)}
	\frac{1-u q^{-\mu_j}t^{j-1}\chi_x}{1-u q^{-\mu_j}t^j\chi_x},\\
	&\frac{N_{\lambda, \mu+1_i}(u)}{N_{\lambda, \mu}(u)} = (1- t^{1-\ell(\lambda)} u/q \chi_y)\prod_{i=1}^{\ell(\lambda)}\frac{1-u q^{\lambda_i-1}t^{2-i}/ \chi_y}{1-u q^{\lambda_i-1}t^{1-i}/ \chi_y}\,,
\end{align}
where
\begin{equation}
    \chi_x = q^{\lambda_i}t^{1-i}\,,\quad \chi_y = q^{\mu_i}t^{1-i}\,.
\end{equation}
\end{lemma}

\begin{notation}
In what follows in this proof, we set
\begin{align}
&s_{\lbl{i}{k}{n}} =q^{\lambda^{(i)}_k} t^{1-k}v_i 
\quad (1 \leq k \leq n_i, \, i=1,\ldots, N), \\
&s_{|\vnn| + \lbl{i}{k}{m}} = t^{1- n_{i} -k}v_i 
\quad (1 \leq k \leq m_i, \, i=1,\ldots, N),\\
&\sigma_{\lbl{i}{k}{m}} =  q^{\mu^{(i)}_k} t^{1-k}v_i
\quad (1 \leq k \leq m_i, \, i=1,\ldots, N).
\end{align}
\end{notation}

\begin{definition}
We set
\begin{align}
        &\hat{\mathsf{N}}^{|\vnn|,|\vmm|}_{[\vm]}(\boldsymbol{s}) := \prod_{i=1}^N t^{(2\sum_{j=1}^{i-1}n_j+n_i - |\vnn|)|\lambda^{(i)}|} \nonumber \\
        &\qquad\times\prod_{l=1}^{N} \prod_{k=1}^{m_l} \left(\prod_{j=1, j\neq l}^{N}\prod_{i=1}^{n_j} \frac{(qs_{\lbl{l}{k}{m}}/ts_{\lbl{j}{i}{n}};q)_{\mu_k}}{(qs_{\lbl{l}{k}{m}}/s_{\lbl{j}{i}{n}};q)_{\mu_k}}\prod_{j=1}^{l-1}\prod_{i=1}^{m_j}\frac{(qs_{\lbl{l}{k}{m}}/ts_{\lbl{j}{i}{m}};q)_{\mu_k}}{(qs_{\lbl{l}{k}{m}}/s_{\lbl{j}{i}{m}};q)_{\mu_k}}
        \right)\nonumber \\
        &\qquad\times \prod_{1\leq k<l\leq N} \prod_{\substack{1 \leq i \leq m_k \\ 1 \leq j \leq m_l}}\frac{(qt^{-n_k+n_l}\sigma_{\lbl{k}{i}{m}}/t\sigma_{\lbl{l}{j}{m}})_{\mu^{(l)}_j}}{(qt^{-n_k+n_l}\sigma_{\lbl{k}{i}{m}}/\sigma_{\lbl{l}{j}{m}})_{\mu^{(l)}_j}}\prod_{1\leq k<l\leq N} 
        \prod_{\substack{1 \leq i \leq n_k \\ 1 \leq j \leq n_l}} 
        \frac{(qs_{\lbl{l}{j}{n}}/t s_{\lbl{k}{i}{n}};q)_{-\lambda^{(l)}_j+\lambda^{(k)}_i}}
        {(qs_{\lbl{l}{j}{n}}/s_{\lbl{k}{i}{n}};q)_{-\lambda^{(l)}_j+\lambda^{(k)}_i}}\,.
\end{align}
\end{definition}

\begin{lemma}
Under $\lambda^{(i)}_k \to \lambda^{(i)}_k +1$, that is, $s_{\lbl{i}{k}{n}}\to qs_{\lbl{i}{k}{n}}$, we have
\begin{equation}
   \frac{ \gamma^{-\sum_{j=1}^N (j-1)(|\lambda^{(j)}|+\delta_{j,i})}\mathcal{R}^{\vnn}_{\vl+1^{(i)}_k}(\vv)}{ \gamma^{-\sum_{j=1}^N (j-1)|\lambda^{(j)}|}\mathcal{R}^{\vnn}_{\vl}(\vv)} = \prod_{j=1}^{i-1}\frac{1-q^{-\lambda^{(i)}_k} t^{-n_j+k-1}/v_{ij}}{1-q^{-\lambda^{(i)}_k-1}t^{-n_j + k}/v_{ij}}\,,
\end{equation}
with $v_{ij} = v_i/v_j$,
and
\begin{align}
\frac{\hat{\mathsf{N}}^{|\vnn|,|\vmm|}_{[\vm]}(\ldots,qs_{\lbl{i}{k}{n}},\ldots)}{\hat{\mathsf{N}}^{|\vnn|,|\vmm|}_{[\vm]}(\boldsymbol{s})} 
&=\prod_{l=1}^{i-1}\frac{1-t^{-n_l-m_l}v_l/s_{\lbl{i}{k}{n}}}{1-t^{-n_l} v_l/s_{\lbl{i}{k}{n}}}
        \prod_{l=i+1}^N \frac{1-t^{-n_l-m_l}v_l/s_{\lbl{i}{k}{n}}}{1-t^{-n_l} v_l/s_{\lbl{i}{k}{n}}}\nonumber
        \\
        &\times\prod_{l=1, l\neq i}^N \prod_{j=1}^{m_l}\frac{1-q^{\mu_j} s_{\lbl{l}{j}{m}}/s_{\lbl{i}{k}{n}}}{1-q^{\mu_j} s_{\lbl{l}{j}{m}}/ t s_{\lbl{i}{k}{n}}}\nonumber
        \\
        &\times 
\left( 
\prod_{j=1}^{i-1}\prod_{1\leq l\leq n_j}\frac{1-s_{\lbl{j}{l}{n}}/qs_{\lbl{i}{k}{n}}}{1-ts_{\lbl{j}{l}{n}}/qs_{\lbl{i}{k}{n}}} 
\right) 
\left(
\prod_{j=i+1}^{N}\prod_{1\leq l\leq n_j}\frac{1-ts_{\lbl{i}{k}{n}}/s_{\lbl{j}{l}{n}}}{1-s_{\lbl{i}{k}{n}}/s_{\lbl{j}{l}{n}}}\right)\,.
\end{align}
\end{lemma}

\begin{lemma}
Similarly, under $\mu^{(i)}_k \to \mu^{(i)}_k +1$, that is, $\sigma_{\lbl{i}{k}{n}}\to q\sigma_{\lbl{i}{k}{n}}$, we have
\begin{align}
        \frac{\hat{\mathsf{N}}^{|\vnn|,|\vmm|}_{[\vm+1^{(i)}_k]}(\boldsymbol{s})}{\hat{\mathsf{N}}^{|\vnn|,|\vmm|}_{[\vm]}(\boldsymbol{s})} &=\prod_{j=1, j\neq i}^{N}\prod_{l=1}^{n_j} 
\frac{1-q^{\mu^{(i)}_k+1}s_{\lbl{i}{k}{m}}/ts_{\lbl{j}{l}{n}}}
{1-q^{\mu^{(i)}_k+1}s_{\lbl{i}{k}{m}}/s_{\lbl{j}{l}{n}}}\nonumber \\
&\times
\prod_{j=1}^{i-1}\left(\frac{1-q^{\mu^{(i)}_k+1}t^{-k}t^{-n_i+n_j}v_{ij}}{1-q^{\mu^{(i)}_k+1}t^{m_j-k}t^{-n_i+n_j}v_{ij}}
\prod_{l=1}^{m_j}\frac{1-t^{-n_j+n_i}\sigma_{\lbl{j}{l}{m}}/t \sigma_{\lbl{i}{k}{m}}}{1-t^{-n_j+n_i}\sigma_{\lbl{j}{l}{m}}/ \sigma_{\lbl{i}{k}{m}}}\right)\nonumber
\\
&\times
\prod_{j=i+1}^{N}\left(\frac{1-q t^{-n_i+n_j-1}v_{ij}\chi_y}{1-q t^{-n_i+n_j}t^{m_j-1}v_{ij}\chi_y}\prod_{l=1}^{m_j}\frac{1-q t^{-n_i +n_j}\sigma_{\lbl{i}{k}{m}}/\sigma_{\lbl{j}{l}{m}}}{1-q t^{-n_i +n_j}\sigma_{\lbl{i}{k}{m}}/t\sigma_{\lbl{j}{l}{m}}}
\right)\,.
\end{align}
\end{lemma}

Combining these identities, we complete the proof of (\ref{eq_Step3}).

\section{Some Proofs of Lemmas and Propositions}\label{App: proofs}

\subsection{Proof of Proposition \ref{prop: commutativity of S}}\label{sec: pf of screening}

By the operator product formulas  (\ref{eq: Lam S OPE 1})-(\ref{eq: Lam S OPE 2}), 
the operator $\Lambda^{(j)}(z)$ with $j\neq i,i+1$ does not contribute in the commutation relation,  
and it can be shown that 
\begin{align}
\left[:\Lambda^{(i)}(z)\Lambda^{(i+1)}(\gamma^{-2}z):, S^{(i)}(w)\right]=0. 
\end{align} 
Hence, it is enough to consider the relation only with $u_i\Lambda^{(i)}(z)+u_{i+1}\Lambda^{(i+1)}(z)$. 

We have 
\begin{align}
\Lambda^{(i)}(z)S^{(i)}(w)-S^{(i)}(w)\Lambda^{(i)}(z)t=(1-t)\delta\left( \frac{tw}{qz}\right)
:\Lambda^{(i)}(tw/q)S^{(i)}(w):,
\end{align}
and 
\begin{align}
\Lambda^{(i+1)}(z)S^{(i)}(w)-S^{(i)}(w)\Lambda^{(i+1)}(z)t^{-1}
&=(1-t^{-1})\delta\left( \frac{tw}{z}\right) :\Lambda^{(i+1)}(tw)S^{(i)}(w):\nonumber \\
&=(1-t^{-1})\delta\left( \frac{tw}{z}\right) :\Lambda^{(i)}(tw)S^{(i)}(qw):. 
\end{align}
Therefore, by the property $g_i(qz)=\frac{u_{i+1}}{tu_i}g_i(z)$ with respect to the $q$-difference, 
we obtain 
\begin{align}
&\left(u_i\Lambda^{(i)}(z)+u_{i+1}\Lambda^{(i+1)}(z)\right)S^{(i)}(w)g_i(w)\nonumber \\
&\quad -S^{(i)}(w)\left(tu_i\Lambda^{(i)}(z)+t^{-1}u_{i+1}\Lambda^{(i+1)}(z)\right)g_i(w)\nonumber \\
&=(t-1)u_i(T_{q,w}-1)
 \delta\left( \frac{tw}{qz} \right):\Lambda^{(i)}(tw/q)S^{(i)}(w):g(w). \label{eq: comm rel of lam and S}
\end{align}

\qed

\subsection{Proof of Lemma \ref{lem: rel of X and Phi}}
\label{sec: pf rel of X and Phi}

First, we show the relation for $k=0$. 
In this proof, we write 
\begin{align}
\Lambda^{(i_1,\ldots, i_r)}(z)=
:\Lambda^{(i_1)}(z) \cdots \Lambda^{(i_r)}((q/t)^{r-1}z): 
\end{align}
for $i_1<\cdots < i_r$. 
By the operator products (\ref{eq: Lam Phi0 OPE1})-(\ref{eq: Lam Phi0 OPE2}) 
and the relation $:\Phi^{(0)}(w)\Lambda^{(1)}(tw):=\Phi^{(0)}(qw)\Psi^+(w)$, 
it can be shown that if $i_1=1$, 
\begin{align}
&\Lambda^{(i_1,\ldots, i_r)}(z) \Phi^{(0)}(x)
-  t^{-1} \frac{1-(q/t)^rz/tx}{1-z/tx} 
\Phi^{(0)}(x) \Lambda^{(i_1,\ldots, i_r)}(z)\nonumber \\
&=(1-t^{-1})\delta(tx/z)
:\Lambda^{(i_2)}((q/t)tx) \cdots \Lambda^{(i_r)}((q/t)^{r-1}tx)\Phi^{(0)}(qx)\Psi^+(x):. 
\end{align}
If $i_1\geq 2$, 
\begin{align}
 \Lambda^{(i_1,\ldots, i_r)}(z) \Phi^{(0)}(x)
- \frac{1-(q/t)^rz/tx}{1-z/tx} 
\Phi^{(0)}(x) \Lambda^{(i_1,\ldots, i_r)}(z)
=0. 
\end{align}
Thus, we obtain the relation in the case $k=0$:
\begin{equation}
X^{(r)}(z) \Phi^{(0)}(x) - 
 \frac{1-(q/t)^rz/tx}{1-z/tx} \Phi^{(0)}(x)X^{(r)}(z)
= u_1 (1-t^{-1})\delta(tx/z) Y^{(r)}(x)\Phi^{(0)}(qx)\Psi^+(x)\,.
\end{equation}
Applying the screening operators to this relation from the right side, 
we have the case $k\neq 0$. 
Indeed, $\Psi^+(x)$ commutes with $S^{(i)}(y)$:
\begin{align}\label{eq: psi+ S}
\Psi^+(z) S^{(i)}(y) =  S^{(i)}(y)\Psi^+(z)\,. 
\end{align}
Noting that we have 
\begin{align}\label{eq: Tqx g}
T_{q,x}g(x,y_1,\ldots,y_k)=\frac{u_{1}}{u_{k+1}}g(x,y_1,\ldots,y_k), 
\end{align}
and by virtue of (\ref{eq: psi+ S}) and 
commutativity of the screening operators, 
we can establish the relation for general $k$. 

We can show the commutativity of the screening operators  as follows. 
First, it is clear that 
\begin{align}
&\Phi^{(0)}(x)\cdot \left[ X^{(r)}(z), S^{(1)}(y_1)\cdots S^{(k)}(y_k)g(x,y_1,\ldots,y_k)\right]\nonumber \\
&=\sum_{i=1}^{k}\Phi^{(0)}(x)S^{(1)}(y_1) \cdots \left[X^{(r)}(z),S^{(i)}(y_i) \right]\cdots S^{(k)}(y_k)g(x,y_1,\ldots,y_k). 
\label{eq: comm of X and SSS}
\end{align}
By calculating the commutation relation as in the proof of Proposition \ref{prop: commutativity of S}, 
the RHS of (\ref{eq: comm of X and SSS}) consists of terms as 
\begin{align}
&(1-T_{q,y_i})\delta\left(\frac{t\gamma^{2\ell}y_i}{qz} \right)
\Phi^{(0)}(x)S^{(1)}(y_1) \cdots \nonumber \\ 
&\quad \cdots :\Lambda^{(j_1,\ldots,j_{\ell}, i,j_{\ell+2} \ldots, j_r)}(t \gamma^{2\ell} y_i/q) \, S^{(i)}(y_i):
\cdots S^{(k)}(y_k)g(x,y_1,\ldots,y_k). 
\end{align}
Note that $j_{\ell+2}\geq i+2$. 
Let us investigate the positions of poles in $y_i$. 
Combining the $\theta$-functions containing $y_i$ in $g(x,y_1,\ldots,y_k)$
and the operator products among screening currents and $\Phi^{(0)}(x)$, 
there appears the factor 
\begin{align}
&\frac{1}{\theta_{q}(ty_i/y_{i-1})\theta_{q}(ty_{i+1}/y_{i})} 
 \frac{(ty_i/y_{i-1};q)_{\infty}(ty_{i+1}/y_{i};q)_{\infty}}{(y_i/y_{i-1};q)_{\infty}(y_{i+1}/y_{i};q)_{\infty}}\nonumber \\
&=\frac{1}{(qy_{i-1}/ty_{i};q)_{\infty}(qy_{i}/ty_{i+1};q)_{\infty}}
\frac{1}{(y_{i}/y_{i-1};q)_{\infty}(y_{i+1}/y_{i};q)_{\infty}}, \quad y_0:=x.
\end{align} 
From $\Phi^{(0)}(x)$ and $\Lambda^{(i)}(ty_1/q)$ 
in $\Lambda^{(j_1,\ldots,j_{\ell}, i,j_{\ell+2} \ldots, j_r)}$, 
we have
\begin{align}
\frac{1-t^{\delta_{i,1}-1}ty_i/qx}{1-y_i/tx}.
\end{align}
Noticing that for $i\geq 2$, 
the operator product of $S^{(i-1)}(y_{i-1})$ and $\Lambda^{(i)}(ty_i/q)$ are 
\begin{align}
S^{(i-1)}(y_{i-1})\Lambda^{(i)}(ty_i/q)= 
 \frac{1-ty_i/qy_{i-1}}{1-y_i/qy_{i-1}}:S^{(i-1)}(y_{i-1})\Lambda^{(i)}(ty_i/q):, 
\end{align}
we have the following set of poles of (\ref{eq: comm of X and SSS}) in $y_i$: 
\begin{align}
&y_i=tx,\label{large pole1}\\
&y_i= t^{-1}q^{2+n}y_{i-1}, \quad y_i=q^ny_{i+1}, \label{small pole1}\\
&y_i=t\,q^{-1-n}y_{i+1} \quad (i \geq 1,\, n=0,1,2,\dots),\label{large pole2}
\end{align}
and 
\begin{align}
&y_i=q^{-n+1}y_{i-1} \quad \text{for $i\geq 2$},\label{small pole2} \\
&y_i=q^{-n}x \quad \text{for $i=1$} \quad (n=0,1,2\dots). \label{small pole3}
\end{align}
If $r=1$, it gives us the all poles. 
In general, this list does not exhaust the possible poles. 
In case $r\geq 2$, 
from $S^{(j)}(y_j)$ and $\Lambda^{(j_m)}$ in 
$\Lambda^{(j_1,\ldots,j_{\ell}, i,j_{\ell+2} \ldots, j_r)}$ with $m\neq \ell+1$, 
we have extra poles.  
From the operator product formulas for them, we have 
\begin{align}
\prod_{m=1}^{\ell} \frac{1-t^{-1}(q/t)^{-\ell+m-1}y_{i}/y_{j_m}}{1-(q/t)^{-\ell+m-1}y_{i}/y_{j_m}}
\frac{1-(q/t)^{-\ell+m-2}y_{i}/y_{j_m-1}}{1-q^{-1}(q/t)^{-\ell+m-1}y_{i}/y_{j_m-1}} \nonumber\\ 
\times \prod_{m=\ell+2}^{r} \frac{1-t(q/t)^{\ell-m+1}y_{j_m}/y_{i}}{1-(q/t)^{\ell-m+1}y_{j_m}/y_{i}}
\frac{1-(q/t)^{\ell-m+2}y_{j_m-1}/y_{i}}{1-q(q/t)^{\ell-m+1}y_{j_m-1}/y_{i}}. 
\end{align}
In addition, 
from $\Phi^{(0)}(x)$ and $\Lambda^{(j_m)}$, the following factor arises: 
\begin{align}
\prod_{m\neq \ell+1}\frac{1-t^{\delta_{j_m,1}-1} (q/t)^{-\ell+m-2} y_i/x}
{1-t^{-1} (q/t)^{-\ell+m-1} y_i/x}. 
\end{align}
Summarizing these, 
we can show that the poles in $y_i$ are in the following positions 
(Though not all following points are poles, 
all poles should be in the followings or (\ref{large pole1})-(\ref{small pole3})). 
For $i\geq 1$, 
\begin{align}
&y_i=(q/t)^{-n-1} y_{j+1}, \quad y_i=q (q/t)^{-n-1}y_j \quad (j>i), \label{large pole3}\\
&y_i=q (q/t)^{-n} x, \quad (n=0,1,2\ldots), \label{large pole4}
\end{align}
and for $i\geq 2$, 
\begin{align}
&y_i=(q/t)^{n+1} y_j \quad (1\leq j<i ), \quad y_i=q (q/t)^{n+1}y_{j-1} \quad (2\leq j<i ),\label{small pole4}\\
&y_i=q(q/t)^{n+1}x \quad (n=0,1,2\ldots). \label{small pole5}
\end{align}
For the given integration contour, 
the poles (\ref{small pole1}), (\ref{small pole2}), (\ref{small pole3}), (\ref{small pole4}) 
and (\ref{small pole5}) are in the disk $\{z; |z|<|qy_i|  \}$. 
On the other hand, 
the poles  (\ref{large pole1}), (\ref{large pole2}), (\ref{large pole3}) and (\ref{large pole4}) are 
in $\{ z; |z|>|y_i| \}$. 
Therefore, the change of variable $y_i \rightarrow q y_i$ is not affected by these poles, 
and the commutation relation (\ref{eq: comm of X and SSS}) becomes zero after the integrals. 
\qed

\subsection{Proof of Proposition \ref{Prop_R}}\label{Proof_R}

By taking the constant terms of  $V^{(\vnn)}(x_1,\ldots,x_{|\vnn|})\ketzero$ with respect to $x_i$, 
the proportional constant 
$\mathcal{R}^{\vnn}_{\vl}$ is calculated as the expansion coefficient in front of $\ket{Q_{\vl}}$ 
in the basis of generalized Macdonald functions. 
We first consider only the operators that contain the creation operators $a^{(N)}_{-n}$'s 
with respect to the $N$-th Fock space. 
That is, we take the constant terms of 
\begin{align}
&\prod_{i=1}^{n_N} y_{i,0}^{-\lN_i} 
\times 
\Phi^{(N-1)}(y_{1,0})\cdots \Phi^{(N-1)}(y_{n_N,0})\ketzero \nonumber \\ 
&=\prod_{i=1}^{n_N} y_{i,0}^{-\lN_i} 
\oint \prod_{\substack{1\leq i\leq n_N\\ 1\leq m\leq N-1} } \frac{dy_{i,m}}{2 \pi \sqrt{-1}y_{i,m}} 
\sum_{(r_{i,m}) \in \widetilde {\mathrm{M}} }
\prod_{i=1}^{n_N}\prod_{m=1}^{N-1}R_{r_{i,m}}(\alpha^{(N)}_{i,m})\left(\frac{y_{i,m}}{y_{i,m-1}}\right)^{r_{i,m}} \nonumber \\
&\quad \times :\Phi^{(0)}(y_{1,0})S^{(1)}(y_{1,1})\cdots S^{(N-1)}(y_{1,N-1}): 
\times \cdots \nonumber \\ 
&\qquad\qquad\qquad \qquad\qquad
\cdots \times 
:\Phi^{(0)}(y_{n_N,0})S^{(1)}(y_{n_N,1})\cdots S^{(N-1)}(y_{n_N,N-1}):\ketzero. \label{eq: pf of R 1}
\end{align}
Here, we used the expansion formula (\ref{eq: expand scr. vetex}), 
and $\widetilde {\mathrm{M}}=\mathrm{Mat}(n_N,N-1;\mathbb{Z})$ is the 
set of $n_N \times (N-1)$ matrices with integral entries. 
We denote $x_{[N,i]_{\vnn}}$ by $y_{i,0}$ in this proof. 
Further we set 
\begin{align}
&\alpha^{(k)}_{i,l}=t^{-n_l+i}\frac{u_l}{u_k}, \\
&R_r(\alpha)=\frac{(\alpha;q)_r}{(q\alpha/t;q)_r}. 
\end{align}
Let $C(z)=\sum_{k\geq 0} C_k z^k$, $\widetilde C(z)=\sum_{k\geq 0} \widetilde C_kz^k$ and 
$C^{(\pm)}(z)=\sum_{k\geq 0} C^{(\pm)}_kz^k$ 
be the formal power series defined by 
\begin{align}
& C(z) =\frac{(q z/ t ; q)_\infty}{(t z; q)_\infty}, \quad 
 \widetilde C(z) = (1- z) \frac{(q z/ t ; q)_\infty}{(t z; q)_\infty}, \\
&C^{(+)}(z) = \frac{(t z ; q)_\infty}{(z ; q)_\infty}, \quad 
C^{(-)}(z) = \frac{(qz  ; q)_\infty}{(q z/t ; q)_\infty}. 
\end{align}
These series correspond to the operator product formulas among 
$\Phi^{(0)}$ and $S^{(i)}$'s. 
Moreover, we write 
\begin{align}
&E^{(m)}_i(e)=-\sum_{s=i+1}^{n_N} e^{(m)}_{i,s}+\sum_{s=1}^{i-1}e_{s,i}^{(m)},\\
&K^{(m)}_i(k,\ell)=-\sum_{s=i+1}^{n_N} \ell^{(m-1)}_{i,s}+\sum_{s=1}^{i-1}k_{s,i}^{(m)},\\ 
&L^{(m)}_i(k,\ell)=-\sum_{s=i+1}^{n_N} k^{(m)}_{i,s}+\sum_{s=1}^{i-1}\ell_{s,i}^{(m-1)}
\end{align}
for $e=((e^{(m)}_{i,j})_{i,j=1}^{n_N})_{1\leq m\leq N-1}, 
k=((k^{(m)}_{i,j})_{i,j=1}^{n_N})_{1\leq m\leq N-1} \in M_{n_N}^{N-1}$, and 
$\ell=((\ell^{(m)}_{i,j})_{i,j=1}^{n_N})_{0\leq m\leq N-2}\in M_{n_N}^{N-1}$. 
Here $M_n$ is the set of strictly upper triangular $n\times n$ matrices with nonnegative integral entries. 
With these notations, (\ref{eq: pf of R 1}) can be rewritten as 
\begin{align}
&\prod_{i=1}^{n_N} y_{i,0}^{-\lN_i} \cdot 
\oint \prod_{\substack{1\leq i\leq n_N\\ 1\leq m\leq N-1} } \frac{dy_{i,m}}{2 \pi \sqrt{-1}y_{i,m}} 
\sum_{(r_{i,m}) \in \widetilde {\mathrm{M}} }
\prod_{i=1}^{n_N}\prod_{m=1}^{N-1}R_{r_{i,m}}(\alpha^{(N)}_{i,m})\left(\frac{y_{i,m}}{y_{i,m-1}}\right)^{r_{i,m}} \nonumber \\
&\quad \times 
\prod_{1\leq i<j\leq n_N} C\left(\frac{y_{j,0}}{y_{i,0}}\right) 
 \prod_{m=1}^{N-1} C^{(-)}\left(\frac{y_{j,m-1}}{y_{i,m}}\right) 
\widetilde C\left(\frac{y_{j,m}}{y_{i,m}}\right) 
C^{(+)}\left(\frac{y_{j,m}}{y_{i,m-1}}\right) \nonumber \\
& \quad \times :\prod_{i=1}^{n_N} \Phi^{(0)}(y_{i,0})S^{(1)}(y_{i,1})\cdots S^{(N-1)}(y_{i,N-1}): \ketzero \nonumber \\
&=
\oint \prod_{\substack{1\leq i\leq n_N\\ 1\leq m\leq N-1} } \frac{dy_{i,m}}{2 \pi \sqrt{-1}y_{i,m}} 
\sum_{\substack{(r_{i,m}) \in \widetilde {\mathrm{M}} \\ e, \ell, k \in M_{n_N}^{N-1}\\ 
(\sfa^{(m)}_i) \in  \widetilde{\mathrm{M}}_{\geq 0} }}
\prod_{i=1}^{n_N} y_{i,0}^{-\lN_i+E^{(0)}_{i}(e)+K_i^{(1)}(k,\ell)-r_{i,1}+\sfa^{(0)}_i}\nonumber\\ 
& \qquad  \qquad  \times 
\prod_{m=1}^{N-2} \prod_{i=1}^{n_N} y_{i,,m}^{E^{(m)}_i(e) +K^{(m+1)}_i(k,\ell) +L^{(m)}_i(k,\ell) -r_{i,m+1}+r_{i,m}+\sfa^{(m)}_i}\nonumber\\
& \qquad  \qquad \times 
\prod_{i=1}^{n_N} y_{i,N-1}^{E^{(N-1)}_{i}(e)+L_i^{(N-1)}(k,\ell)+r_{i,N-1}+\sfa^{(N-1)}_i}\nonumber \\ 
&\qquad  \qquad  \times 
\prod_{i=1}^{n_N}\prod_{m=1}^{N-1}R_{r_{i,m}}(\alpha^{(N)}_{i,m})\cdot 
\prod_{1\leq i<j\leq n_N} C_{e^{(0)}_{i,j}}
 \prod_{m=1}^{N-1}  
C^{(-)}_{k^{(m)}_{i,j}}\widetilde C_{e^{(m)}_{i,j}}  
C^{(+)}_{\ell^{(m-1)}_{i,j}}\nonumber \\ 
& \qquad \qquad \times 
:\prod_{i=1}^{n_N} \Phi^{(0)}_{-\sfa^{(0)}_i} S^{(1)}_{-\sfa^{(1)}_i}\cdots S^{(N-1)}_{-\sfa^{(N-1)}_i}: \ketzero. 
\end{align}
Here, $ \widetilde{\mathrm{M}}_{\geq 0}=\mathrm{Mat}(n_N, N-1;\mathbb{Z}_{\geq 0})$,  
$\Phi^{(0)}(z)=\sum_{n \in \mathbb{Z}} \Phi^{(0)}_{n}z^{-n}$, and 
$S^{(i)}(z)=\sum_{n \in \mathbb{Z}} S^{(i)}_{n}z^{-n}$. 
Since the integral gives us the constant terms in $y_{i,m}$, 
we have 
\begin{align}
r_{i,N-1}&=-E^{(N-1)}_i-L^{(N-1)}_i-\sfa^{(N-1)}_i, \\
r_{i,N-2}&=-E^{(N-1)}_i-L^{(N-2)}_i-K^{(N-1)}_i-\sfa^{(N-2)}_i+r_{i,N-1} \nonumber\\
&=-E^{(N-1)}_i-E^{(N-2)}_i-L^{(N-1)}_i-L^{(N-2)}_i-K^{(N-1)}_i-\sfa^{(N-1)}_i-\sfa^{(N-2)}_i,\\ 
& \cdots\nonumber\\
r_{i,1}&=-\sum_{m=1}^{N-1}\left( E_i^{(m)} +L_i^{(m)} +K_i^{(m)} +\sfa^{(m)}_i \right)+K_i^{(1)},
\end{align}
and 
\begin{align}
\sfa^{(0)}_i+E^{(0)}_i+\sum_{m=1}^{N-1}\left( E_i^{(m)} +L_i^{(m)} +K_i^{(m)} +\sfa^{(m)}_i \right)-\lN_i=0. 
\end{align}
Since $\sum_{i=1}^{n_N} E^{(m)}_i=0$ and $\sum_{i=1}^{n_N} (L_i^{(m)} +K_i^{(m)}) =0$, 
it is shown that 
\begin{align}
|\lN|-\sum_{i=1}^{n_N}\sfa^{(N-1)}_i=\sum_{m=0}^{N-2}\sum_{i=1}^{n_N}\sfa^{(m)}_i\geq 0. 
\end{align}
Therefore,  
$\sum_{i=1}^{n_N}\sfa_i^{(N-1)}$ takes its maximum value $|\lN|$ 
when $\sfa^{(m)}_i=0$ for all $i$ and $m \leq N-2$. 
Since only the operators $:\prod_{i=1}^{n_N}S^{(N-1)}_{-\sfa^{(N-1)}_i}:$
have the creation operators acting on the $N$-th Fock space, 
it is clear that the maximum degree in the $N$-th Fock component is $|\lN|$.

Before taking expansion coefficients in the basis of generalized Macdonald functions, 
we investigate the one in the basis of products of ordinary Macdonald functions 
$\prod_{i=1}^NQ_{\mu^{(i)}}(a^{(i)}_{-n})\ketzero$. 
Consider the terms of level  $|\lN|$ with respect to the $N$-th Fock space. 
Then, $\sfa^{(N-1)}_{i}$ satisfies 
\begin{align}
\sfa^{(N-1)}_i=\lN_i-E^{(0)}_i-\sum_{m=1}^{N-1}\left( E_i^{(m)} +L_i^{(m)} +K_i^{(m)} \right). 
\end{align}
Furthermore, 
by the form of $E^{(m)}_i$, $K^{(m)}_i$ and $L^{(m)}_i$, 
it can be seen that only the following vectors appear: 
\begin{align}\label{eq: pf of R 2}
:\prod_{i=1}^{n_N}S^{(N-1)}_{-\mu_i}:\ketzero,\quad  \mu \geq \lN. 
\end{align}
By  the Pieri formula (Fact \ref{fact: Pieri}), 
we can write the terms of level $|\lambda^{(N)}|$ with respect to the $N$-th Fock space as 
\begin{align}
:\prod_{i=1}^{n_N}S^{(N-1)}_{-\mu_i}:\ketzero \Big/ \left< a^{(N-1)}_n\Big| n \in \mathbb{Z} \right>
&=\gamma^{(N-1)|\mu|} \prod_{i=1}^{n_N}g^{(N)}_{\mu_i} \ketzero, \nonumber\\
&=\gamma^{(N-1)|\mu|} Q_{\mu}(a^{(N)}_{-n})\ketzero+\sum_{\rho>\mu} (\mathrm{const.}) Q_{\rho}(a^{(N)}_{-n})\ketzero. \label{eq: pf of R 3}
\end{align}
Here, $g^{(N)}_n$ is defined by 
\begin{align}
\sum_{n\geq 0} z^n g^{(N)}_{n}= 
:\exp\left( \sum_{n>0} \frac{1-t^n}{1-q^n} a^{(N)}_{-n}z^n\right):. 
\end{align}
Therefore, 
there appears $Q_{\lN}(a^{(N)}_{-n})\ketzero$ only in the case that 
$\mu=\lN$ on (\ref{eq: pf of R 2}), 
i.e., the case that 
\begin{align}
e^{(m)}_{i,j}=\ell^{(m)}_{i,j}=k^{(m)}_{i,j}=0
\end{align}
for all $i, j, m$. 
Then
\begin{align}
r_{i,m}=-\sfa^{(N-1)}_i=-\lN_i 
\end{align}
for all $i, m$.

From the above discussion, 
we have 
\begin{align}
&x^{-\vl}V^{(\vnn)}(x_1,\ldots, x_{|\vnn|})\ketzero\Big| _{\mbox{const. of } x_i}\nonumber \\
&=
\prod_{k=1}^{N-1}\prod_{i=1}^{n_k}x_{[k,i]_{\vnn}}^{-\lambda^{(k)}_i} \,
\Phi^{(0)}(x_1)\cdots  \Phi^{(N-2)}(x_{[N,0]_{\vnn}})
\Bigg\{ \gamma^{(N-1)|\lN|}\prod_{i=1}^{n_N}\prod_{m=1}^{N-1} R_{-\lN_i} (\alpha^{(N-1)}_{i,m} ) 
Q_{\lN}(a^{(N)}_{-n})\ketzero \nonumber \\
&\quad +\sum_{\mN>\lN}(\mathrm{const.}) Q_{\mN}(a^{(N)}_{-n})\ketzero
+\mathcal{O}\left(\substack{|\mN|<|\lN| \\ \mbox{and } |\vm|=|\lN| }\right) \Bigg\}\Bigg|_{\mbox{const. of } x_i}. 
\end{align}
Here $\mathcal{O}\left( P \right)$ expresses the terms 
$\prod_{i=1}^N Q_{\mu^{(i)}}(a^{(i)}_{-n})\ketzero$ with $\vm$ satisfying the proposition $P$. 
By repeating the similar argument $N-1$ times, 
we obtain 
\begin{align}
&x^{-\vl}V^{(\vnn)}(x_1,\ldots, x_{|\vnn|})\ketzero\Big| _{\mbox{const. of } x_i}\nonumber \\
&=
\left( 
\gamma^{\sum_{i=1}^N (i-1)|\lambda^{(i)}|}
\prod_{k=2}^N\prod_{i=1}^{n_k}\prod_{m=1}^{k-1} 
R_{-\lambda^{(k)}_i} (\alpha^{(k-1)}_{i,m} ) \right)
\prod_{i=1}^N Q_{\lambda^{(i)}}(a^{(i)}_{-n})\ketzero \nonumber \\
&\quad +
\sum_{\substack{\mu^{(j)}\geq \lambda^{(j)} (\forall j)\\ \mu^{(k)}\neq \lambda^{(k)} \mbox{ for some $k$}  }}(\mathrm{const.}) \prod_{i=1}^N Q_{\mu^{(i)}}(a^{(i)}_{-n})\ketzero
+\mathcal{O}
\left(
\vm \overstar{<} \vl
\right).\label{eq: pf of R 4}
\end{align}

The existence theorem of generalized Macdonald functions (Fact \ref{fact:existence thm of Gn Mac}) 
shows that the coefficient in front of $\ket{Q_{\vl}}$ in the basis of generalized Macdonald functions 
is the same as the one in front of $\prod_{i=1}^N Q_{\lambda^{(i)}}(a^{(i)}_{-n})\ketzero$ 
in (\ref{eq: pf of R 4}). 
\qed

\subsection{Proof of Lemma \ref{lem: trf formula 1} }
\label{sec: trf formula lem1}

First, it can be shown that
\begin{align}\label{eq: 1}
&\prod_{1\leq i<j\leq n+m}
{(q^{-\theta_{j}}q s_j/t s_i;q)_{\theta_{i}} \over (q^{-\theta_{j}}s_j/s_i;q)_{\theta_{i}} } \nonumber\\
&=\prod_{1\leq i<j\leq n+m} 
\frac{q^{\theta_i}s_i^{-1}-q^{\theta_j}s_j^{-1}}{s_i^{-1}-s_j^{-1}}
\frac{(ts_i/s_j;q)_{\theta_j}}{(qs_i/s_j;q)_{\theta_j}}
t^{-\theta_j} 
\frac{(qs_j/ts_i;q)_{\theta_i-\theta_j}}{(qs_j/s_i;q)_{\theta_i-\theta_j}}
\end{align}
and
\begin{align}\label{eq: 2}
\prod_{1\leq i<j\leq n} 
\frac{(q\ts{j}/t\ts{i};q)_{\infty}}{(q\ts{j}/\ts{i};q)_{\infty}} 
=\prod_{1\leq i<j\leq n} 
\frac{(q s_j/t s_i;q)_{\infty}}{(q s_j/ s_i;q)_{\infty}} 
\frac{(qs_j/s_i;q)_{\theta_i-\theta_j}}{(qs_j/ts_i;q)_{\theta_i-\theta_j}}. 
\end{align}
By (\ref{eq: 1}) and (\ref{eq: 2}), 
we have 
\begin{align}\label{eq: 3}
&\prod_{1\leq i<j\leq n} 
\frac{(q\ts{j}/t\ts{i};q)_{\infty}}{(q\ts{j}/\ts{i};q)_{\infty}} \cdot 
d_{n+m}(\theta;s|q,t) \nonumber\\
&=
\prod_{1\leq i<j\leq n} 
\frac{(q s_j/t s_i;q)_{\infty}}{(q s_j/ s_i;q)_{\infty}} \cdot 
\Mphi{n+m-1, 1}{\theta}
{t&,\ldots,&t \\ s_1^{-1}&,\ldots,& s_{n+m-1}^{-1}}
{ts_{n+m}\\qs_{n+m}}  \nonumber\\
& \quad \times \prod_{i=1}^{n+m-1} q^{\theta_i} t^{-i\theta_i} \cdot 
\prod_{j=n+1}^{n+m-1} \prod_{i=1}^{j-1} \frac{(qs_j/ts_i;q)_{\theta_i-\theta_j}}{(qs_j/s_i;q)_{\theta_i-\theta_j}}. 
\end{align}
Furthermore, 
we can see
\begin{align}\label{eq: 4}
&\prod_{k=1}^{m-1} \prod_{i=1}^{n+k} \frac{(q\ts{n+k}/t\ts{i};q)_{\sigma_k}}{(q\ts{n+k}/\ts{i};q)_{\sigma_k}}  \\
&=\lim_{h\rightarrow 1}
\prod_{k=1}^{m-1} 
 \prod_{i=1}^{n+k}
\frac{( q q^{\rho_k} s_{n+k}/ts_i ;q)_{\theta_i}}{(h q q^{\rho_k} s_{n+k}/s_i ;q)_{\theta_i}}
\frac{( q s_{n+k}/ts_i ;q)_{\rho_k}}{(h q s_{n+k}/s_i ;q)_{\rho_k}}
\frac{(qs_{n+k}/s_i;q)_{\theta_i-\theta_{n+k}}}{(qs_{n+k}/ts_i;q)_{\theta_i-\theta_{n+k}}} \nonumber
\end{align}
and 
\begin{align}\label{eq: 5}
&\prod_{1\leq k <l \leq m-1} 
\frac{(t q^{-\sigma_k}\ts{n+l}/\ts{n+k};q)_{\sigma_l}}{(q^{-\sigma_k}\ts{n+l}/\ts{n+k};q)_{\sigma_l}}\nonumber  \\
&= 
\prod_{1\leq k <l \leq m-1} 
\frac{(t q^{-\rho_k} s_{n+l}/ s_{n+k};q)_{\rho_l}}{(q^{-\rho_k} s_{n+l}/ s_{n+k};q)_{\rho_l}}
\frac{(q q^{\rho_k} s_{n+k}/ ts_{n+l};q)_{\theta_{n+l}}}{(q^{\rho_k} s_{n+k}/ s_{n+l};q)_{\theta_{n+l}}} 
\times t^{\theta_{n+l}}. 
\end{align}
Combining (\ref{eq: 3}), (\ref{eq: 4}) and (\ref{eq: 5}) yields Lemma \ref{lem: trf formula 1}. 
\qed

\subsection{Proof of Lemma \ref{lem: trf formula 2}}\label{sec: trf formula lem2}
Set for short 
\begin{align}
&A=\prod_{k=1}^{m-1}\frac{q^{\nu_k+\rho_k}s_{n+k}-q^{\nu_m} s_{n+m}}{q^{\rho_{k}}s_{n+k}-s_{n+k}} \nonumber \\
&\qquad \times
\prod_{k=1}^{m-1}
\frac{(t;q)_{\nu_{k}}}
{(q;q)_{\nu_{k}}}
\frac{(q q^{\rho_k}s_{n+k}/ts_{n+m};q)_{\nu_{k}}}
{(q q^{\rho_k}s_{n+k}/s_{n+m};q)_{\nu_{k}}}
\frac{(t q^{-\rho_k}s_{n+m}/s_{n+k};q)_{\nu_{m}}}
{(q q^{-\rho_k}s_{n+m}/s_{n+k};q)_{\nu_{m}}}, \\
&B=
\prod_{1\leq k<l \leq m-1}
\frac{q^{\nu_k+\rho_k}s_{n+k}-q^{\nu_l+\rho_l} s_{n+l}}{q^{\rho_{k}}s_{n+k}-q^{\rho_l}s_{n+l}}\cdot 
\prod_{\substack{ k\neq l}}
\frac{(t q^{\rho_k-\rho_l}s_{n+k}/s_{n+l};q)_{\nu_{k}}}
{(q q^{\rho_k-\rho_l}s_{n+k}/s_{n+l};q)_{\nu_{k}}}, \\
&C=\prod_{k=1}^{m-1}\prod_{i=1}^{n+m-1}
\frac{(hq q^{\rho_k}s_{n+k}/ts_{i};q)_{\nu_{k}}}
{(h q q^{\rho_k}s_{n+k}/s_{i};q)_{\nu_{k}}}, \\
&D=\frac{(q/t;q)_{\nu_m}}{(q;q)_{\nu_m}} 
\prod_{i=1}^{n+m-1}
\frac{(h q s_{n+m}/t s_{i};q)_{\nu_{m}}}
{(h q s_{n+m}/s_{i};q)_{\nu_{m}}}. 
\end{align}
Then $\phi^{m,n+m-1}_{\nu}=ABCD$. 
First, we can get 
\begin{align}\label{eq: A}
A= 
&\prod_{k=1}^{m-1}
(q/t)^{\theta_k}
\frac{(t;q)_{\theta_{k}}}{(q;q)_{\theta_{k}}}
\frac{(t q^{-\mu_k+\mu_m}s_{n+m}/s_{n+k};q)_{\theta_{k}}}
{(q q^{-\mu_k+\mu_m}s_{n+m}/s_{n+k};q)_{\theta_{k}}} \nonumber \\ 
&\times 
\prod_{k=1}^{m-1}
\frac{(t q^{-\mu_k}s_{n+m}/s_{n+k};q)_{\mu_{m}}}
{(q^{-\mu_k}s_{n+m}/s_{n+k};q)_{\mu_{m}}}. 
\end{align}
The first product in (\ref{eq: A}) 
reproduces the factors in $d_m((\theta_i); (q^{\mu_i}s_{n+i})|q,t)$, 
i.e., the first product in (\ref{eq: form of d}). 
Next, we have 
\begin{align}
&\lim_{h\rightarrow 1}
\widetilde{\mathsf{N}}^{n,m-1}_{\rho}(h; s_1,\ldots, s_{n+m-1}) C 
= \mathsf{N}^{n,m-1}_{\mu}(s_1,\ldots, s_{n+m-1}) E ,
\end{align}
where
\begin{align}
&E\equiv \prod_{1\leq k<l\leq m-1} 
t^{-\theta_k}
\frac{(t q^{-\mu_k+\mu_l}s_{n+l}/s_{n+k};q)_{\theta_k-\theta_l}}
{( q^{-\mu_k+\mu_l}s_{n+l}/s_{n+k};q)_{\theta_k-\theta_l}}.
\end{align}
$BE$ reproduces the factors in $d_m((\theta_i); (q^{\mu_i}s_{n+i})|q,t)$:
\begin{align}
BE=\prod_{1\leq i <j \leq m} 
\frac{(tq^{-\mu_k+\mu_l}s_{n+l}/s_{n+k};q)_{\theta_k}}
{(qq^{-\mu_k+\mu_l}s_{n+l}/s_{n+k};q)_{\theta_k}}
\frac{(qq^{-\mu_k+\mu_l-\theta_l}s_{n+l}/ts_{n+k};q)_{\theta_k}}
{(q^{-\mu_k+\mu_l-\theta_l}s_{n+l}/s_{n+k};q)_{\theta_k}}.
\end{align}
This corresponds to the second product in (\ref{eq: form of d}).
The product of $\mathsf{N}^{n,m-1}_{\mu}$, $D$ and the remaining factor in (\ref{eq: A}) 
is 
\begin{align}
&\mathsf{N}^{n,m-1}_{\mu}(s_1,\ldots, s_{n+m-1}) \cdot 
\prod_{k=1}^{m-1}\frac{(t q^{-\mu_k}s_{n+m}/s_{n+k};q)_{\mu_{m}}}
{(q^{-\mu_k}s_{n+m}/s_{n+k};q)_{\mu_{m}}}\cdot 
\lim_{h\rightarrow 1}D \nonumber \\
&= \mathsf{N}^{n,m}_{\mu}(s_1,\ldots, s_{n+m}). 
\end{align}
Therefore, Lemma \ref{lem: trf formula 2} follows. 
\qed

\section{Kac determinant revisited}\label{App: reargof SM}

The formula for the Kac determinant with respect to the vectors $\ket{X_{\vl}}$ 
has been discussed in \cite{O}. 
That shows the fact that $\ket{X_{\vl}}$ form a basis on the Fock space (Fact \ref{fact: PBW basis}). 
For the sake of reader's convenience, we revisit the proof, 
clarifying the choice of the integral cycles. 
Here, we construct the $q$-invariance cycles by using the elliptic theta function. 

\begin{definition}
Let $1\leq k\leq N-1$ 
and  $u_k=q^s t^{-r} u_{k+1}$ ($r, s \in \mathbb{Z}_{>0}$ ). 
Define the vector $\ket{\chi^{(k)}_{r,s}} \in \mathcal{F}_{\vu}$ by  
the integral 
\begin{align}
\ket{\chi^{(k)}_{r,s}}:=
\oint \frac{dz}{z} 
S^{(k)}(z_1)\cdots S^{(k)}(z_r) 
\prod_{i=1}^r \frac{\theta_q(t^{2i} u_kz_i/u_{k+1})}{\theta_q(t z_i)}\ketzero. 
\end{align}
Here and hereafter, we use the shorthand notation 
\begin{align}
\oint \frac{dz}{z}:=\oint_{T} \prod_{i=1}^r \frac{dz_i}{2 \pi \sqrt{-1} z_i},
\end{align}
where the cycle is the $r$-dimensional torus $T: |z_1|=\cdots =|z_r|=1$. 
Note that $\vu$ is the spectral parameter of the codomain of $S^{(k)}(z_1)$. 
\end{definition}

\begin{proposition}\label{prop: nonzero}
The vector $\ket{\chi^{(k)}_{r,s}}$ does not vanish. 
In particular,  this is of level $rs$. 
\end{proposition}

Let us prepare a lemma with respect to the symmetrization of theta functions. 
Set 
\begin{align}
\widehat{F}_{r,s}  (z_1,\ldots , z_r)
&:=\frac{1}{r!}\sum_{\sigma \in \mathfrak{S}_r} \prod_{i=1}^r \theta_q(q^s t^{2i-r} z_{\sigma(i)}) \cdot
\prod_{\substack{i<j \\ \sigma(i) > \sigma(j)}}
t^{-1}\frac{\theta_q(t z_{\sigma(i)}/z_{\sigma(j)})}{ \theta_q(t^{-1}z_{\sigma(i)}/z_{\sigma(j)})}.
\end{align}

\begin{lemma}\label{lem: symmetrize formula}
\begin{align}
\widehat{F}_{r,s} (z_1,\ldots , z_r)= \frac{1}{r !} \prod_{i=2}^{r} \frac{\theta_q (t^i)}{\theta_q(t)} \cdot
\prod_{1\leq i<j \leq r} \frac{\theta_q (z_i/z_j)}{\theta_q (tz_i/z_j)} \cdot
\prod_{i=1}^r \theta_q(q^s t z_i). 
\end{align}
\end{lemma}

As for the proof of this lemma, see the proof of Lemma 4 in \cite{JLMP}.

\noindent
\textit{Proof of Proposition \ref{prop: nonzero}.} 
First, by using the operator products of 
screening currents and 
Lemma \ref{lem: symmetrize formula}, 
we have 
\begin{align}
\ket{\chi^{(k)}_{r,s}} 
&= \oint \frac{dz}{z} \Delta(z) 
\prod_{1\leq i<j\leq r} \frac{\theta_q(t z_i/ z_j)}{\theta_q(z_i/z_j)} \cdot
\prod_{i=1}^r \frac{\theta_q(q^s t^{2i-r}z_i)}{\theta_q(tz_i)} 
:S^{(k)}(z_1)\cdots S^{(k)}(z_r) : \ketzero \nonumber \\
&=\oint \frac{dz}{z} \Delta(z) 
\prod_{1\leq i<j\leq r}\frac{\theta_q(t z_i/ z_j)}{\theta_q(z_i/z_j)} \cdot
\prod_{i=1}^r \frac{1}{\theta_q(tz_i)} \widehat F_{r,s} (z_1,\ldots , z_r) 
:S^{(k)}(z_1)\cdots S^{(k)}(z_r) : \ketzero \nonumber\\
&= \frac{1}{r!}\prod_{i=2}^r \frac{\theta_q(t^i)}{\theta_q(t)} \cdot
\oint \frac{dz}{z} \Delta(z) 
q^{\frac{r}{2}s(1-s)} \prod_{i=1}^r(t z_i)^{-s} \cdot
:\prod_{i=1}^rS^{(k)}(z_i): \ketzero. \label{eq: scr act} 
\end{align}
Here, $\Delta(z)$ is defined in (\ref{eq: Delta(z)}). 
Note that $:\prod_{i=1}^rS^{(k)}(z_i): \ketzero$ can be regarded as the kernel function 
for the Macdonald functions. 
Hence, it is expanded in terms of the Macdonald functions (See Fact \ref{fact: kernel}). 
Note that $\prod_{i=1}^rz_i^{-s}$ is the Macdonald polynomial 
with a rectangular Young diagram in $r$ variables. 
Therefore,  (\ref{eq: scr act}) can be written as 
\begin{align}
&\frac{q^{rs(1-s)/2}t^{-rs}}{r!}\prod_{i=2}^r \frac{\theta_q(t^i)}{\theta_q(t)}\cdot 
\sum_{\lambda} P_{\lambda}(\alpha_{-n}^{(k)})\ketzero 
\langle P^{(r)}_{(s^r)}(z),Q^{(r)}_{\lambda}(z) \rangle'_r \nonumber \\
&=\frac{q^{rs(1-s)/2}t^{-rs}}{r!}
\prod_{i=2}^r \frac{\theta_q(t^i)}{\theta_q(t)} \cdot 
P_{(s^r)}(\alpha_{-n}^{(k)}) \ketzero 
\langle P^{(r)}_{(s^r)}(z),Q^{(r)}_{(s^r)}(z) \rangle'_r, 
\end{align}
where $P_{\lambda}^{(r)}(z)$ and $Q^{(r)}_{\lambda}(z)$ denote 
the Macdonald polynomials in $r$ variables,  
\begin{align}
\alpha_{-n}^{(k)}:=\gamma^{kn}(-\gamma^n a^{(k)}_{-n}+a^{(k+1)}_{-n}),  
\end{align}
and $\langle -,- \rangle'_r$ denotes the scalar product defined in Appendix \ref{App: ord. Macdonald}. 
Since the Macdonald polynomials are pairwise orthogonal for $\langle -,- \rangle'_r$ and 
the inner product $\langle P^{(r)}_{(s^r)},Q^{(r)}_{(s^r)} \rangle'_r$ can be evaluated (Fact \ref{fact: <P,Q>'}) to be nonvanishing. 
Thus, $\ket{\chi^{(k)}_{r,s}} \neq 0$ and of level $rs$. 
\qed 

We show the following commutativity with the algebra $\WA$. 
The proof is similar to the case corresponding to the Minimal model, 
given in \cite{JLMP}. 

\begin{proposition}\label{prop: commutativity}
Let $r,s \in \mathbb{Z}_{>0}$ and $k,j \in \{1,\ldots, r\}$. 
Further we assume $|t|<|q|$. 
Then 
\begin{align}
\Big[ X^{(j)}(z), 
\oint \frac{dw}{w} 
S^{(k)}(w_1) \cdots S^{(k)}(w_r) 
\prod_{i=1}^{r} \frac{\theta_q( t^{2i}\frac{u_k}{u_{k+1}}w_i)}{\theta_q(tw_i)} 
\Big]
=0,
\end{align}
where the spectral parameter of the codomain of  $S^{(k)}(w_1)$ is $\vu$ with $u_k=q^st^{-r}u_{k+1}$. 

\begin{proof}
As in the proof of Proposition \ref{prop: commutativity of S}, 
it suffices to consider the relation only with $\Lambda^{(k)}(z)+\Lambda^{(k+1)}(z)$. 
By (\ref{eq: comm rel of lam and S}), we have 
\begin{align}
&\Big[ \Lambda^{(k)}(z)+\Lambda^{(k+1)}(z), 
\oint \frac{dw}{w} 
S^{(k)}(w_1) \cdots S^{(k)}(w_r) \cdot 
\prod_{i=1}^{r} 
\frac{\theta_q( t^{2i}\frac{u_k}{u_{k+1}}w_i)}{\theta_q(tw_i)} 
\Big] \nonumber \\
&
=\sum_{m=1}^r\oint \frac{dw}{w}
(t-1)t^{m-1} u_k 
\left( T_{q,w_m}-1 \right)
\delta\left(\frac{tw_m}{qz} \right)
\Delta(w) 
\prod_{1\leq i <j\leq r}
\frac{\theta_q( tw_i/w_j)}{\theta_q(w_i/w_j)} \cdot
\prod_{i=1}^r\frac{\theta_q( q^st^{2i-r}w_i)}{\theta_q(tw_i)} \nonumber \\
&\qquad \times\prod_{i=1}^{m-1}\frac{1-t^{-1}w_m/w_i}{1-w_m/w_i}\cdot
\prod_{i=m+1}^{r}\frac{1-tw_i/w_m}{1-w_i/w_m}  \cdot
:\Lambda^{(k)}(tw_m/q) \prod_{i=1}^rS^{(k)}(z_i):. \\
\end{align}
By symmetrizing the variables  $w_i$'s, we have
\begin{align}
& 
\frac{1}{r!}\sum_{m=1}^r
\sum_{\sigma \in \mathfrak{S}_r} 
\oint \frac{dw}{w}
(t-1)u_k 
\left( T_{q,w_{\sigma(m)}}-1 \right)
\delta\left(\frac{tw_{\sigma(m)}}{qz} \right)
\Delta(w) \nonumber\\
&\qquad \times \prod_{1\leq i <j\leq r}
\frac{\theta_q( tw_{\sigma(i)}/w_{\sigma(j)})}{\theta_q(w_{\sigma(i)}/w_{\sigma(j)})} \cdot
\prod_{i=1}^r\frac{\theta_q( q^st^{2i-r}w_{\sigma(i)})}{\theta_q(tw_{\sigma(i)})} \nonumber \\
&\qquad \times\prod_{i\neq m}\frac{1-tw_{\sigma(i)}/w_{\sigma(m)}}{1-w_{\sigma(i)}/w_{\sigma(m)}}\cdot
:\Lambda^{(k)}(tw_{\sigma(m)}/q) \prod_{i=1}^rS^{(k)}(z_i): \nonumber\\
&=\frac{1}{r!}\sum_{l=1}^r \sum_{m=1}^r
\sum_{\substack{\sigma \in \mathfrak{S}_r\\ \sigma(m)=l}} 
\oint \frac{dw}{w}
(t-1)u_k 
\left( T_{q,w_l}-1 \right)
\delta\left(\frac{tw_{l}}{qz} \right)
\Delta(w)\nonumber \\
&\qquad \times \prod_{1\leq i <j\leq r}
\frac{\theta_q( tw_{\sigma(i)}/w_{\sigma(j)})}{\theta_q(w_{\sigma(i)}/w_{\sigma(j)})} \cdot
\prod_{i=1}^r\frac{\theta_q( q^st^{2i-r}w_{\sigma(i)})}{\theta_q(tw_{\sigma(i)})} \cdot
\prod_{\sigma(i)\neq l}\frac{1-tw_{\sigma(i)}/w_{l}}{1-w_{\sigma(i)}/w_{l}} \cdot
:\Lambda^{(k)}(tw_{l}/q) \prod_{i=1}^rS^{(k)}(z_i):\nonumber\\
&=\frac{1}{r!}\sum_{l=1}^r
\oint \frac{dw}{w}
(t-1)u_k 
\left( T_{q,w_l}-1 \right)
\delta\left(\frac{tw_{l}}{qz} \right)
\Delta(w) \nonumber\\
&\qquad \times 
\left( \sum_{\sigma \in \mathfrak{S}_r} 
\prod_{1\leq i <j\leq r}
\frac{\theta_q( tw_{\sigma(i)}/w_{\sigma(j)})}{\theta_q(w_{\sigma(i)}/w_{\sigma(j)})} \cdot
\prod_{i=1}^r\frac{\theta_q( q^st^{2i-r}w_{\sigma(i)})}{\theta_q(tw_{\sigma(i)})} 
\right)
\prod_{i\neq l}\frac{1-tw_{i}/w_{l}}{1-w_{i}/w_{l}} \cdot
:\Lambda^{(k)}(tw_{l}/q) \prod_{i=1}^rS^{(k)}(z_i):. 
\end{align}
By Lemma \ref{lem: symmetrize formula}, 
this can be rewritten as  
\begin{align}
&\frac{1}{r!}\sum_{l=1}^r
\oint \frac{dw}{w}
(t-1)u_k 
\left( T_{q,w_l}-1 \right)
\delta\left(\frac{tw_{l}}{qz} \right)
\Delta(w) \nonumber\\
&\qquad \times 
\prod_{i=2}^r\frac{\theta_q(t^i)}{\theta_q(t)} \cdot
\prod_{i=1}^r q^{\frac{s(1-s)}{2}}(tz_i)^{-s} \cdot
\prod_{i\neq l}\frac{1-tw_{i}/w_{l}}{1-w_{i}/w_{l}} \cdot
:\Lambda^{(k)}(tw_{l}/q) \prod_{i=1}^rS^{(k)}(z_i):. 
\label{eq: comm rel of lam and SSS}
\end{align}
In this expression, 
we have poles in $w_l$ 
at $w_l=0$, $w_l=q^nt w_i$ and 
$w_l=q^{-n+1}t^{-1}w_i$ ($i\neq l$, $n=1,2,\ldots$). 
Since $|t|<|q|$, they do not change the integral while we $q$-shift the cycle as 
$w_l\rightarrow q w_l$. 
Therefore, 
the integral (\ref{eq: comm rel of lam and SSS}) is zero. 
\end{proof}
\end{proposition}

Proposition \ref{prop: nonzero} and Proposition \ref{prop: commutativity}
show the existence of the singular vectors of the algebra $\WA$. 

\begin{corollary}\label{cor: sing vct}
The vector $\ket{\chi^{(k)}_{r,s}}$ is a singular vector of level $rs$, i.e., 
\begin{align}
X^{(i)}_n\ket{\chi^{(k)}_{r,s}}=0  
\end{align}
for all $n >0$ and $i=1,\ldots, N$. 
\end{corollary}

We revisit the proof of the following formula for 
the Kac determinant 
$\mathrm{det}_n := \det \left( \Braket{X_{\vl}|X_{\vm}} \right)_{\vl, \vm \vdash n}$.

\begin{proposition}\label{prop: KacDet}
We have 
\begin{align}\label{eq:KacDet for DIM}
\mathrm{det}_n&=
\prod_{\vl \vdash n} \prod_{k=1}^N b_{\lambda^{(k)}}(q) b'_{\lambda^{(k)}}(t^{-1})
\nonumber\\
& \times \prod_{\substack{1\leq r,s\\ rs\leq n}}
\left( 
(u_1 u_2 \cdots u_N)^{2}
 \prod_{1\leq i < j \leq N} (u_i-q^st^{-r}u_j)(u_i-q^{-r}t^s u_j) \right)^{P^{(N)}(n-rs)}. 
\end{align}
Here  
$b_{\lambda}(q) := \prod_{i\geq 1} \prod_{k=1}^{m_i} (1-q^k)$, 
$b'_{\lambda}(q) := \prod_{i\geq 1} \prod_{k=1}^{m_i} (-1+q^k)$. 
$P^{(N)}(n)$ is the number of $N$-tuples of Young diagrams of size $n$, i.e.,
$\# \big\{ \vl=(\lo, \ldots , \lN) \big| |\vl | =n\big\}$. 
In particular, if $N=1$, 
\begin{equation}
\mathrm{det}_n=\prod_{\lambda \vdash n} 
b_{\lambda}(q) b'_{\lambda}(t^{-1})
\times u_1^{2\sum_{\lambda \vdash n}\ell(\lambda)}. 
\end{equation}

\proof
The inner product $\braket{X_{\vl}|X_{\vl}}$ can be calculated by 
commutation relations of $X^{(i)}_n$.
The parameters $u_1, \ldots, u_N$ arise 
from the eigenvalues of $X^{(i)}_0$. 
Therefore, it can be seen that $\braket{X_{\vl}|X_{\vl}}$ is a polynomial in 
\begin{align}
m_{(1^i)}(u_1,\ldots,u_N)=\sum_{j_1<\cdots<j_i} u_{j_1}\cdots u_{j_i}
\end{align}
over $\mathbb{Q}(q^{\frac{1}{2}},t^{\frac{1}{2}})$, 
and thus so does the $\mathrm{det}_n$. 
Define the action of the  symmetric group $\mathfrak{S}_N$ (the Weyl group of type $A_{N-1}$) 
on polynomials in $u_j$ in the usual way. 
Since $m_{(1^i)}(u_1,\ldots,u_N)$ is invariant with respect to this action, 
$\mathrm{det}_n$ is also invariant, i.e., 
a symmetric polynomial in $u_j$.

Furthermore, let us introduce the new parameters $u_i'$ and $u''$ by 
\begin{align}
\prod_{i=1}^N u_i'=1, \quad u_i=u_i' u''. 
\end{align}
Then $\braket{X_{\vl}|X_{\vm}}$ can be decomposed as 
\begin{align}
\braket{X_{\vl}|X_{\vm}}= (u'')^{\sum_{k=1}^N k (\ell(\lambda^{(k)})+\ell(\mu^{(k)}))}
\times (\mbox{polynomial in $u'_i$}). 
\end{align}
Therefore,  $\mathrm{det}_n$ can be written as 
\begin{align}
\mathrm{det}_n=(u'')^{2 \sum_{|\vl|=n}\sum_{k=1}^N k \ell(\lambda^{(k)})}
\times F(u'_1,\ldots, u'_N), 
\end{align}
where $F(u'_1,\ldots, u'_N)$ is some polynomial in $u'_i$. 
Note that the maximum degree of $F(u'_1,\ldots, u'_N)$ with respect to 
each $u'_i$ is $2\sum_{|\vl|=n}\sum_{k=1}^N  \ell(\lambda^{(k)})$. 

By Corollary \ref{cor: sing vct}, it can be seen that 
for $r, s \in \mathbb{Z}_{>0}$ with $rs\leq n$, 
the Kac determinant $\mathrm{det}_n$ has the factors 
\begin{align}
(u_k-q^{s}t^{-r}u_{k+1})^{P^{(N)}(n-rs)}=\left( u''(u'_k-q^{s}t^{-r}u'_{k+1}) \right)^{P^{(N)}(n-rs)}
\end{align}
in the usual way. 
By the $\mathfrak{S}_N$ invariance, 
$\mathrm{det}_n$ has also the factor 
\begin{align}
(u_i-q^{\pm s}t^{\mp r}u_{j})^{P^{(N)}(n-rs)}=\left( u''(u'_i-q^{\pm s}t^{\mp r}u'_{j}) \right)^{P^{(N)}(n-rs)}
\end{align}
for $i\neq j$. 
Noticing the degree of $F(u'_1,\ldots, u'_N)$, 
we can see that 
\begin{align}
&\mathrm{det}_n
 = g_{N,n}(q,t) 
    \times (u'')^{2\sum_{\vl \vdash n} \sum_{i=1}^N i\,  \ell(\lambda^{(i)})} \nonumber \\
&\qquad \quad \times \prod_{1 \leq i<j\leq N}
    \prod_{\substack{1\leq r,s \\ rs \leq n}}
     \left( (u'_{i}-q^{s} t^{-r}u'_{j})(u'_{i}-q^{-r} t^{s}u'_{j}) \right)^{P^{(N)}(n-rs)} \nonumber \\ 
&= g_{N,n}(q,t) \nonumber \\
& \quad  \times\prod_{\substack{1\leq r,s \\ rs \leq n}} 
    \left( (u_1u_2\cdots u_N)^2 
      \prod_{1 \leq i<j\leq N}
      (u_{i}-q^{s} t^{-r}u_{j})(u_{i}-q^{-r} t^{s}u_{j}) \right)^{P^{(N)}(n-rs)}.
\end{align}
Here, $g_{N,n}(q,t) \in \mathbb{Q}(q^{\frac{1}{2}},t^{\frac{1}{2}})$. 
Thus, 
we obtained the vanishing loci of the Kac determinant $\mathrm{det}_n$. 
The prefactor $g_{N,n}(q,t)$ has been evaluated in \cite{O}. \qed
\end{proposition}

As a corollary of Proposition \ref{prop: KacDet}, Fact \ref{fact: PBW basis} follows.

\section{Examples}\label{App: example}

\subsection{Examples of $\ket{K_{\vl}}$}\label{sec: ex of K}
We present examples of the transition matrix  $\alpha^{(\pm)}_{\vl,\vm}$ from $\ket{K_{\vl}}$ to 
the PBW-type basis $\ket{X_{\vl}}$.

Examples of $\alpha^{(+)}_{\vl,\vm}$ in the case $N=1$:

\begin{align}
\begin{array}{c||ccc}
\vl \; \backslash\;  \vm  &(1) & (2) & (1,1)\\ \hline \hline
(1)&1&0&0 \\
(2)&0& \frac{(q-1) u_1}{t} & 1 \\
 (1,1)&0& -\frac{q (t-1) u_1}{t} & 1 \\
\end{array},
\end{align}
\begin{align}
\begin{array}{c||ccc}
\vl \; \backslash\;  \vm  &(3) & (2,1) & (1,1,1)\\ \hline \hline
(3)& \frac{(q-1)^2 (2 t q-q+t) u_1^2}{t^3} & \frac{(q-1) (2 t q-q+t+1) u_1}{t^2} & 1 \\
(2,1)& -\frac{(q-1) q (t-1) (t q-q+1) u_1^2}{t^3} & -\frac{\left(t^2 q^2-2 t q^2+q^2+t q-2 q+1\right) u_1}{t^2} & 1 \\
(1,1,1)& -\frac{q^2 (q-t-2) (t-1)^2 u_1^2}{t^3} & -\frac{q (t-1) (t q-q+t+2) u_1}{t^2} & 1 \\
\end{array}.
\end{align}

Examples of $\alpha^{(+)}_{\vl,\vm}$ in the case $N=2$:

\begin{align}
\begin{array}{c||cc}
\vl \; \backslash\;  \vm & (\emptyset, (1)) & ((1), \emptyset)\\ \hline \hline
 (\emptyset, (1))& -\frac{t}{q u_2} & 1 \\
((1), \emptyset)&  -\frac{t}{q u_1} & 1 \\
\end{array},
\end{align}

\begin{align}
\begin{array}{c||ccccc}
\vl \;\backslash\; \vm & (\emptyset, (2)) & (\emptyset, (1^2)) &((1), (1))&((2), \emptyset) &((1^2), \emptyset)  \\ \hline \hline
(\emptyset, (2)) & -\frac{(q-1) (-q u_1+q t u_1-t u_1+q t u_2)}{q^3 u_2} & \frac{t^2}{q^3 u_2^2} & -\frac{(q+1)
   t}{q^2 u_2} & \frac{(q-1) (q u_2-u_1)}{q t} & 1 \\
(\emptyset, (1,1)) & -\frac{(t-1) t (q u_1+t u_1-u_1-u_2)}{q u_2} & \frac{t^3}{q^2 u_2^2} & -\frac{t (t+1)}{q u_2} &
   \frac{q (t-1) (t u_1-u_2)}{t} & 1 \\
((1), (1)) & -\frac{(q-1) (t-1)}{q} & \frac{t^2}{q^2 u_1 u_2} & -\frac{t (u_1+u_2)}{q u_1 u_2} & 0 & 1 \\
((2),\emptyset) & -\frac{(q-1) (q t u_1-q u_2+q t u_2-t u_2)}{q^3 u_1} & \frac{t^2}{q^3 u_1^2} & -\frac{(q+1)
   t}{q^2 u_1} & \frac{(q-1) (q u_1-u_2)}{q t} & 1 \\
((1,1),\emptyset) & -\frac{(t-1) t (-u_1+q u_2+t u_2-u_2)}{q u_1} & \frac{t^3}{q^2 u_1^2} & -\frac{t (t+1)}{q u_1} &
   \frac{q (t-1) (t u_2-u_1)}{t} & 1 \\
\end{array}.
\end{align}

Examples of $\alpha^{(-)}_{\vl,\vm}$ in the case $N=2$: 

\begin{align}
\begin{array}{c||cc}
\vl \; \backslash\;  \vm & ((1), \emptyset) & (\emptyset, (1))\\ \hline \hline
((1), \emptyset)& 1 & -\frac{1}{u_1} \\
(\emptyset,(1))& 1 & -\frac{1}{u_2} \\
\end{array},
\end{align}

\begin{align}
\begin{array}{c||ccccc}
\vl \;\backslash\; \vm & ((2),\emptyset) & ((1^2),\emptyset) &((1), (1))&(\emptyset,(2)) &(\emptyset,(1^2))  \\ \hline \hline
 ((2),\emptyset)& \frac{(q-1) (q u_1-u_2)}{q t} & 1 & -\frac{q+1}{q u_1} & -\frac{(q-1) (q t u_1-q u_2+q t u_2-t
   u_2)}{q t^2 u_1} & \frac{1}{q u_1^2} \\
((1^2),\emptyset) & \frac{q (t-1) (t u_2-u_1)}{t} & 1 & -\frac{t+1}{u_1} & -\frac{q (t-1) (-u_1+q u_2+t u_2-u_2)}{t
   u_1} & \frac{t}{u_1^2} \\
((1), (1))& 0 & 1 & -\frac{u_1+u_2}{u_1 u_2} & -\frac{(q-1) q (t-1)}{t^2} & \frac{1}{u_1 u_2} \\
(\emptyset,(2))& \frac{(q-1) (q u_2-u_1)}{q t} & 1 & -\frac{q+1}{q u_2} & -\frac{(q-1) (-q u_1+q t u_1-t u_1+q t
   u_2)}{q t^2 u_2} & \frac{1}{q u_2^2} \\
(\emptyset,(1^2))& \frac{q (t-1) (t u_1-u_2)}{t} & 1 & -\frac{t+1}{u_2} & -\frac{q (t-1) (q u_1+t u_1-u_1-u_2)}{t
   u_2} & \frac{t}{u_2^2} \\
\end{array}.
\end{align}

Examples of $\alpha^{(+)}_{\vl,\vm}$ in the case $N=3$:

\begin{align}
\begin{array}{c||ccc}
\vl \; \backslash\;  \vm & (\emptyset,\emptyset, (1)) & (\emptyset,(1), \emptyset)& ((1), \emptyset,\emptyset)\\ \hline \hline
(\emptyset,\emptyset, (1)) & \frac{t^2}{q^2 u_3^2} & -\frac{t}{q u_3} & 1 \\
(\emptyset,(1), \emptyset)& \frac{t^2}{q^2 u_2^2} & -\frac{t}{q u_2} & 1 \\
((1), \emptyset,\emptyset)& \frac{t^2}{q^2 u_1^2} & -\frac{t}{q u_1} & 1 \\
\end{array}.
\end{align}

Examples of $\alpha^{(-)}_{\vl,\vm}$ in the case $N=3$:

\begin{align}
\begin{array}{c||ccc}
\vl \; \backslash\;  \vm & ((1), \emptyset,\emptyset) & (\emptyset,(1), \emptyset)&(\emptyset,\emptyset, (1)) \\ \hline \hline
((1), \emptyset,\emptyset)& 1 & -\frac{1}{u_1} & \frac{1}{u_1^2} \\
(\emptyset,(1), \emptyset)& 1 & -\frac{1}{u_2} & \frac{1}{u_2^2} \\
(\emptyset,\emptyset, (1)) & 1 & -\frac{1}{u_3} & \frac{1}{u_3^2} \\
\end{array}.
\end{align}

\subsection{Examples of matrix elements $\braket{X_{\vl}|\cV(w)|X_{\vm}}$}\label{sec: Ex of Mat el}

In this section, 
we demonstrate how to calculate the matrix elements 
$\braket{X_{\vl}|\cV(w)|X_{\vm}}$ 
by the defining relation of $\cV(w)$. 
Let us first explain it in general case. 
If $\vl\neq (\emptyset,\ldots, \emptyset)$, 
let $j=\min\{i|\lambda^{(i)}\neq\emptyset\}$. 
The defining relation gives
\begin{align}
\braket{X_{\vl}|\cV(w)|X_{\vm}}=&
w\Braket{X_{(\emptyset,\ldots, (\lambda^{(j)}_{2},\lambda^{(j)}_{3},\ldots), 
\lambda^{(j+1)},\ldots, \lambda^{(N)})}| 
X^{(j)}_{\lambda^{(j)}_1-1} \cV(w) | X_{\vm}}\\
&+\Braket{X_{(\emptyset,\ldots, (\lambda^{(j)}_{2},\lambda^{(j)}_{3},\ldots), 
\lambda^{(j+1)},\ldots, \lambda^{(N)})}|
\cV(w) 
\left(X^{(j)}_{\lambda^{(j)}_1} -(q/t)^jwX^{(j)}_{\lambda^{(j)}_1-1}\right)|\Xo_{\vm}}. \nonumber 
\end{align}
The first term can be rewritten by 
matrix elements $\braket{X_{\boldsymbol{\nu}}|\cV(w)|X_{\vm}}$ 
satisfying $|\boldsymbol{\nu}|=|\vl|-1$ 
(particularly, in the case $\lambda^{(j)}_1-1<\lambda^{(j)}_2$), 
and the second term can be expanded by vectors $\ket{X_{\boldsymbol{\rho}}}$ 
of level $|\boldsymbol{\rho}|=|\vm|-\lambda^{(j)}_{1}$ or 
$|\vm|-\lambda^{(j)}_{1}+1$.
(If $|\vm|-\lambda^{(j)}_{1},|\vm|-\lambda^{(j)}_{1}+1<0$, 
it means just $0$ vectors.) 
If $\vl =(\emptyset,\ldots, \emptyset)$, 
moving negative modes $X^{(i)}_{-n}$ to the left side of $\cV(w)$ by the defining relation, 
we have the expression in terms of Young diagrams of smaller size. 
In this way, the matrix elements $\braket{X_{\boldsymbol{\nu}}|\cV(w)|X_{\vm}}$ 
can be inductively and uniquely determined.

The followings are examples.

\subsubsection*{In $N=1$ case}

First, by the defining relation of $\cV(w)$, we have 
\begin{align}
\braket{X_{(1)}|\cV(w)|X_{\emptyset}}=
w \bra{0} \Xo_0 \cV(w)\ket{0}+\bra{0} \cV(w)\left(\Xo_1 -(t/q)w \Xo_{0}\right)\ket{0}. 
\end{align}
By $X^{(1)}_0\ket{0}=u_1\ket{0}$, $\bra{0}X^{(1)}_0=v_1\bra{0}$ and $\Xo_1\ket{0}=0$, 
we get
\begin{align}\label{eq: <(1)|V|(0)>}
\braket{X_{(1)}|\cV(w)|X_{\emptyset}}=
w v_1-(t/q)w u_1. 
\end{align}
Similarly, it is easily seen that 
\begin{align}\label{eq: <(0)|V|(1)>}
\braket{X_{\emptyset}|\cV(w)|X_{(1)}}
&=
(q/t)w^{-1}\bra{0}  \cV(w)\Xo_0\ket{0}
-(q/t)w^{-1} \bra{0} \left(\Xo_0 -(t/q)w \Xo_{-1}\right) \cV(w)\ket{0}\nonumber \\ 
&=(q/t)w^{-1}u_1-(q/t)w^{-1}v_1. 
\end{align}
Moreover, we have 
\begin{align}
\braket{X_{(1)}|\cV(w)|X_{(1)}}=&w\bra{0} \Xo_0 \cV(w)\Xo_{-1}\ket{0} 
+ \bra{0} \cV(w) \left(\Xo_1 -(t/q)w \Xo_0 \right)\Xo_{-1}\ket{0}. 
\end{align}
By the direct calculation of the free field expression, 
we have 
\begin{align}
&X^{(1)}_0\ket{X_{(1)}}=(t^{-1}-qt^{-1}+q)u_1 \ket{X_{(1)}},\label{eq: act Ex 2} \\
&X^{(1)}_1 \ket{X_{(1)}}= -(1-q)(1-t^{-1})u_1^2 \ket{0}. \label{eq: act Ex 3}
\end{align}
By (\ref{eq: act Ex 2}) and (\ref{eq: act Ex 3}), we obtain 
\begin{align}
\braket{X_{(1)}|\cV(w)|X_{(1)}}=& wv_1 \braket{X_{\emptyset}|\cV(w)|X_{(1)}}
-(1-q)(1-t^{-1})u_1^2  \nonumber \\ 
&-(t/q)w(t^{-1}-qt^{-1}+q)u_1 \braket{X_{\emptyset}|\cV(w)|X_{(1)}}.
\end{align}
The matrix element $\braket{X_{\emptyset}|\cV(w)|X_{(1)}}$ is 
already calculated. 
Since $\ket{X_{(1)}}=\ket{K_{(1)}}$ and $\bra{X_{(1)}}=\bra{K_{(1)}}$ 
in this particular case, 
$\braket{X_{(1)}|\cV(w)|X_{(1)}}$ 
is factorized and corresponds to the Nekrasov factor 
(the right hand side of (\ref{eq: matrix element})): 
\begin{align}
\braket{X_{(1)}|\cV(w)|X_{(1)}}=\braket{K_{(1)}|\cV(w)|K_{(1)}}=
\frac{(q v_1-u_1) (t u_1-v_1)}{t}. \label{eq: <(1)|V|(1)>}
\end{align}
(\ref{eq: <(1)|V|(0)>}),(\ref{eq: <(0)|V|(1)>}) and (\ref{eq: <(1)|V|(1)>}) are the simplest examples 
of our main theorem (Theorem \ref{thm: matrix elements of V}).

We list other cases. 
Let us first prepare the formula for the action of the algebra $\WA$ 
on the PBW-type basis $\ket{X_{\vl}}$. 
\begin{align}
&\Xo_0\left(
\begin{array}{c}
\ket{X_{(2)}} \\
\ket{X_{(1,1)}} \\
\end{array}
\right)
=
\left(
\begin{array}{cc}
 \frac{\left(t q^2-q^2+t^2 q-2 t q+q+t\right) u_1}{t^2} & \frac{(q-1) (t-1)}{t} \\
 \frac{(q-1)^2 q (t-1)^2 u_1^2}{t^3} & \frac{(t q-q+1)^2 u_1}{t^2} \\
\end{array}
\right)
\left(
\begin{array}{c}
\ket{X_{(2)}} \\
\ket{X_{(1,1)}} \\
\end{array}
\right), \\
&\Xo_1\left(
\begin{array}{c}
\ket{X_{(2)}} \\
\ket{X_{(1,1)}} \\
\end{array}
\right)
=
\left(
\begin{array}{c}
 \frac{(q-1) (t-1) (q t+t+1) u_1}{t^2} \\
 \frac{(q-1) (t-1) \left(t^2 q^2-t q^2+t^2 q-q+t+1\right) u_1^2}{t^3} \\
\end{array}
\right)\ket{X_{(1)}},
\end{align}
\begin{align}
\bra{X_{(1)}}X^{(1)}_0&=(t^{-1}-qt^{-1}+q)v_1 \bra{X_{(1)}},
\end{align}
\begin{align}
\bra{X_{(1,1)}}X^{(1)}_0&=
\frac{(q-1)^2 q (t-1)^2 v_1^2}{t^3}\bra{X_{(2)}}
+\frac{v_1 (q t-q+1)^2}{t^2}\bra{X_{(1,1)}}.
\end{align}
By using these relation, the matrix elements for larger Young diagrams are inductively determined as follows: 
\begin{align}
\braket{X_{(2)}|\cV(w)|X_{\emptyset}}
&=w \braket{X_{(1)}|\cV(w)|X_{\emptyset}}, \\
\braket{X_{(1,1)}|\cV(w)|X_{\emptyset}}&= 
wv_1(t^{-1}-qt^{-1}+q)\braket{X_{(1)}|\cV(w)|X_{\emptyset}}-(t/q)wu_1\braket{X_{(1)}|\cV(w)|X_{\emptyset}}, \\
\braket{X_{(3)}|\cV(w)|X_{\emptyset}}
&=w \braket{X_{(2)}|\cV(w)|X_{\emptyset}},\\
\braket{X_{(2,1)}|\cV(w)|X_{\emptyset}}&=w
\braket{X_{(1,1)}|\cV(w)|X_{\emptyset}},\\
\braket{X_{(1,1,1)}|\cV(w)|X_{\emptyset}}&=
w\frac{(q-1)^2 q (t-1)^2 v_1^2}{t^3}\braket{X_{(2)}|\cV(w)|X_{\emptyset}}
+w\frac{v_1 (q t-q+1)^2}{t^2}\braket{X_{(1,1)}|\cV(w)|X_{\emptyset}}\nonumber \\
&\quad +(t/q)wu_1\braket{X_{(1,1)}|\cV(w)|X_{\emptyset}},
\end{align}
\begin{align}
\braket{X_{\emptyset}|\cV(w)|X_{(2)}}&= 
(q/t)w^{-1}\braket{X_{\emptyset}|\cV(w)|X_{(1)}}, \\
\braket{X_{\emptyset}|\cV(w)|X_{(1,1)}}&= 
(q/t)w^{-1}(t^{-1}-qt^{-1}+q)u_1 \braket{X_{\emptyset}|\cV(w)|X_{(1)}}-(q/t)w^{-1}v_1\braket{X_{\emptyset}|\cV(w)|X_{(1)}}, \\
\braket{X_{\emptyset}|\cV(w)|X_{(3)}}&= 
(q/t)w^{-1}\braket{X_{\emptyset}|\cV(w)|X_{(2)}},\\
\braket{X_{\emptyset}|\cV(w)|X_{(2,1)}}&= (q/t)w^{-1}\braket{X_{\emptyset}|\cV(w)|X_{(1,1)}},\\
\braket{X_{\emptyset}|\cV(w)|X_{(1,1,1)}}&=
(q/t)w^{-1} \frac{(q-1)^2 q (t-1)^2 u_1^2}{t^3}
\braket{X_{\emptyset}|\cV(w)|X_{(2)}} \nonumber\\
& \quad +(q/t)w^{-1}\frac{(t q-q+1)^2 u_1}{t^2}
\braket{X_{\emptyset}|\cV(w)|X_{(1,1)}} \nonumber\\
&\quad -(q/t)w^{-1}v_1\braket{X_{\emptyset}|\cV(w)|X_{(1,1)}},
\end{align}
\begin{align}
\braket{X_{(2)}|\cV(w)|X_{(1)}}&=
w \braket{X_{(1)}|\cV(w)|X_{(1)}}-(t/q)w \frac{(q-1) (t-1) u_1^2}{t}, \\
\braket{X_{(1,1)}|\cV(w)|X_{(1)}}&= 
w(t^{-1}-qt^{-1}+q)v_1 \braket{X_{(1)}|\cV(w)|X_{(1)}}
+\frac{(q-1) (t-1) u_1^2}{t}\braket{X_{(1)}|\cV(w)|X_{\emptyset}} \nonumber \\
&\quad -(t/q)w(t^{-1}-qt^{-1}+q)u_1 \braket{X_{(1)}|\cV(w)|X_{(1)}}. 
\end{align}
By combining these matrix elements and 
the examples of transition matrices $\alpha^{(\pm)}_{\vl, \vm}$ 
from $\ket{K_{\vl}}$ to $\ket{X_{\vm}}$ in the last subsection, 
we can check Theorem \ref{thm: matrix elements of V}.

\subsubsection*{In $N=2$ case}

If there is only one box in all Young diagrams, 
it is clear that 
\begin{align}
\braket{X_{(1),\emptyset}|\cV(w)|X_{\emptyset,\emptyset}}&=
w(v_1+v_2)-w(t/q)(u_1+u_2),\\ 
\braket{X_{\emptyset,(1)}|\cV(w)|X_{\emptyset,\emptyset}}&=
w(v_1v_2)-w(t/q)^2(u_1u_2),\\
\braket{X_{\emptyset,\emptyset}|\cV(w)|X_{(1),\emptyset}}&=
(q/t)w^{-1}(u_1+u_2)-(q/t)w^{-1}(v_1+v_2),
\\ 
\braket{X_{\emptyset,\emptyset}|\cV(w)|X_{\emptyset,(1)}}&=
(q/t)^2w^{-1}(u_1u_2)-(q/t)^2w^{-1}(v_1v_2).
\end{align}
As in the $N=1$ case, we prepare the formula for the action of $X^{(i)}_n$. 
\begin{align}
\left(
\begin{array}{c}
\bra{X_{(1),\emptyset}} \\
\bra{X_{\emptyset, (1)}}
\end{array}
\right)X^{(i)}_0
=
\left(
\begin{array}{cc}
\chi^{(i)}_{1,1} & \chi^{(i)}_{1,2} \\
\chi^{(i)}_{2,1} & \chi^{(i)}_{2,2} \\
\end{array}
\right)
\left(
\begin{array}{c}
\bra{X_{(1),\emptyset}} \\
\bra{X_{\emptyset, (1)}}
\end{array}
\right),
\end{align}
where
\begin{align}
&\chi^{(1)}_{1,1}= \frac{(t-1) v_1 (-q v_1+q t v_1+v_1+t v_2)}{t},\\
&\chi^{(1)}_{1,2}=\frac{(t-1) \sqrt{\frac{q}{t}} v_2 \left(t v_1 q^2-v_1
   q^2+t v_2 q^2-v_2 q^2-t^2 v_1 q+t v_1 q+v_1 q+v_2 q+t^2 v_1-t v_1\right)}{q t},\\
&\chi^{(1)}_{2,1}= \frac{(t-1) v_1 v_2 (-q v_1+q t v_1+v_1+t v_2)}{t},\\
&\chi^{(1)}_{2,2}=\frac{(t-1) v_1 v_2 \left(t v_2 q^2-v_2 q^2-t v_2
   q+v_2 q+t^2 v_1+t^2 v_2\right)}{\sqrt{\frac{q}{t}} t^2},\\
&\chi^{(2)}_{1,1}=\frac{(t-1) v_1 v_2 \left(q^2 t v_1+q^2 t v_2-q^2 v_1-q^2 v_2-q t v_1-q t v_2+q v_1+q v_2+t^2v_1\right)}{t^2},\\
&\chi^{(2)}_{1,2}=\frac{q (t-1) v_1 v_2 (q t v_1+q t v_2-q v_1-q v_2-t v_1+v_1+v_2)}{t^2 \sqrt{\frac{q}{t}}},\\
&\chi^{(2)}_{2,1}=\frac{(t-1) v_1^2 v_2^2 \left(q^2 t-q^2+q t^2-2 q t+q+t\right)}{t^2},\\
&\chi^{(2)}_{2,2}=\frac{(t-1) v_1^2 v_2^2 \left(q^2 t-q^2+q t^2-2 q t+q+t\right)}{t^2 \sqrt{\frac{q}{t}}}. 
\end{align}
By these, the matrix elements can be written as 
\begin{align}
\braket{X_{(2),\emptyset}|\cV(w)|X_{\emptyset,\emptyset}}&=
w\braket{X_{(1),\emptyset}|\cV(w)|X_{\emptyset,\emptyset}},\\ 
\braket{X_{(1,1),\emptyset}|\cV(w)|X_{\emptyset,\emptyset}}&=
w\chi^{(1)}_{1,1}\braket{X_{(1),\emptyset}|\cV(w)|X_{\emptyset,\emptyset}}+
w\chi^{(1)}_{1,2}\braket{X_{\emptyset,(1)}|\cV(w)|X_{\emptyset,\emptyset}}\nonumber \\
&\quad -w(t/q)(u_1+u_2)\braket{X_{(1),\emptyset}|\cV(w)|X_{\emptyset,\emptyset}},\\
\braket{X_{(1),(1)}|\cV(w)|X_{\emptyset,\emptyset}}&=
w\chi^{(1)}_{2,1}\braket{X_{(1),\emptyset}|\cV(w)|X_{\emptyset,\emptyset}}+
w\chi^{(1)}_{2,2}\braket{X_{\emptyset,(1)}|\cV(w)|X_{\emptyset,\emptyset}}\nonumber \\
&\quad -w(t/q)(u_1+u_2)\braket{X_{\emptyset,(1)}|\cV(w)|X_{\emptyset,\emptyset}}, \\
\braket{X_{\emptyset,(2)}|\cV(w)|X_{\emptyset,\emptyset}}&=
w\braket{X_{\emptyset,(1)}|\cV(w)|X_{\emptyset,\emptyset}},\\
\braket{X_{\emptyset,(1,1)}|\cV(w)|X_{\emptyset,\emptyset}}&=
w\chi^{(2)}_{2,1}\braket{X_{(1),\emptyset}|\cV(w)|X_{\emptyset,\emptyset}}+
w\chi^{(2)}_{2,2}\braket{X_{\emptyset,(1)}|\cV(w)|X_{\emptyset,\emptyset}}\nonumber \\
&\quad -w(t/q)^2(u_1u_2)\braket{X_{\emptyset,(1)}|\cV(w)|X_{\emptyset,\emptyset}}. 
\end{align}
Furthermore, the direct calculation gives
\begin{align}
X^{(i)}_0\left(
\begin{array}{c}
\ket{X_{(1),\emptyset}} \\
\ket{X_{\emptyset, (1)}}
\end{array}
\right)
=
\left(
\begin{array}{cc}
\kappa^{(i)}_{1,1} & \kappa^{(i)}_{1,2} \\
\kappa^{(i)}_{2,1} & \kappa^{(i)}_{2,2} \\
\end{array}
\right)
\left(
\begin{array}{c}
\ket{X_{(1),\emptyset}} \\
\ket{X_{\emptyset, (1)}}
\end{array}
\right), 
\end{align}
\begin{align}
X^{(i)}_1 
\left(
\begin{array}{c}
\ket{X_{(1),\emptyset}} \\
\ket{X_{\emptyset, (1)}}
\end{array}
\right)
=
\left(
\begin{array}{c}
\zeta^{(i)}_{1} \\
\zeta^{(i)}_{2} 
\end{array}
\right)
\ket{X_{\emptyset, \emptyset}}. 
\end{align}
Here, 
\begin{align}
\kappa^{(1)}_{1,1}=&
q^{-1} t^{-2}(t-1) (q^2 t u_1^2+q^2 t u_2^2+q^2 t u_1 u_2-q^2 u_1^2-q^2 u_2^2-q^2 u_1 u_2-q t^2 u_2^2\\
&-q t^2u_1 u_2+q t u_2^2+2 q t u_1 u_2+q u_1^2+q u_2^2+q u_1 u_2-t u_2^2-t u_1 u_2),\\
\kappa^{(1)}_{1,2}=&\frac{(t-1) u_2 (q t u_2-q u_2+t u_1+u_2)}{t^2 \sqrt{\frac{q}{t}}},\\ 
\kappa^{(1)}_{2,1}=&-t^{-3}(t-1) u_1 u_2 (-q^2 t u_1-q^2 t u_2+q^2 u_1+q^2 u_2+q t^2 u_2\\
&-q t u_2-q u_1-q u_2-t^2 u_2+tu_2),\\
\kappa^{(1)}_{2,2}=&\frac{(t-1) u_1 u_2 \sqrt{\frac{q}{t}} (q t u_2-q u_2+t u_1+u_2)}{t^2},\\
\kappa^{(2)}_{1,1}=&
q^{-1} t^{-3}(t-1) u_1 u_2 (q^3 t u_1+q^3 t u_2\nonumber \\
&-q^3 u_1-q^3 u_2-q^2 t u_1-q^2 t u_2+q^2 u_1+q^2 u_2+q t^2u_1+q t^2 u_2-t^3 u_2),\\
\kappa^{(2)}_{1,2}=&\frac{(t-1) u_1 u_2 \left(q^2 t u_1+q^2 t u_2-q^2 u_1-q^2 u_2-q t u_1-q t u_2+q u_1+q u_2+t^2
   u_2\right)}{t^3 \sqrt{\frac{q}{t}}},\\
\kappa^{(2)}_{2,1}=&\frac{q (t-1) u_1^2 u_2^2 \left(q^2 t-q^2+q t^2-2 q t+q+t\right)}{t^4},\\
\kappa^{(2)}_{2,2}=&\frac{(t-1) u_1^2 u_2^2 \sqrt{\frac{q}{t}} \left(q^2 t-q^2+q t^2-2 q t+q+t\right)}{t^3},
\end{align}
\begin{align}
\zeta^{(1)}_{1}=&\frac{(q-1) (t-1) \left(q u_1^2+q u_2 u_1+q u_2^2-t u_2 u_1\right)}{q t},\\
\zeta^{(1)}_{2}=&\frac{(q-1) q (t-1) u_1 u_2 (u_1+u_2)}{t^2},\\
\zeta^{(2)}_{1}=&\frac{(q-1) (t-1) u_1 u_2 (u_1+u_2)}{t},\\
\zeta^{(2)}_{2}=&\frac{(q-1) (t-1) u_1^2 u_2^2 (q+t)}{t^2}.
\end{align}
Thus, we obtain 
\begin{align}
\braket{X_{(1),\emptyset}|\cV(w)|X_{(1),\emptyset}}&=
w(v_1+v_2)
+\kappa^{(1)}_{1,1}\braket{X_{\emptyset,\emptyset}|\cV(w)|X_{(1),\emptyset}}
+\kappa^{(1)}_{1,2}\braket{X_{\emptyset,\emptyset}|\cV(w)|X_{\emptyset,(1)}}
-w(t/q)\zeta^{(1)}_{1},
\\
\braket{X_{(1),\emptyset}|\cV(w)|X_{\emptyset,(1)}}&=
w(v_1+v_2)
+\kappa^{(1)}_{2,1}\braket{X_{\emptyset,\emptyset}|\cV(w)|X_{(1),\emptyset}}
+\kappa^{(1)}_{2,2}\braket{X_{\emptyset,\emptyset}|\cV(w)|X_{\emptyset,(1)}}
-w(t/q)\zeta^{(1)}_{2},
\\
\braket{X_{\emptyset,(1)}|\cV(w)|X_{(1),\emptyset}}&=
wv_1v_2
+\kappa^{(2)}_{1,1}\braket{X_{\emptyset,\emptyset}|\cV(w)|X_{(1),\emptyset}}
+\kappa^{(2)}_{1,2}\braket{X_{\emptyset,\emptyset}|\cV(w)|X_{\emptyset,(1)}}
-w(t/q)^2\zeta^{(2)}_{1},
\\
\braket{X_{\emptyset,(1)}|\cV(w)|X_{\emptyset,(1)}}&=
wv_1v_2
+\kappa^{(2)}_{2,1}\braket{X_{\emptyset,\emptyset}|\cV(w)|X_{(1),\emptyset}}
+\kappa^{(2)}_{2,2}\braket{X_{\emptyset,\emptyset}|\cV(w)|X_{\emptyset,(1)}}
-w(t/q)^2\zeta^{(2)}_{2}.
\end{align}

\subsubsection*{In $N=3$ case}
Also in the representations of higher level, 
we can calculate the matrix elements similarly: 
\begin{align}
\braket{X_{(1),\emptyset,\emptyset}|\cV(w)|X_{\emptyset,\emptyset,\emptyset}}&=
w(v_1+v_2+v_3)-w(t/q)(u_1+u_2+u_3),\\ 
\braket{X_{\emptyset,(1),\emptyset}|\cV(w)|X_{\emptyset,\emptyset,\emptyset}}&=
w(v_1v_2+v_1v_3+v_2v_3)-w(t/q)^2(u_1u_2+u_1u_3+u_2u_3),\\
\braket{X_{\emptyset,\emptyset,(1)}|\cV(w)|X_{\emptyset,\emptyset,\emptyset}}&=
wv_1v_2v_3-w(t/q)^3u_1u_2u_3,\\
\braket{X_{\emptyset,\emptyset,\emptyset}|\cV(w)|X_{(1),\emptyset,\emptyset}}&=
(q/t)w^{-1}(u_1+u_2+u_3)-(q/t)w^{-1}(v_1+v_2+v_3),\\ 
\braket{X_{\emptyset,\emptyset,\emptyset}|\cV(w)|X_{\emptyset,(1),\emptyset}}&=
(q/t)^2w^{-1}(u_1u_2+u_1u_3+u_2u_3)-(q/t)^2w^{-1}(v_1v_2+v_1v_3+v_2v_3),\\
\braket{X_{\emptyset,\emptyset,\emptyset}|\cV(w)|X_{\emptyset,\emptyset,(1)}}&=
(q/t)^3w^{-1}u_1u_2u_3-(q/t)^3w^{-1}v_1v_2v_3.
\end{align}

\section*{List of Notations } \label{App: list of notation}

\subsection*{General Notations}
\begin{align*}
&N, \quad \text{fixed positive number} \\
&\gamma=(t/q)^{1/2},\\
&\vu=(u_1, \ldots, u_N),\quad \text{spectral parameters of the $N$-fold Fock tensor spaces} \\
&\quad t^{\pm \delta_i}\cdot \vu:=(u_1,\ldots,u_{i-1}, t^{\pm 1} u_i, u_{i+1}, \ldots, u_N),\\ 
&\quad t^{\alpha_i}\cdot \vu:=(u_1,\ldots, u_{i-1}, t u_i, t^{-1} u_{i+1}, u_{i+2}\ldots, u_N), \\
&\quad t^{\pm \vnn} \cdot \vu:=(t^{\pm n_1}u_1, \ldots, t^{\pm n_{N}} u_N), \\
&\boldsymbol{n} = (n_1, \dots, n_N), \quad 
\text{$n_i$ stands for the number of the $\Phi^{(i)}$'s in $V^{(\vnn)}(x)$}
\\
&\quad|\vnn|:=\sum_{i=1}^N n_i, \quad \text{the total number of the $\Phi^{(i)}$'s in $V^{(\vnn)}(x)$}
\\
&\quad\ik=\lbl{i}{k}{n}:=\sum_{s=1}^{i-1}n_s+k, 
\\
&\boldsymbol{m} = (m_1, \dots, m_N), \quad \text{ $m_i$ is the number of the $\Phi^{(i)}$'s in  $\ket{Q_{\boldsymbol{\lambda}}}$}
\\
&\quad|\boldsymbol{m}|:=\sum_{i=1}^N m_i, \quad \text{the total number of the $\Phi^{(i)}$'s in   $\ket{Q_{\boldsymbol{\lambda}}}$}
\\
&\quad\lbl{i}{k}{m}:=\sum_{s=1}^{i-1}m_s+k,
\\
&\boldsymbol{s}=(s_i), \quad \text{generic parameter in Macdonald functions }\\
&\quad \text{Especially, in some propositions in Sec. \ref{Sec_GenMac}, and in Sec. \ref{Sec_matrix element}, they are specialized as}\\
&\quad \quad s_{\lbl{i}{k}{n}}=q^{\lambda^{(i)}_k} t^{1-k}v_i \quad (1 \leq k \leq n_i, \, i=1,\ldots, N), 
\\
&\quad \quad s_{|\vnn|+[i,k]_{\vmm}}= t^{1- n_{i} -k}v_i \quad (1 \leq k \leq m_i, \, i=1,\ldots, N),
\\
&[\vl]^{\vnn}=([\vl]^{\vnn}_{i})_{1\leq i\leq |\vnn|}
:=(\lo_1,\ldots,\lo_{n_1},\lt_1,\ldots,\lt_{n_2},\ldots,\lN_1,\ldots,\lN_{n_N}), \\
&\aotimes{i=n}{m}A_i:=A_{n} \otimes \cdots \otimes A_{m},
\\
&n(\lambda)=\sum_{i\geq 1}(i-1)\lambda_i,\\
&g_{\lambda}=q^{n(\lambda')}t^{-n(\lambda)},\\
&\mathcal{G}(z) = \prod_{i,j=0}^\infty (1 - z q^i t^{-j}).
\shortintertext{The Nekrasov factor}
&N_{\lambda,\mu}(u)= \prod_{(i,j)\in \lambda} \left( 1- u q^{a_{\lambda}(i,j)}t^{\ell_{\mu}(i,j)+1} \right)  \prod_{(i,j)\in \mu} \left( 1- u q^{-a_{\mu}(i,j)-1} t^{-\ell_{\lambda}(i,j)} \right).\\
\shortintertext{Taki's Flaming factor}
&f_{\lambda}=(-1)^{|\lambda|}q^{n(\lambda')+|\lambda|/2}t^{-n(\lambda)-|\lambda|/2}.
\end{align*}

\subsection*{Algebras and Representations}

\begin{align*}
&\cU, \quad \text{the Ding-Iohara-Miki algebra (Def. \ref{def: DIM})}\\
&\quad x^{\pm}(z), \psi^{\pm}(z),\quad \text{the Drinfeld currents of the DIM algebra} \\
&\cF^{(1,M)}_u,\quad \text{the level-$(1,M)$ (horizontal) representations (Fact \ref{Fact_Hor_rep})}\\
&\quad \eta(z), \xi(z), \varphi^{\pm}(z),\quad \text{the currents in the horizontal Fock representations}\\ &\quad a_n, \quad \text{generators of the Heisenberg algebra (eq. (\ref{def_boson}))} \\
&\quad a^{(i)}_n=\overbrace{1\otimes \cdots \otimes 1}^{i-1}\otimes
a_n\otimes\overbrace{ 1 \otimes \dots \otimes 1}^{N-i}, \\
&\cF^{(0,1)}_v, \quad \text{the level-$(0,1)$ (vertical) representations (Fact \ref{Fact_vertical_rep})}\\
&\Phi(x), \Phi^*(x), \Phi_{\lambda}(x), \Phi^*_{\lambda}(x), \quad \text{the intertwining operator of the DIM algebra (Fact \ref{Fact_intertwiner})}\\
&\WA, \quad \text{the algebra generated by $X^{(i)}_n$'s (Def. \ref{Def_WA})}\\
&\quad X^{(k)}(z) 
= \sum_{1\leq j_1 <\cdots <j_k \leq N} : \Lambda^{(j_1)}(z) \cdots \Lambda^{(j_k)}((q/t)^{k-1}z): u_{j_1} \cdots u_{j_k},\quad \text{(Def. \ref{Def_X^k})} \\
&\quad \Lambda^{(i)} := \varphi^-(\gamma^{1/2}z)\otimes \cdots \otimes\varphi^-(\gamma^{i-3/2}z)\otimes\overbrace{\eta(\gamma^{i-1} z)}^{i\text{-th Fock space}}\otimes 1\otimes\cdots\otimes1, \\
&\quad Y^{(r)}(x):=\sum_{2\leq i_2<\cdots < i_r\leq N}
:\Lambda^{(i_2)}((q/t)tx) \cdots \Lambda^{(i_r)}((q/t)^{r-1}tx): 
u_{i_2}\cdots u_{i_r}, \\
&\quad \Lambda^{(i_1,\dots,i_k)}(z) : = : \Lambda^{(i_1)}(z) \cdots \Lambda^{(i_k)}(\gamma^{2(1-k)}z):.\\
\shortintertext{PBW-type bases (Def. \ref{def: PBW-type basis})}
 &\Ket{X_{\boldsymbol{\lambda}}} = 
  X^{(1)}_{-\lambda^{(1)}_1} X^{(1)}_{-\lambda^{(1)}_2} \cdots  X^{(2)}_{-\lambda^{(2)}_1} X^{(2)}_{-\lambda^{(2)}_2}\cdots X^{(N)}_{-\lambda^{(N)}_1} X^{(N)}_{-\lambda^{(N)}_2} \cdots \ketzero, \\
 & \Bra{X_{\vl}} =
\brazero \cdots X^{(N)}_{\lambda^{(N)}_2} X^{(N)}_{\lambda^{(N)}_1} \cdots X^{(2)}_{\lambda^{(2)}_2} X^{(2)}_{\lambda^{(2)}_1} \cdots 
X^{(1)}_{\lambda^{(1)}_2} X^{(1)}_{\lambda^{(1)}_1} .
\end{align*}

\subsection*{Vertex Operators associated with Generalized Macdonald functions}

\begin{align*}
\shortintertext{Screening currents (Def. \ref{def: screening})}
\qquad &S^{(i)}(z) := \overbrace{1\otimes\cdots\otimes 1}^{i-1}\otimes \phi^{\mathrm{sc}}(\gamma^{i-1} z)\otimes \overbrace{1\otimes \cdots \otimes 1}^{N-i-1}\,,\qquad i=1,\dots,N-1\,,\\
&\phi^{\mathrm{sc}}(z):= \exp\left(-\sum_{n>0} \frac{1}{n}\frac{1-t^{n}}{1-q^n}\gamma^{2n} a_{-n} z^n\right)\exp\left(\sum_{n>0} \frac{1}{n}\frac{1-t^{-n}}{1-q^{-n}}a_{n} z^{-n}\right) \\
&\qquad \qquad\qquad \otimes\exp\left(\sum_{n>0} \frac{1}{n}\frac{1-t^{n}}{1-q^n}\gamma^n a_{-n} z^n\right)\exp\left(-\sum_{n>0} \frac{1}{n}\frac{1-t^{-n}}{1-q^{-n}}\gamma^{-n} a_{n} z^{-n}\right)\,.
\shortintertext{Shifted screening currents }
&\widetilde{S}^{(k)} (z)=S^{(k)}(\gamma^{-2k}t^{-1} z). 
\shortintertext{Screened vertex operators (Def. \ref{Def_screendV})}
&\Phi^{(0)}(x)=:\exp\left( \sum_{n>0} \frac{1}{n}\frac{1-t^{n}}{1-q^n}a^{(1)}_{-n} u^n\right)\exp\left(\sum_{n>0} \frac{1}{n}\frac{1-\gamma^{2n}t^{n}}{1-q^{-n}}t^{-n}a^{(1)}_{n} x^{-n}\right)\\
&\qquad\qquad \qquad \qquad  \times \exp\left(\sum_{n>0} \frac{1}{n}\frac{1-\gamma^{2n}}{1-q^{-n}}
\sum_{j=2}^N\gamma^{(j-1)n} a^{(j)}_{n} x^{-n}\right):, \\
&\Phi^{(k)}(x):=
\prod_{i=1}^k \frac{(q;q)_{\infty} (q/t;q)_{\infty}}{(\frac{q u_i}{u_{k+1}};q)_{\infty}(\frac{q u_{k+1}}{tu_{i}};q)_{\infty}}
\oint_{C} \prod_{i=1}^k \frac{dy_i}{2 \pi \sqrt{-1}y_i}\Phi^{(0)}(x) S^{(1)}(y_1)\cdots S^{(k)}(y_k) g(x,y_1,\ldots,y_k), \\ 
&g(x,y_1,\ldots,y_k):=\frac{\theta_q(tu_1 y_1/u_{k+1}x)}{\theta_q(t y_1/x)}
\prod_{i=1}^{k-1}\frac{\theta_q(tu_{i+1} y_{i+1}/u_{k+1}y_i)}{\theta_q(t y_{i+1}/y_i)}\,.
\shortintertext{Cartan operator arising from the commutation relation between $\Phi^{(k)}(x)$ and $X^{(i)}(z)$ (Lem. \ref{lem: rel of X and Phi})}
&\Psi^+(z):= \exp\left(\sum \frac{1}{n}(1-\gamma^{2n}) \sum_{j=1}^N\gamma^{(j-1)n} a^{(j)}_n z^{-n}\right).
\shortintertext{Composition of screened vertex operators (Def. \ref{def_Vn})}
&V^{(\vnn)}(x_1,\ldots,x_{|\vnn|})=
\Phi^{(0)}(x_1)\cdots \Phi^{(0)}(x_{n_{1}})
\Phi^{(1)}(x_{n_1+1})\cdots \Phi^{(1)}(x_{n_{1}+n_{2}})\cdots\\ 
&\qquad \qquad \qquad \qquad \qquad \qquad\qquad \qquad \qquad \cdots
\Phi^{(N-1)}(x_{\lbl{N}{1}{n}})\cdots \Phi^{(N-1)}(x_{|\vnn|}).
\end{align*}

\subsection*{Symmetric functions and vectors in the $N$-fold tensor Fock spaces}

\begin{align*}
&\pp_{\lambda}, \quad \text{the power sum symmetric function}\\
&m_{\lambda}, \quad \text{the monomial symmetric functions}\\
&P_{\lambda},Q_{\lambda} \quad \text{the ordinary Macdonald functions} \\
&P_{\lambda}(a_{-n}^{(i)}),Q_{\lambda}(a_{-n}^{(i)}) 
\quad 
\begin{array}{l}
\text{Macdonald functions obtained by replacing }\\
\text{the power sum symmetric function $\pp_n$ by the boson $a^{(i)}_{-n}$}
\end{array}\\
&P_{\lambda}^{(r)}, Q_{\lambda}^{(r)},\quad \text{the ordinary Macdonald polynomials in $r$ variables}\\
\shortintertext{Generalized Macdonald functions (eigenfunctions of $\Xo_0$) (Thm. \ref{thm: GM})}
&\ket{P_{\vl}},\bra{P_{\vl}}, \\ 
&\ket{Q_{\vl}}:=\prod_{i=1}^N\frac{c_{\lambda^{(i)}}}{c'_{\lambda^{(i)}}} \ket{P_{\vl}}, \\
&\quad c_{\lambda}=\prod_{(i,j)\in \lambda}(1-q^{a_{\lambda}(i,j)}t^{\ell_{\lambda}(i,j)+1}), \\
&\quad c'_{\lambda}= \prod_{(i,j)\in \lambda}(1-q^{a_{\lambda}(i,j)+1}t^{\ell_{\lambda}(i,j)}).\\
\shortintertext{Integral form (Def. \ref{def: K})} 
&\ket{K_{\vl}}=\mathcal{C}_{\vl}^{(+)} \ket{P_{\vl}}, \\
&\bra{K_{\vl}}=\mathcal{C}_{\vl} ^{(-)}\bra{P_{\vl}},\\
&\quad\mathcal{C}_{\vl}^{(+)}  := \xi_{\vl}^{(+)} \times 
\prod_{1\leq i<j \leq N} N_{\lambda^{(i)}, \lambda^{(j)}}(qu_i/tu_j) 
\prod_{k=1}^N c_{\lambda^{(k)}}, \\
&\quad\mathcal{C}_{\vl} ^{(-)} :=  \xi_{\vl}^{(-)} \times
\prod_{1\leq i<j \leq N} N_{\lambda^{(j)}, \lambda^{(i)}}(qu_j/tu_i) 
\prod_{k=1}^N c_{\lambda^{(k)}},\\
& \quad\xi_{\vl}^{(+)}:=
\prod_{i=1}^N(-1)^{(N-i+1)|\lambda^{(i)}|}
u_{i}^{(-N+i)|\lambda^{(i)}|+\sum_{k=1}^{i}|\lambda^{(k)}|}\nonumber \\
&\quad\qquad \quad \times \prod_{i=1}^N
(q/t)^{\left(\frac{1-i}{2}\right)|\lambda^{(i)}|} 
q^{\left(i-N\right)\left(n(\lambda^{(i)'})+|\lambda^{(i)}|\right)}
t^{\left(N-i-1\right) \left( n(\lambda^{(i)})+|\lambda^{(i)}| \right)},\\
& \quad\xi_{\vl}^{(-)}:=\prod_{i=1}^N(-1)^{i|\lambda^{(i)}|}
u_{i}^{(-i+1)|\lambda^{(i)}|+\sum_{k=i}^{N}|\lambda^{(k)}|}\nonumber \\
&\quad\qquad \quad \times \prod_{i=1}^N
(q/t)^{\left(\frac{i-1}{2}\right)|\lambda^{(i)}|} t^{|\lambda^{(i)}|}
q^{\left(1-i\right)\left(n(\lambda^{(i)'})+|\lambda^{(i)}|\right)}
t^{\left(i-2\right) \left( n(\lambda^{(i)})+|\lambda^{(i)}| \right)}.
\shortintertext{Factor arising from the application of Rumanujan's ${_1}\psi_1$ summation formula (Prop. \ref{Prop_R})}
&\mathcal{R}^{\vnn}_{\vl}(\vu) =
\gamma^{\sum_{i=1}^N (i-1)|\lambda^{(i)}|}
\prod_{k=2}^N\prod_{i=1}^{n_k}\prod_{l=1}^{k-1} 
\frac{(t^{-n_l+i}u_l/u_k;q )_{-\lambda^{(k)}_i}}{(qt^{-n_l+i-1}u_l/u_k;q )_{-\lambda^{(k)}_i}}\,.
\end{align*}

\subsection*{Factors in the Macdonald  functions and hypergeometric series}
\begin{align*}
&d_n((\theta_i)_{1\leq i \leq n-1};(s_i)_{1\leq i \leq n}|q,t)	\nonumber  \\
& \qquad =\prod_{i=1}^{n-1}
(q/t)^{\theta_{i}} { (t;q)_{\theta_{i}} \over (q;q)_{ \theta_{i}} }
{(t s_n/s_i;q)_{\theta_{i}} \over (q s_n/s_i;q)_{\theta_{i}} }\prod_{1\leq i<j\leq n-1}
{(t s_j/s_i ;q)_{\theta_{i}} \over (q s_j/s_i;q)_{\theta_{i}} }
{(q^{-\theta_{j}}q s_j/t s_i;q)_{\theta_{i}} \over 
	(q^{-\theta_{j}}s_j/s_i;q)_{\theta_{i}} },\\
&c_1(-;s_1,q,t)=1\,,\\
&c_n((\theta_{i,j})_{1\leq i<j\leq n};(s_i)_{1\leq i \leq n}|q,t)\\
&\quad = c_{n-1}((\theta_{i,j})_{1\leq i<j\leq n-1}; (q^{-\theta_{i,n}}s_i)_{1\leq i \leq n}|q,t) \times d_n((\theta_{i,n})_{1\leq i \leq n-1};(s_i)_{1\leq i \leq n}|q,t).
\shortintertext{The Macdonald functions (Def. \ref{def: p_n f_n})}
&p_n(x;s|q,t) =\sum_{\theta \in M_n} c_n(\theta;s|q,t) \prod_{1\leq i<j\leq n}(x_j/x_i)^{\theta_{i,j}}\,,\\ 
&f_n(x;s|q,t) =\prod_{1\leq k<\ell\leq n}(1-x_\ell/x_k)\,\cdot\,p_n(x;s|q^{-1},t^{-1}).
\shortintertext{Kajihara and Noumi's multiple basic hypergeometric series (Def. \ref{def_phi_m_n})}
&\Mphi{m,n}{\mu}
{a_1,\ldots,a_m\\ x_1,\ldots,x_m}
{b_1,\ldots,b_n\\ c_1,\ldots,c_n} =
\prod_{i<j}{}
\frac{q^{\mu_i}x_i-q^{\mu_j}x_j}{x_i-x_j}\ 
\prod_{i,j}{}
\frac{(a_jx_i/x_j;q)_{\mu_i}}
{(qx_i/x_j;q)_{\mu_i}}\ 
\prod_{i,k}{}
\frac{(b_kx_i;q)_{\mu_i}}
{(c_kx_i;q)_{\mu_i}}, \\
&\Mphiu{m,n}
{a_1,\ldots,a_m\\ x_1,\ldots,x_m}
{b_1,\ldots,b_n\\ c_1,\ldots,c_n}{u} =
\sum_{\mu \in \mathbb{Z}_{\geq 0}^m} 
u^{\sum_{i=1}^m \mu_i}
\Mphi{m,n}{\mu}
{a_1,\ldots,a_m\\ x_1,\ldots,x_m}
{b_1y_1,\ldots,b_ny_n\\ c_1y_1,\ldots,c_ny_n}.
\shortintertext{Factors arising in the transformation formula from $p_{n+m}$ to another series containg $p_m$ as inner summation (Def. \ref{Def_N_mu})}
&\mathsf{N}^{n,m}_{\mu}(s_1,\ldots,s_{n+m}) :=
\prod_{k=1}^m\left( \prod_{i=1}^{n+k} \frac{(qs_{n+k}/ts_i;q)_{\mu_k}}{(qs_{n+k}/s_i;q)_{\mu_k}} \right)
\prod_{1\leq i <j \leq m} \frac{(t\, q^{-\mu_i}s_{n+j}/s_{n+i};q)_{\mu_j}}{(q^{-\mu_i}s_{n+j}/s_{n+i};q)_{\mu_j}}, \\
&\widetilde{\mathsf{N}}^{n,m}_{\mu}(h;s_1,\ldots,s_{n+m}) := 
\prod_{k=1}^m\left( \prod_{i=1}^{n+k} \frac{(qs_{n+k}/ts_i;q)_{\mu_k}}{(h qs_{n+k}/s_i;q)_{\mu_k}} \right)
\prod_{1\leq i <j \leq m} \frac{(t\, q^{-\mu_i}s_{n+j}/s_{n+i};q)_{\mu_j}}{(q^{-\mu_i}s_{n+j}/s_{n+i};q)_{\mu_j}},
\end{align*}
where $n$, $m$ are nonnegative integers and $\mu=(\mu_i)_{1\leq i \leq m} \in \mathbb{Z}^m$.

\subsection*{Mukad\'e Operator and its Relatives}
\begin{align*}
\shortintertext{The defining relation of the vertex operator $\cV(x) : \cF_{\boldsymbol{u}} \to\cF_{\boldsymbol{v}}$ (Def. \ref{def_V_gen}) }
&\left(1-\frac{x}{z}\right)X^{(i)}(z) \cV(x) 
=\left(1- (t/q)^i\frac{x}{z}\right)\cV(x) X^{(i)}(z) 
\qquad i \in \{1,2,\dots,N\}.
\shortintertext{Realization of $\cV(x)$ for the special case $v_i=t^{-n_i}u_i$ (Def. \ref{def_Vspecial})}
&\widetilde V^{(\vnn)}(x)=
\lim_{x_i \rightarrow t^{|\vnn|-i} x}
\prod_{1\leq i< j\leq |\vnn|} \frac{(tx_j/x_i;q)_{\infty}}{(qx_j/tx_i;q)_{\infty}} 
V^{(\vnn)} (x_1,\ldots ,x_{|\vnn|})A^{-1}_{(|\vnn|)}(x),\\
&A_{(r)}(x) = \exp\left( \sum_{n>0} \frac{(1-(q/t)^r)(1-t^{(1-r)n})t^{2r}}{n(1-q^n)(1-t^{-n})}
\sum_{i=1}^{N} \gamma^{(i-1)n}a^{(i)}_{n}x^{-n} \right). \\
\shortintertext{Mukad\'{e} operators connected toward vertical and horizontal directions (Def. \ref{def_TV_TH})}
&\mathcal{T}^{V}(\boldsymbol{u}, \boldsymbol{v};w),\quad  
\mathcal{T}^{H}(\boldsymbol{u}, \boldsymbol{v};w) ,\quad
 \cTH_{\boldsymbol{\lambda},\boldsymbol{\mu}}(\boldsymbol{u}, \boldsymbol{v};w),\quad \cTV_{\boldsymbol{\lambda},\boldsymbol{\mu}}(\boldsymbol{u}, \boldsymbol{v};w). 
\shortintertext{Mukad\'{e} operator specialized so that Young diagrams are restricted to only one row (eq. (\ref{eq: Jackson til Ti}))}
& \widetilde{\mathcal{T}}_i(x)=\widetilde{\mathcal{T}}_i(\vu; x).
\end{align*}


\end{document}